\documentclass[11pt]{amsart}
\usepackage[utf8]{inputenc}
\usepackage[T1]{fontenc}
\usepackage{times}

\usepackage{amssymb}
\usepackage{a4wide}
\usepackage{url}
\usepackage{upref}
\usepackage{xspace}
\usepackage{graphicx}
\usepackage{mathtools}
\mathtoolsset{showonlyrefs}
\RequirePackage[pdftex,
     pdfview=XYZ,
     pdfstartview=FitV,
     pdffitwindow,
  hyperfootnotes=false,
  raiselinks=true,
  hypertexnames=false,
  colorlinks=true,
  bookmarks=true,
  bookmarksopen=false,
  anchorcolor=black,
  citecolor=black,
  linkcolor=black,
  filecolor=black,
  urlcolor=black]{hyperref}

\usepackage{mathrsfs}
\let\mathcal\mathscr
\usepackage[all,cmtip]{xy}
\newdir^{ (}{{}*!/-5pt/\dir^{(}}
\newdir_{ (}{{}*!/-5pt/\dir_{(}}

\def\mainmatter{\def\baselinestretch{1.1}\normalfont}
\def\backmatter{\def\baselinestretch{1}\normalfont}

\AtBeginDocument{%
}

\newcounter{step}
\newcommand{\step}{\par\refstepcounter{step}\textup{(\thestep)}}

\makeatletter
\@namedef{subjclassname@2010}{\textup{2020} Mathematics Subject Classification}

\renewcommand{\tocsubsection}[3]{%
  \hspace*{1.5pc}\indentlabel{\@ifnotempty{#2}{\ignorespaces#1 #2.\quad}}#3}

\def\Subsection{\@startsection{subsection}{2}%
  \z@{.7\linespacing\@plus.7\linespacing}{.3\linespacing}
  {\normalfont\bfseries}}
\makeatother

\newenvironment{enumeratea}
{\bgroup\begin{enumerate}}
{\end{enumerate}\egroup}

\numberwithin{equation}{section}

\theoremstyle{plain}
\newtheorem{thm}[equation]{Theorem}
\newtheorem{lemma}[equation]{Lemma}
\newtheorem{prop}[equation]{Proposition}
\newtheorem{cor}[equation]{Corollary}

\theoremstyle{definition}
\newtheorem{remark}[equation]{Remark}
\newtheorem{defi}[equation]{Definition}

\newtheorem{example}[equation]{Example}
\newtheorem{assumption}[equation]{Assumption}
\newtheorem{caveat}[equation]{Caveat}
\newtheorem{convention}[equation]{Convention}

\renewcommand\to{\mathchoice{\longrightarrow}{\rightarrow}{\rightarrow}{\rightarrow}}
\newcommand\mto{\mathchoice{\longmapsto}{\mapsto}{\mapsto}{\mapsto}}
\newcommand\hto{\mathrel{\lhook\joinrel\to}}
\newcommand\To[1]{\mathchoice{\xrightarrow{
\kern4pt#1\kern3pt}}{\stackrel{#1}{\longrightarrow}}{}{}}
\def\Hto#1{\mathrel{\lhook\joinrel\To{#1}}}
\newcommand{\isom}{\stackrel{\sim}{\longrightarrow}}
\let\ra\rightarrow

\let\oldbigoplus\bigoplus
\renewcommand{\bigoplus}{\mathop{\textstyle\oldbigoplus}\displaylimits}

\newcommand\cA{\mathcal{A}}
\newcommand\cC{\mathcal{C}}
\newcommand\cD{\mathcal{D}}
\newcommand\cE{\mathcal{E}}
\newcommand\cF{\mathcal{F}}
\newcommand\cG{\mathcal{G}}
\newcommand\cH{\mathcal{H}}
\newcommand\cO{\mathcal{O}}
\newcommand\cT{\mathcal{T}}
\newcommand\cU{\mathcal{U}}
\newcommand\cV{\mathcal{V}}

\newcommand{\cHom}{\cH\!om}
\newcommand{\RcHom}{\bR\!\cHom}

\DeclareMathAlphabet{\eucal}{U}{eus}{m}{n}
\newcommand{\ccC}{\eucal{C}}
\newcommand{\ccP}{\eucal{P}}

\newcommand{\Db}{\mathfrak{Db}}
\newcommand{\Cb}{\mathfrak{C}}

\newcommand\bD{\boldsymbol{D}}
\newcommand\bH{\boldsymbol{H}}
\newcommand\bR{\boldsymbol{R}}
\newcommand{\epsilonb}{{\boldsymbol\varepsilon}}

\newcommand{\PP}{\mathbb{P}}

\newcommand{\QQ}{\mathbf{Q}}
\newcommand{\ZZ}{\mathbf{Z}}
\newcommand{\RR}{\mathbf{R}}
\newcommand{\CC}{\mathbf{C}}
\newcommand{\NN}{\mathbf{N}}
\newcommand{\DD}{\mathbf{D}}

\newcommand{\alg}{\mathrm{alg}}
\newcommand{\an}{\mathrm{an}}
\newcommand{\dR}{{\mathrm{dR}}}
\newcommand{\Betti}{\mathrm{B}}
\newcommand{\BM}{{\scriptscriptstyle\mathrm{BM}}}
\newcommand{\coH}{\mathrm{H}}
\let\per\comp
\newcommand{\irr}{{\mathrm{irr}}}
\newcommand{\nb}{\mathrm{nb}}
\newcommand{\reg}{\mathrm{reg}}
\newcommand{\rb}{{\mathrm{b}}}
\newcommand{\rc}{{\mathrm{c}}}
\newcommand{\rd}{{\mathrm{d}}}\let\de\rd
\newcommand{\rmid}{{\mathrm{mid}}}
\newcommand{\rmod}{\mathrm{mod}}
\newcommand{\rrd}{\mathrm{rd}}
\newcommand{\sm}{\mathrm{sm}}

\newcommand{\catD}{\mathsf{D}}
\newcommand{\sfi}{\mathsf{i}}
\newcommand{\pii}{\pi\sfi}
\newcommand{\catPer}{\mathsf{Per}}
\newcommand\Bpairing{\mathsf B}
\newcommand\Ppairing{\mathsf P}
\newcommand\Qpairing{\mathsf Q}
\newcommand\Rpairing{\mathsf R}
\newcommand\Spairing{\mathsf S}
\let\Ipairing\Bpairing
\let\DRpairing\Qpairing

\let\kk\bmk
\newcommand{\bK}{{\boldsymbol{K}}}

\DeclareMathOperator{\can}{can}
\DeclareMathOperator{\DR}{DR}
\DeclareMathOperator{\gr}{gr}
\DeclareMathOperator{\Hom}{Hom}
\DeclareMathOperator{\id}{Id}
\DeclareMathOperator{\image}{im}
\DeclareMathOperator{\res}{res}
\DeclareMathOperator{\Res}{Res}
\DeclareMathOperator{\tr}{tr}

\def\loccit{loc.\kern3pt cit.{}\xspace}
\def\cf{see\kern.3em}
\def\eg{e.g.\kern.3em}
\def\ie{i.e.,\ }
\def\resp{\text{resp.}\kern.3em}

\let\moins\smallsetminus
\let\leq\leqslant
\let\geq\geqslant
\let\wh\widehat
\let\wt\widetilde

\let\epsilon\varepsilon

\let\ov\overline
\let\emptyset\varnothing

\let\oldvee\vee
\renewcommand{\vee}{{\scriptscriptstyle\oldvee}}
\let\Div D

\newcommand{\bbullet}{{\scriptscriptstyle\bullet}}
\newcommand{\cbbullet}{{\raisebox{1pt}{$\bbullet$}}}

\newcommand{\wti}{{\wt\imath}}
\newcommand{\wtj}{\wt\jmath}
\newcommand\bM{{\partial M}}
\newcommand{\ieme}{\textsuperscript{e}}

\newcommand{\lpr}{(\!(}
\newcommand{\rpr}{)\!)}
\newcommand{\lcr}{[\![}
\newcommand{\rcr}{]\!]}

\newcommand{\triup}{{\scriptscriptstyle\triangle}}
\newcommand{\smallsquare}{{\scriptscriptstyle\square}}
\newcommand{\sep}{\mkern1mu|\mkern1mu}

\newcommand{\gen}[1]{\left\langle{#1}\right\rangle}

\providecommand{\eprint}[1]{\href{http://arxiv.org/abs/#1}{\texttt{arXiv\string:\allowbreak#1}}}

\begin{document}
\title{Quadratic relations between periods of connections}

\author[J. Fresán]{Javier Fresán}
\address[J. Fresán]{CMLS, École polytechnique, Institut Polytechnique de Paris, 91128 Palaiseau cedex, France}
\email{javier.fresan@polytechnique.edu}
\urladdr{http://javier.fresan.perso.math.cnrs.fr}

\author[C.~Sabbah]{Claude Sabbah}
\address[C.~Sabbah]{CMLS, CNRS, École polytechnique, Institut Polytechnique de Paris, 91128 Palaiseau cedex, France}
\email{Claude.Sabbah@polytechnique.edu}
\urladdr{http://www.math.polytechnique.fr/perso/sabbah}

\author[J.-D. Yu]{Jeng-Daw Yu}
\address[J.-D. Yu]{Department of Mathematics, National Taiwan University,
	Taipei 10617, Taiwan}
\email{jdyu@ntu.edu.tw}
\urladdr{http://homepage.ntu.edu.tw/~jdyu/}

\thanks{This research was partly supported by the PICS project TWN 8094 from CNRS. The research of J.F. was also partly supported by the grant ANR-18-CE40-0017 of Agence Nationale de la Recherche.}

\begin{abstract}
We prove the existence of quadratic relations between periods of meromorphic flat bundles on complex manifolds with poles along a divisor with normal crossings under the assumption of ``goodness''. In dimension one, for which goodness is always satisfied, we provide methods to compute the various pairings involved. In an appendix, we recall some classical results needed for the proofs.
\end{abstract}



\maketitle
\thispagestyle{empty}
\vspace*{-2\baselineskip}
\enlargethispage{2\baselineskip}%
\tableofcontents
\mainmatter

\section{Introduction}
Let $U$ be an affine open set of the complex projective line.
For an algebraic vector bundle~$V$ on $U$ endowed with an algebraic connection $\nabla$ and a non-degenerate pairing compatible with $\nabla$, there exists a natural pairing $\Ppairing$, called the \emph{period pairing}, between the first de~Rham cohomology space $\coH^1_\dR(U,V)$ and the first homology space $\coH_1(U^\an,V^\nabla)$---also known as the twisted homology space---on the associated analytic space $U^\an$
with coefficients in the local system~$V^\nabla$ of horizontal sections of $(V,\nabla)$. This pairing is obtained by integrating $1$-forms with values in~$V$ against $1$-cycles twisted by $V^\nabla$, by means of the given pairing $V\otimes V\to\cO(U)$. There also exists a similar pairing $\Ppairing^\rc$ between the de~Rham cohomology with compact support $\coH^1_{\dR,\rc}(U,V)$ and the Borel-Moore twisted homology $\coH_1^\BM(U^\an,V^\nabla)$. The latter pairing is less used than the former because the expression of de~Rham classes with compact support is not as simple as that of de~Rham classes with no support condition. Moreover,
since Borel-Moore cycles may have boundary in $\PP^1\moins U$, this pairing needs a regularization procedure to be expressed.

In the case when $\nabla$ has regular singularities at infinity, both pairings $\Ppairing$ and $\Ppairing^\rc$ are known to be non-degenerate, as well as the de~Rham duality pairing
\[
\Spairing:\coH^1_{\dR,\rc}(U,V)\otimes \coH^1_\dR(U,V)\to\CC
\]
and the intersection pairing
\[
\Bpairing:\coH_1^\BM(U^\an,V^\nabla)\otimes \coH_1(U^\an,V^\nabla)\to\CC.
\]
Poincaré duality, regarded as an isomorphism between cohomology and homology, implies a relation between the matrices of these pairings:
\[
(2\pii)\Bpairing=\Ppairing\cdot\Spairing^{-1}\cdot{}^t\Ppairing^\rc.
\]
In many interesting examples, both cohomologies $\coH^1_\dR(U,V)$ and $\coH^1_{\dR,\rc}(U,V)$
\resp homologies $\coH_1(U^\an,V^\nabla)$ and $\coH_1^\BM(U^\an,V^\nabla)$ coincide,
and there is no distinction between $\Ppairing$ and $\Ppairing^\rc$. This relation is then regarded as a family of polynomial relations of degree two between the entries of the period matrix $\Ppairing$. If $V$ and~$V^\nabla$
are defined respectively over subfields~$\bK$ and~$\kk$ of $\CC$, the pairings $\Bpairing$ and $\Spairing$ take values in these fields and such relations take the form of an equality between a quadratic expression of the periods with coefficients in~$\bK$ and an element in $(2\pii)\kk$. In a series of papers, Matsumoto et~al.\ \hbox{\cite{C-M95, K-M95, K-M97, K-M-M99}} have developed this technique and have used it to produce quadratic relations between classical special functions, \cf also \cite{Goto15a,Goto15b}. The aforementioned articles show many interesting examples where these quadratic relations can be effectively computed. This procedure can be extended in higher dimension, giving rise to examples related to hyperplane arrangements. We refer to the book \cite{A-K11} where many examples are treated with details.

If one relaxes the condition of regular singularities for the connection $\nabla$, the period pairings $\Ppairing,\Ppairing^\rc$ may be degenerate because they pair spaces of possible different dimensions. In \cite{M-M-T00}, the authors have extended the quadratic relations to connections, with irregular singularities at some points of $\PP^1\moins U$, obtained by adding a rational $1$-form to the trivial connection $\rd$ on the trivial bundle $V=\cO_U$. In order to retain a perfect period pairing,
the twisted cycles have to be replaced with twisted cycles with rapid decay, which happen to coincide with those introduced by Bloch-Esnault \cite{B-E04b}. These are supported by chains with boundary in $\PP^1\moins U$ that reach these boundary points in suitable directions so that the corresponding period integrals converge. It is then convenient to work with the topological space obtained by completing $U^\an$ with the space of directions at each point of $\PP^1\moins U$, that is, the real oriented blow-up space of $\PP^1$ at each point of $\PP^1\moins U$.
Correspondingly, the de~Rham cohomologies with and without support conditions are computed by means of the $C^\infty$ de~Rham cohomology with rapid decay and moderate growth respectively on this blow-up space, in a way very similar to that developed in \cite{M-M-T00} (\cf also~\cite{Matsumoto98}).\enlargethispage{\baselineskip}%

The main purpose of this article is to establish the existence of quadratic relations that take into account rapid decay and moderate growth properties on the real oriented blow-up space, and to give some methods to compute them for any differential equation on a Riemann surface, without any assumption on the kind of singularities it may have. We then obtain quadratic relations with enough generality in order to analyze, say, quadratic relations between Bessel moments, which are periods of symmetric powers of the Bessel differential equation \cite{F-S-Y20b}. Treating this example was the main motivation for developing the theory in this direction, and the reader will find in \loccit\ a much detailed example of the various notions explained in the present article.

As for the case with regular singularities, these quadratic relations also exist in higher dimensions and may be interesting when considering period pairings for irregular GKZ systems or hyperplane arrangements twisted by exponentials of rational functions for example. The general rule is to replace forms with ``no support condition'' (\resp ``Borel-Moore'' cycles) by forms (\resp cycles) with ``moderate growth'', and similarly ``compact support'' by ``rapid decay''.

\subsubsection*{Contents and organization of the article}
In Section \ref{sec:pairingsflatbundles}, we consider the case of a differentiable manifold with corners and a differentiable vector bundle on it endowed with an integrable connection having poles along the boundary. No analyticity property is needed, but an assumption on the behaviour near the boundary of the de~Rham complexes of differential forms seems to be essential (Assump\-tions \ref{hyp:hyp}). Corollaries~\ref{cor:quadraticreldR} and \ref{cor:quadraticrelB} provide quadratic relations in a very general form. We emphasize \emph{middle quadratic relations} since they tend to be the most non-trivial part of quadratic relations. We revisit the period pairings introduced by Bloch and Esnault \cite{B-E04b} in the form considered by Hien \cite{Hien07,Hien09, Hien10}. Applications to period pairings similar to those of \loccit\ in any dimension are given in \hbox{Section}~\ref{subsec:comparisonperiods}. The~main idea, taken from \cite{Hien07,Hien09, Hien10}, is to define local period \hbox{pairings} in the sheaf-theoretical sense, prove that they are perfect, and get for free the perfectness of the global period pairings by applying the Poincaré\nobreakdash-Verdier duality theorem. In~that case, the manifold with corners is nothing but the real oriented blow-up of a complex manifold along a divisor with normal crossings. Provided Assumptions~\ref{hyp:hyp} are satisfied, these ideas apply to the general framework for proving quadratic relations between periods.

Section \ref{sec:dimone} focuses on the case of meromorphic connections on Riemann surfaces. A~formula ``à~la~\v Cech'' for computing the de~Rham pairing is provided by Theorem \ref{th:duality}. In~the case of rank-one bundles with meromorphic connection, this result already appears in Deligne's notes~\cite{Deligne84b}. We also give a formula for computing period matrices (Proposition \ref{prop:calculP}) that goes back, for connections with regular singularities, at least to \cite[Rem.\,2.16]{D-M86}. Quadratic relations in the sense of Matsumoto et al.\ are obtained in \eqref{eq:quadraticone}. The reader will find these notions illustrated in the example treated in \cite{F-S-Y20b}.

The appendix recalls with detailed proofs\footnote{The proofs are omitted in the published version.} classical results which are taken from the Séminaire Cartan \cite{Cartan48,Cartan50}, the books of de~Rham \cite{deRham73,deRham84}, of Malgrange \cite{Malgrange66}, and of Kashiwara and Schapira \cite{K-S90}. Let us emphasize that a similar approach has already been used by Kita and Yoshida~\cite{K-Y94} in order to define and compute the Betti intersection pairing for ordinary twisted homology (\cf also \cite[Chap.\,2]{A-K11}). We~improve their results by extending them to connections of any rank and possibly with irregular singularities.

Needless to say, the results of this article are not essentially new, but we have tried to give them with enough generality, rigor, and details so that they can be used in various situations without reproving the intermediate steps.

\section{Pairings for bundles with flat connection and quadratic relations}\label{sec:pairingsflatbundles}

\subsection{Setting and notation}\label{subsec:settings}
We consider a $C^\infty$ manifold $M$ of real dimension $m$ with corners, that we assume to be connected and orientable. The boundary of~$M$ is denoted by $\bM$, and the inclusion of the interior $M^\circ=M\moins\bM$ by $j \colon M^\circ\hto M$. In the neighbourhood of each point of~$\bM$, the pair $(M,\bM)$ is diffeomorphic to $(\RR_+{}^p,\partial(\RR_+{}^p))\times\RR^{m-p}$ for some integer $p\in[1,m]$. The sheaf $\cC^\infty_M$ of $C^\infty$ functions (\resp $\Db_M$ of distributions) on~$M$ is locally the sheaf-theoretic restriction to $\RR_+{}^p\times\RR^{m-p}$ of the sheaf of $C^\infty$ functions (\resp distributions) on $\RR^m$. We~will also consider the sheaf $\cC_M^{\infty}(*)$ of~$C^\infty$ functions on $M^\circ$ having at most poles along~$\bM$. The de~Rham complex~$(\cE^\cbbullet_M,\rd)$ of $C^\infty$ differential forms on~$M$ is a resolution of the constant sheaf $\CC_M$. The complex of currents $(\Cb_{M,\bbullet},\partial)$, deno\-ted homologically, can be regarded cohomologically as $(\Db_M^{m-\cbbullet},\rd)$. Recall (\cf Appendix~\ref{app:currents}) that a current of homological index~$q$ is a linear form on test differential forms of degree $q$. Then the inclusion of complexes $(\cE^{m-\cbbullet}_M,\rd)\hto(\Db_M^{m-\cbbullet},\rd)$ is a quasi-isomorphism. In other words, $(\Db_M^{m-\cbbullet},\rd)$ is a resolution of~$\CC_M[m]$. We refer the reader to Appendix~\ref{app:currents} for the notation and results concerning currents with moderate growth. The \hbox{results} in this section are a mere adaptation of those contained in de~Rham's book \cite{deRham73,deRham84} (\cf Chapter~IV in~\cite{deRham84}). In order to simplify the discussion, and since we are only interested in this case, we assume throughout this section that $M$ is moreover \emph{compact}.

\subsection{De Rham formalism for quadratic relations}\label{subsec:dRquadratic}

\subsubsection*{Setting and the main assumption}
We keep the assumptions and notations as in Section \ref{subsec:settings}. Let $(\cV,\nabla)$ be a locally free $\cC_M^{\infty}(*)$-module of finite rank on~$M$ endowed with a flat connection~$\nabla$ and a \emph{flat non-degenerate} pairing
\[
\gen{\cbbullet,\cbbullet}:\cV\otimes\cV\to\cC_M^{\infty}(*),
\]
\ie compatible with the connections. It induces a self-duality
$\lambda: \cV \isom \cV^\vee$
which factorizes $\gen{\cbbullet,\cbbullet}$ as the composition
\[ \cV\otimes\cV \To{1\otimes\lambda} \cV\otimes\cV^\vee
\To{(\cbbullet\sep\cbbullet)} \cC_M^{\infty}(*), \]
where $(\cbbullet\sep\cbbullet)$ denotes the natural duality pairing.
We can define the de Rham complexes of $(\cV,\nabla)$ with respectively $\cC_M^{\infty,\rrd}$ and $\cC_M^{\infty,\rmod}$ coefficients (\cf Definition \ref{def:rdmod}). Since $\nabla$ has only poles as possible singularities along~$\bM$, the complexes
\[
\DR^{\rrd}(\cV,\nabla)=(\cE_M^{\rrd,\cbbullet}\otimes\cV,\nabla)\quad\text{and}\quad\DR^{\rmod}(\cV,\nabla)=(\cE_M^{\rmod,\cbbullet}\otimes\cV,\nabla)
\]
are well-defined.
The de~Rham cohomologies with rapid decay $\coH_{\dR,\rrd}^r(M,\cV)$ and the de~Rham cohomologies with moderate growth $\coH_{\dR,\rmod}^r(M,\cV)$ are defined to be the hypercohomologies of these respective complexes, which are equal to the cohomologies of the global sections. We~make the following assumptions.

\pagebreak[2]
\begin{assumption}\mbox{}\label{hyp:hyp}
\begin{enumerate}
\item\label{hyp:1}
The de~Rham complexes $\DR^{\rrd}(\cV,\nabla)$ and $\DR^{\rmod}(\cV,\nabla)$ have cohomology sheaves concentrated in degree $0$, denoted respectively by $\cV^\rrd$ and~$\cV^\rmod$.

\item\label{hyp:2}
The natural pairings
defined by means of $\gen{\cbbullet,\cbbullet}:\cV\otimes\cV\to\cC^\infty_M(*)$
and of the wedge product:
\begin{align*}
\DRpairing_{\rrd,\rmod}:\DR^{\rrd}(\cV,\nabla)\otimes\DR^{\rmod}(\cV,\nabla)\to(\cE^{\rrd,\cbbullet}_M,\rd),\\
\DRpairing_{\rmod,\rrd}:\DR^{\rmod}(\cV,\nabla)\otimes\DR^{\rrd}(\cV,\nabla)\to(\cE^{\rrd,\cbbullet}_M,\rd)
\end{align*}
are perfect (\cf Example \ref{ex:dualizing}), \ie the pairings induced on the $\cH^0$ are perfect:
\begin{align*}
\cV^\rrd\otimes\cV^\rmod&\to j_!\CC_{M^\circ},\\
\cV^\rmod\otimes\cV^\rrd&\to j_!\CC_{M^\circ}.
\end{align*}
\end{enumerate}
\end{assumption}

One can drop the assumption on the existence of a non-degenerate pairing $\gen{\cbbullet,\cbbullet}$ and use $(\cbbullet\sep\cbbullet)$ instead. One can easily adapt the results of this section by distinguishing between $\cV$ and $\cV^\vee$.

\subsubsection*{The de Rham pairings}
Verdier duality implies (Proposition~\ref{prop:accperfect}) that the natural pairings
\begin{equation}\label{eq:rdmodpairings}
\begin{aligned}
\coH^{m-r}(M,\cV^\rrd)\otimes \coH^r(M,\cV^\rmod)&\to \coH^m_\rc(M^\circ,\CC)\simeq\CC,\\
\coH^{m-r}(M,\cV^\rmod)\otimes \coH^r(M,\cV^\rrd)&\to \coH^m_\rc(M^\circ,\CC)\simeq\CC
\end{aligned}
\end{equation}
are non-degenerate. We will interpret these pairings in terms of the de~Rham coho\-mologies $\coH_{\dR,\rrd}^{m-r}(M,\cV)$ and $\coH_{\dR,\rmod}^r(M,\cV)$ to which these cohomologies can respectively be identified by means of Assumption \ref{hyp:hyp}\eqref{hyp:1}. We~note that the diagram
\begin{equation}\label{eq:Srdmod}
\begin{array}{c}
\xymatrix@C=1.5cm{
\DR^{\rrd}(\cV,\nabla)\otimes\DR^{\rmod}(\cV,\nabla)\ar@<-8ex>[d]\ar[r]^-{\DRpairing_{\rrd,\rmod}}&(\cE^{\rrd,\cbbullet}_M,\rd)\ar@{=}[d]\\
\DR^{\rmod}(\cV,\nabla)\otimes\DR^{\rrd}(\cV,\nabla)\ar@<-8ex>[u]\ar[r]^-{\DRpairing_{\rmod,\rrd}}&(\cE^{\rrd,\cbbullet}_M,\rd)
}
\end{array}
\end{equation}
commutes in the following sense. For any rapid-decay forms $\omega_1^\rrd$ and $\omega_2^\rrd$ with coefficients in $\cV$, we let $\omega_1^\rmod,\omega_2^\rmod$ be the same forms considered as forms with moderate growth. Commutativity means that for any such $\omega_1^\rrd,\omega_2^\rrd$ the pairings satisfy
\[
\DRpairing_{\rrd,\rmod}(\omega_1^\rrd,\omega_2^\rmod)=\DRpairing_{\rmod,\rrd}(\omega_1^\rmod,\omega_2^\rrd).
\]

For any $r\geq0$, let $\DRpairing^{m-r}_{\rrd,\rmod}$ be the global de~Rham pairing
\[
\DRpairing^{m-r}_{\rrd,\rmod}:\Gamma(M,\cE_M^{\rrd,m-r}\otimes\cV)^\nabla\otimes\Gamma(M,\cE_M^{\rmod,r}\otimes\cV)^\nabla\to\CC
\]
obtained by restricting to horizontal sections the composition of $\Gamma(M,\DRpairing_{\rrd,\rmod})$ with integration $\int_M:\cE_M^{\rrd,m}\to\CC$. Define $\DRpairing^{m-r}_{\rmod,\rrd}$ similarly. Since rapid decay forms vanish on $\bM$, it~follows from Stokes formula that $\DRpairing^{m-r}_{\rrd,\rmod}$ (\resp $\DRpairing^{m-r}_{\rmod,\rrd}$) vanishes when one of the terms is a coboundary. We~thus obtain from the perfectness of \eqref{eq:rdmodpairings}:

\begin{cor}\label{cor:VerdierS}
Under Assumptions \ref{hyp:hyp}, the induced pairing
\[
\DRpairing^{m-r}_{\rrd,\rmod}:\coH_{\dR,\rrd}^{m-r}(M,\cV)\otimes \coH_{\dR,\rmod}^r(M,\cV)\to\CC
\]
is perfect, and so is $\DRpairing^{m-r}_{\rmod,\rrd}\,$.\qed
\end{cor}

\subsubsection*{Twisted currents with rapid decay and moderate growth}
We consider the following complexes of currents with coefficients in~$\cV$ (\cf \eqref{eq:tensorcurrents} for the boundary operator):
\begin{itemize}
\item
twisted currents with rapid decay
\[
(\Cb^{\rrd}_{M,\bbullet}(\cV),\partial)=(\Cb_{M,\bbullet}^\rmod,\partial)\otimes\DR^{\rrd}(\cV,\nabla)\simeq(\Cb_{M,\bbullet}^\rmod,\partial)\otimes\cV^\rrd,
\]
\item
twisted currents with moderate growth
\[
(\Cb^{\rmod}_{M,\bbullet}(\cV),\partial)=(\Cb_{M,\bbullet}^\rmod,\partial)\otimes\DR^{\rmod}(\cV,\nabla) \simeq(\Cb_{M,\bbullet}^\rmod,\partial)\otimes\cV^\rmod.
\]
\end{itemize}
We define respectively
the de~Rham homology with rapid decay and with moderate growth with coefficients in~$\cV$ as
\begin{align*}
\coH_{\dR,r}^\rrd(M,\cV)&=\coH_r\bigl(\Gamma(M,\Cb^{\rrd}_{M,\bbullet}(\cV))\bigr),\\
\coH_{\dR,r}^\rmod(M,\cV)&=\coH_r\bigl(\Gamma(M,\Cb^{\rmod}_{M,\bbullet}(\cV))\bigr)
\end{align*}
(recall that $M$ is compact). For example, when $(\cV,\nabla)=(\cC_M^\infty(*),\rd)$, rapid decay and moderate de~Rham homologies are respectively the homology and the Borel-Moore homology of~$M^\circ$ (\cf Remark~\ref{rem:BM}).

\subsubsection*{The Poincaré isomorphisms}
For the sake of simplicity, we will write the complex of currents cohomologically via the relation $\Cb_{M,q}^\rmod=\Cb_M^{\rmod,m-q}$, so that we have a quasi-isomorphism $\CC_M\isom(\Cb_M^{\rmod,\cbbullet},\partial)$ (\cf Proposi\-tion \ref{prop:Cmod}). This leads to Poincaré isomorphisms in the bounded derived category $\catD^\rb(\CC_M)$, making the following diagram commutative:\begin{equation}\label{eq:ccPVrdmod}
\begin{array}{c}
\xymatrix@C=1.8cm{
\DR^{\rrd}(\cV,\nabla)\ar[r]^-{\ccP^\rrd(\cV)}_-\sim\ar[d]&\Cb_M^{\rmod,\cbbullet}\otimes\DR^{\rrd}(\cV,\nabla)=\Cb^{\rrd,\cbbullet}_M(\cV)\ar@<-1.45cm>[d]\\
\DR^{\rmod}(\cV,\nabla)\ar[r]^-{\ccP^\rmod(\cV)}_-\sim&\Cb_M^{\rmod,\cbbullet}\otimes\DR^{\rmod}(\cV,\nabla)=\Cb_M^{\rmod,\cbbullet}(\cV).
}
\end{array}
\end{equation}
By taking hypercohomologies (recall that $M$ is compact) and switching back to homo\-logy on the right-hand side, we thus obtain global Poincaré isomorphisms\vspace*{-3pt}
\begin{equation}\label{eq:ccPVell}
\begin{split}
\ccP_r^\rrd(\cV):\coH^{m-r}_{\dR,\rrd}(M,\cV)&\isom \coH_{\dR,r}^\rrd(M,\cV),\\
\ccP_r^\rmod(\cV):\coH^{m-r}_{\dR,\rmod}(M,\cV)&\isom \coH_{\dR,r}^\rmod(M,\cV).
\end{split}
\end{equation}

\subsubsection*{De Rham period pairings}
We define the de~Rham period pairing
\begin{equation}\label{eq:Prdmod}
\Ppairing^{\rrd,\rmod}_\dR:(\Cb^{\rrd,\cbbullet}_{M}(\cV),\partial)\otimes\DR^{\rmod}(\cV,\nabla)\to(\Cb_{M}^{\rrd,\cbbullet},\partial)
\end{equation}
as follows:
for a $q$-dimensional current $T_q=T^{m-q}$ with moderate growth,
an $r$-form $\eta^r$ with rapid decay,
an $s$-form $\omega^s$ with moderate growth
and local sections $v,w$ of $\cV$,
the pairing of the local sections
$T_q\otimes\eta^r\otimes v$ and $\omega^s\otimes w$ is the current
\[
\Ppairing^{\rrd,\rmod}_\dR(T_q\otimes\eta^r\otimes v,\omega^s\otimes w)=T_q\otimes(\langle v,w\rangle\cdot\eta^r\wedge\omega^s)
\]
of dimension $q-r-s$.
Compatibility with differentials follows from \eqref{eq:tensorcurrents}. It has rapid decay since $\langle v,w\rangle$ has at most a pole (and hence moderate growth), and thus $\langle v,w\rangle\cdot\eta^r\wedge\omega^s$ has rapid decay. Note that, if $T_q$ is $C^\infty$, then we recover $\DRpairing_{\rrd,\rmod}$.

Similarly we define the period pairing
\begin{equation}\label{eq:Pmodrd}
\Ppairing^{\rmod,\rrd}_\dR:(\Cb^{\rmod,\cbbullet}_{M}(\cV),\partial)\otimes\DR^{\rrd}(\cV,\nabla)\to(\Cb_{M}^{\rrd,\cbbullet},\partial).
\end{equation}
Then the following (cohomological) diagram commutes:\vspace*{-3pt}
\begin{equation}\label{eq:Speriod}
\begin{array}{c}
\xymatrix@C=1.5cm{
(\Cb^{\rrd,\cbbullet}_{M}(\cV),\partial)\otimes\DR^{\rmod}(\cV,\nabla)
\ar[r]^-{\Ppairing^{\rrd,\rmod}_\dR}
&(\Cb_{M}^{\rrd,\cbbullet},\partial)\\
\DR^{\rrd}(\cV,\nabla)\otimes\DR^{\rmod}(\cV,\nabla)\ar[r]^-{\DRpairing_{\rrd,\rmod}}\ar@<-5ex>@{=}[u]\ar@<8ex>[u]^{\ccP^\rrd(\cV)}_\wr
&(\cE_M^{\rrd,\cbbullet},\partial)\ar[u]^\wr_{\ccP^\rrd}
}
\end{array}
\end{equation}
as well as a similar diagram for $\DRpairing_{\rmod,\rrd}$. For each $r\geq0$, let $\Ppairing^{\rrd,\rmod}_{\dR,r}$ and $\Ppairing^{\rmod,\rrd}_{\dR,r}$ be the global period pairings\vspace*{-3pt}
\begin{align*}
\Ppairing^{\rrd,\rmod}_{\dR,r}:\Gamma(M,\Cb_{M}^{\rrd,m-r}(\cV))^\partial\otimes\Gamma(M,\cE_M^{\rmod,r}\otimes\cV)^\nabla\to\Gamma(M,\Cb_{M}^{\rrd,m})\to\CC\\
\Ppairing^{\rmod,\rrd}_{\dR,r}:\Gamma(M,\Cb_{M}^{\rmod,m-r}(\cV))^\partial\otimes\Gamma(M,\cE_M^{\rrd,r}\otimes\cV)^\nabla\to\Gamma(M,\Cb_{M}^{\rrd,m})\to\CC
\end{align*}
obtained by evaluating on the constant function $1$ the zero-dimensional rapid-decay current provided by $\Ppairing^{\rrd,\rmod}_\dR$ and $\Ppairing^{\rmod,\rrd}_\dR$ respectively. Then, according to Stokes formula, $\Ppairing^{\rrd,\rmod}_{\dR,r}$ and $\Ppairing^{\rmod,\rrd}_{\dR,r}$ vanish if one of the terms is a boundary, and hence induce pairings\vspace*{-3pt}
\begin{align*}
\Ppairing^{\rrd,\rmod}_{\dR,r}:\coH_{\dR,r}^{\rrd}(M,\cV)\otimes \coH^r_{\dR,\rmod}(M,\cV)&\to\CC\\
\Ppairing^{\rmod,\rrd}_{\dR,r}:\coH_{\dR,r}^{\rmod}(M,\cV)\otimes \coH^r_{\dR,\rrd}(M,\cV)&\to\CC.
\end{align*}
It is then clear that the following diagram commutes (and similarly for $\DRpairing^{r}_{\rmod,\rrd}$):\vspace*{-3pt}
\begin{equation}\label{eq:DPP}
\begin{array}{c}
\xymatrix@C=1.8cm{
\coH_{\dR,r}^{\rrd}(M,\cV)\otimes \coH^r_{\dR,\rmod}(M,\cV)\ar[r]^-{\Ppairing^{\rrd,\rmod}_{\dR,r}}
&\CC\ar@{=}[d]\\
\coH^{m-r}_{\dR,\rrd}(M,\cV)\otimes \coH^r_{\dR,\rmod}(M,\cV)\ar[r]^-{\DRpairing^{m-r}_{\rrd,\rmod}}
\ar@<-5ex>@{=}[u]\ar@<8ex>[u]^{\ccP_r^\rrd(\cV)}_\wr
&\phantom{.}\CC.\\
}
\end{array}
\end{equation}
We then deduce from Corollary \ref{cor:VerdierS}:

\begin{cor}\label{cor:VerdierP}
The period pairings $\Ppairing^{\rrd,\rmod}_{\dR,r}$ and $\Ppairing^{\rmod,\rrd}_{\dR,m-r}$ are perfect.\qed
\end{cor}

\subsubsection*{De~Rham intersection pairing}
We can apply once more the Poincaré isomorphisms to \eqref{eq:Speriod}, and get the morphism $\Ipairing_\dR^{\rmod,\rrd}$ (\resp $\Ipairing_\dR^{\rrd,\rmod}$) making the following (cohomological) diagram commute:\vspace*{-5pt}
\[
\xymatrix@C=1.5cm{
\Cb_{M}^{\rmod,\cbbullet}(\cV)\otimes \DR^{\rrd}(\cV,\nabla)\ar[r]^-{\Ppairing^{\rmod,\rrd}_\dR}
\ar@<-8ex>@{=}[d]\ar@<5ex>[d]^{\ccP^\rrd(\cV)}_\wr
&(\Cb_{M}^{\rrd,\cbbullet},\partial)\ar@{=}[d]\\
\Cb_{M}^{\rmod,\cbbullet}(\cV)\otimes \Cb_{M}^{\rrd,\cbbullet}(\cV)\ar[r]^-{\Ipairing_\dR^{\rmod,\rrd}}&(\Cb_{M}^{\rrd,\cbbullet},\partial)
}
\]
and therefore also the following:\vspace*{-5pt}
\begin{equation}\label{eq:BPP}
\begin{array}{c}
\xymatrix@C=1.8cm{
\coH_{\dR,m-r}^{\rmod}(M,\cV)\otimes \coH^{m-r}_{\dR,\rrd}(M,\cV)\ar[r]^-{\Ppairing^{\rmod,\rrd}_{\dR,m-r}}
\ar@<-8ex>@{=}[d]\ar@<5ex>[d]^{\ccP_r^\rrd(\cV)}_\wr
&\CC\ar@{=}[d]\\
\coH_{\dR,m-r}^{\rmod}(M,\cV)\otimes \coH_{\dR,r}^{\rrd}(\cV)\ar[r]^-{\Ipairing_{\dR,m-r}^{\rmod,\rrd}}&\phantom{.}\CC.
}
\end{array}
\end{equation}
For cohomology classes $[\eta^r]\in\coH^{m-r}_{\dR,\rrd}(M,\cV)$ and $[\omega^{m-r}]\in\coH^r_{\dR,\rmod}(M,\cV)$, we have by definition
\[
\DRpairing_{\rrd,\rmod}^r([\eta^r],[\omega^{m-r}])=\Ipairing_{\dR,m-r}^{\rrd,\rmod}(\ccP^\rrd_{m-r}([\eta^r]),\ccP^\rmod_r([\omega^{m-r}])).
\]

\begin{cor}\label{cor:VerdierI}
The de~Rham intersection pairings $\Ipairing_{\dR,r}^{\rrd,\rmod},\Ipairing_{\dR,r}^{\rmod,\rrd}$ are perfect for each $r\geq0$.
\end{cor}

\begin{proof}
This follows from the perfectness of the period pairings (\cf Corollary \ref{cor:VerdierP}).
\end{proof}

\subsubsection*{Quadratic relations}
We assume that the cohomology and homology vector spaces we consider are finite-dimensional and we fix bases of these spaces. We denote by the same letter a pairing and its matrix with respect to the corresponding bases, and we use the matrix notation in which, for a pairing $\mathsf A$, we also denote $\mathsf A(x\otimes y)$ by ${}^tx\cdot\mathsf A\cdot y$,
and the transpose of the matrix $\mathsf{A}$ by ${}^t\mathsf A$.

\begin{cor}[Quadratic relations]\label{cor:quadraticreldR}
The matrices of the pairings satisfy the following relations:
\[
{}^t\Ipairing_{\dR,m-r}^{\rmod,\rrd}=\Ppairing^{\rrd,\rmod}_{\dR,r}\cdot(\DRpairing^{m-r}_{\rrd,\rmod})^{-1}\cdot{}^t\Ppairing^{\rmod,\rrd}_{\dR,m-r}.
\]
\end{cor}

\begin{proof}
In terms of matrices, the commutations \eqref{eq:DPP} and \eqref{eq:BPP} read
\[
{}^t\ccP_r^\rrd(\cV)\cdot\Ppairing^{\rrd,\rmod}_{\dR,r}=\DRpairing^{m-r}_{\rrd,\rmod},\quad
\Ipairing_{\dR,m-r}^{\rmod,\rrd}\cdot\ccP_r^\rrd(\cV)=\Ppairing^{\rmod,\rrd}_{\dR,m-r},
\]
and, by eliminating $\ccP_r^\rrd(\cV)$, we obtain the sought for relation.
\end{proof}

See Remark \ref{rem:symmetrymid} for a justification of the terminology ``quadratic relations''.

\begin{remark}[Quadratic relations in presence of $\pm$-symmetry]\label{rem:symmetry}
Let us assume that $\gen{\cbbullet,\cbbullet}$ is either symmetric or skew-symmetric, a property that we call \emph{$\pm$\nobreakdash-sym\-metric.} Let us denote by
\[
{}^t\DRpairing^{m-r}_{\rrd,\rmod}:\coH_{\dR,\rmod}^r(M,\cV)\otimes \coH_{\dR,\rrd}^{m-r}(M,\cV)\to\CC
\]
the pairing defined (with obvious notation) by
\[
{}^t\DRpairing^{m-r}_{\rrd,\rmod}(h^r,h^{m-r})=\DRpairing^{m-r}_{\rrd,\rmod}(h^{m-r},h^r),
\]
and let ${}^t\DRpairing^{m-r}_{\rmod,\rrd}$ be defined similarly. If~$\eta^{m-r}$ is a rapid decay $(m-r)$-form and~$\omega^r$ an $r$-form with moderate growth, and if $v,w$ are local sections of $\cV$, we have
\begin{align*}
\DRpairing_{\rrd,\rmod}^{m-r}(\eta^{m-r}\otimes v,\omega^r\otimes w)&=\langle v,w\rangle\cdot\eta^{m-r}\wedge\omega^r\\
&=\pm(-1)^{r(m-r)}\langle w,v\rangle\omega^r\wedge\eta^{m-r}\\
&=\pm(-1)^{r(m-r)}\DRpairing_{\rmod,\rrd}^r(\omega^r\otimes w,\eta^{m-r}\otimes v).
\end{align*}
Passing to cohomology, we find
\[
{}^t\DRpairing^{m-r}_{\rrd,\rmod}=\pm(-1)^{r(m-r)}\DRpairing_{\rmod,\rrd}^r,\qquad {}^t\DRpairing^{m-r}_{\rmod,\rrd}=\pm(-1)^{r(m-r)}\DRpairing_{\rrd,\rmod}^r.
\]
For the de~Rham intersection pairing, we similarly find
\[
{}^t\Ipairing^{\rmod,\rrd}_{\dR,m-r}=\pm(-1)^{r(m-r)}\Ipairing_{\dR,r}^{\rrd,\rmod}.
\]
As a consequence, the quadratic relations read, in term of matrices,
\begin{equation*}
\pm(-1)^{r(m-r)}\Ipairing_{\dR,r}^{\rrd,\rmod}=\Ppairing^{\rrd,\rmod}_{\dR,r}\cdot(\DRpairing^{m-r}_{\rrd,\rmod})^{-1}\cdot{}^t\Ppairing^{\rmod,\rrd}_{\dR,m-r}.
\end{equation*}
\end{remark}

\subsection{Middle quadratic relations}\label{subsec:middlequadrel}
For each degree $r$, we set
\begin{align*}
\coH^\rmid_{\dR,r}(M,\cV)&=\image\bigl[\coH^\rrd_{\dR,r}(M,\cV)\to \coH^\rmod_{\dR,r}(M,\cV)\bigr],\\
\coH^{m-r}_{\dR,\rmid}(M,\cV)&=\image\bigl[\coH^{m-r}_{\dR,\rrd}(M,\cV)\to \coH^{m-r}_{\dR,\rmod}(M,\cV)\bigr],
\end{align*}
where the maps are induced by the natural inclusion $\cC_M^{\infty,\rrd} \to \cC_M^{\infty,\rmod}$.
From \eqref{eq:Srdmod} one deduces that $\DRpairing^{m-r}_{\rrd,\rmod}$ and $\DRpairing^{m-r}_{\rmod,\rrd}$ induce the same \emph{non-degenerate} pairing
\[
\DRpairing^{m-r}_{\rmid}:\coH_{\dR,\rmid}^{m-r}(M,\cV)\otimes \coH_{\dR,\rmid}^r(M,\cV)\to\CC.
\]
Similarly, according to \eqref{eq:ccPVrdmod}, $\ccP_r^\rrd(\cV)$ and $\ccP_r^\rmod(\cV)$ induce the same isomorphism
\[
\ccP_r^\rmid(\cV):\coH_{\dR,\rmid}^{m-r}(M,\cV)\to \coH_{\dR,r}^\rmid(M,\cV).
\]
It follows that $\Ppairing^{\rrd,\rmod}_{\dR,r}$ and $\Ppairing^{\rmod,\rrd}_{\dR,r}$ induce the same \emph{non-degenerate} period pairing
\[
\Ppairing^{\rmid}_{\dR,r}:\coH_{\dR,r}^\rmid(M,\cV)\otimes \coH_{\dR,\rmid}^{m-r}(M,\cV)\to\CC.
\]
In the same way, $\Ipairing^{\rrd,\rmod}_{\dR,r}$ and $\Ipairing^{\rmod,\rrd}_{\dR,r}$ induce the same \emph{non-degenerate} intersection pairing
\[
\Ipairing^{\rmid}_{\dR,r}:\coH_{\dR,r}^\rmid(M,\cV)\otimes \coH_{\dR,{m-r}}^\rmid(M,\cV)\to\CC.
\]

We conclude:

\begin{cor}[Middle quadratic relations]\label{cor:midquadraticreldR}
These pairings are non-degenerate and their matrices satisfy the following relations:
\[
{}^t\Ipairing_{\dR,m-r}^{\rmid}=\Ppairing^{\rmid}_{\dR,r}\cdot(\DRpairing^{m-r}_{\rmid})^{-1}\cdot{}^t\Ppairing^{\rmid}_{\dR,m-r}.
\]
\end{cor}

\begin{remark}[$\pm$-Symmetry and quadratic relations]\label{rem:symmetrymid}
Under the symmetry assumption of Remark \ref{rem:symmetry}, the following relations hold:
\[
{}^t\Ipairing_{\dR,m-r}^{\rmid}=\pm(-1)^{r(m-r)}\Ipairing^{\rmid}_{\dR,r},\quad {}^t\DRpairing^r_{\rmid}=\pm(-1)^{r(m-r)}\DRpairing_{\rmid}^{m-r},
\]
and the relations of Corollary \ref{cor:midquadraticreldR} read
\begin{equation*}
\pm(-1)^{r(m-r)}\Ipairing^{\rmid}_{\dR,r}=\Ppairing^{\rmid}_{\dR,r}\cdot(\DRpairing^{m-r}_{\rmid})^{-1}\cdot{}^t\Ppairing^{\rmid}_{\dR,m-r}.
\end{equation*}
If $r=m-r$, we can regard these formulas as \emph{algebraic relations of degree two} on the entries of the matrix $\Ppairing^{\rmid}_{\dR,r}$, with coefficients in the entries of $\Ipairing_{\dR,r}^{\rmid}$ and $\DRpairing^{m-r}_{\rmid}$. This justifies the name ``quadratic relations'' which, strictly speaking, only applies to the middle periods in middle dimension.
\end{remark}

\begin{caveat}[On the notation $\rmid$]
The terminology ``middle quadratic relation'' and the associated notation $\rmid$ could be confusing,
as it is usually associated with the notion of intermediate (or middle) extension of a sheaf across a divisor.
Here, we use it in the naive sense above, and we do not claim in general (that is, in the setting of Section \ref{subsec:comparisonperiods} below) any precise relation with the usual notion of intermediate extension. However, we will check that both notions coincide in complex dimension one (\cf \eqref{eq:Hmid}).
\end{caveat}

\subsection{Betti formalism for quadratic relations}\label{subsec:Bettiquadratic}
In the topological setting we replace the complex of currents $(\Cb_{M,\bbullet}^\rmod,\partial)$ with the complex of sheaves of relative singular chains $(\ccC_{M,\bM,\bbullet},\partial)$. We~will consider the latter sheaves with coefficients in $\CC$ since we will compare them with sheaves of currents. We refer to Appendix~\ref{sec:chains} for basic properties of this complex, and we recall (\cf \eg \cite[p.\,14]{Hien09}) that $(\ccC_{M,\bM,\bbullet},\partial)$, when regarded cohomologically by setting $\ccC_{M,\bM}^{m-\bbullet}=\ccC_{M,\bM,\bbullet}$, is a homotopically fine resolution of~$\CC_M$. Proposition \ref{prop:Csm} enables us to replace the complex of sheaves of singular chains with that of piecewise smooth singular chains, that we will denote in the same way.

\subsubsection*{Betti period pairings}
Integration along a piecewise smooth singular chain in $M$ of a test form with rapid decay along~$\bM$ \hbox{defines} a morphism of chain complexes $(\ccC_{M,\bM,\bbullet},\partial)\to(\Cb_{M,\bbullet}^\rmod,\partial)$: indeed, if~$c$ is a $q$-dimensional chain and $\omega$ a form of degree $q$ on $M^\circ$ with rapid decay, $\int_c\omega=0$ if $c$ is supported in $\bM$; compatibility with $\partial$ follows from the Stokes formula. This morphism is a quasi-isomorphism since it induces an isomorphism on the unique non-zero homology sheaf of both complexes; it yields a morphism between two resolutions of $\CC_M$ taken in homological degree $m$.

We introduce (see Section \ref{sec:chains}) the following complexes of singular chains:
\begin{itemize}
\item
$(\ccC_{M,\bM,\bbullet}(\cV^\rrd),\partial)=(\ccC_{M,\bM,\bbullet},\partial)\otimes\cV^{\rrd}$: chains with coefficients in $\cV$ and rapid decay,
\item
$(\ccC_{M,\bM,\bbullet}(\cV^\rmod),\partial)=(\ccC_{M,\bM,\bbullet},\partial)\otimes\cV^{\rmod}$: chains with coefficients in~$\cV$ and moderate growth.
\end{itemize}
We also set
\begin{align*}
\coH_r^\rrd(M,\cV)&=\coH_r\bigl(\Gamma(\ccC_{M,\bM,\bbullet}(\cV^\rrd),\partial)\bigr)\\
\tag*{and}
\coH_r^\rmod(M,\cV)&=\coH_r\bigl(\Gamma(\ccC_{M,\bM,\bbullet}(\cV^\rmod),\partial)\bigr).
\end{align*}
We thus have quasi-isomorphisms of chain complexes
\begin{equation}\label{eq:singcurrents}
\begin{split}
(\ccC_{M,\bM,\bbullet}(\cV^\rrd),\partial)&\isom(\Cb_{M,\bbullet}^\rrd(\cV),\partial),\\
(\ccC_{M,\bM,\bbullet}(\cV^\rmod),\partial)&\isom(\Cb_{M,\bbullet}^\rmod(\cV),\partial),
\end{split}
\end{equation}
which induce isomorphisms
\begin{equation}\label{eq:isoBdR}
\coH_r^\rrd(M,\cV)\isom \coH_{\dR,r}^\rrd(M,\cV)\quad\text{and}\quad \coH_r^\rmod(M,\cV)\isom \coH_{\dR,r}^\rmod(M,\cV).
\end{equation}

Composing \eqref{eq:Prdmod} and \eqref{eq:Pmodrd} with these morphisms gives back the period pairings as those defined in \cite{Hien09}, that we denote by $\Ppairing^{\rrd,\rmod}_r$ and $\Ppairing^{\rmod,\rrd}_r$. In particular, these Betti period pairings are perfect. Notice that, by means of these identifications, the Poincaré-de~Rham isomorphisms~\eqref{eq:ccPVell} correspond to the Poincaré isomorphisms, denoted similarly:
\begin{equation}\label{eq:ccPVellB}
\begin{split}
\ccP_r^\rrd(\cV):\coH^{m-r}(M,\cV^\rrd)&\isom \coH_r(M,\cV^\rrd),\\
\ccP_r^\rmod(\cV):\coH^{m-r}(M,\cV^\rmod)&\isom \coH_r(M,\cV^\rmod).
\end{split}
\end{equation}

According to Propositions \ref{prop:comparisonhomology} and \ref{prop:Csm}, the homology spaces $\coH_r^\rrd(M,\cV)$ (\resp $\coH_r^\rmod(M,\cV)$) can be computed as the homology of the chain complexes
\[
(\wt\ccC^\sm_{M,\bM,\bbullet}(M,\cV^\rrd)),\partial)\quad (\resp (\wt\ccC^\sm_{M,\bM,\bbullet}(M,\cV^\rmod)),\partial))
\]
of piecewise smooth singular chains with coefficients in $\cV^\rrd$ (\resp $\cV^\rmod$).

On the one hand, let $\sigma:\Delta_r\to M$ be a piecewise smooth singular simplex (that we can assume not contained in~$\bM$ since we only consider relative simplices), and let $v$ be a section of $\cV^\rrd$ in the neighbourhood of $|\sigma|$, so that $\sigma\otimes v$ belongs to $\wt\ccC_{M,\bM,\bbullet}^{\sm}(M,\cV^\rrd)$. On the other hand, note that, by using a partition of unity, any section in $\Gamma(M,\cE_M^{\rmod,r}\otimes\cV)$ is a sum of terms $\omega^r\otimes w$, where~$w$ is a section of~$\cV$
in the neighbourhood of the support of the $C^\infty$ $r$\nobreakdash-form~$\omega^r$. Then $\langle v,w\rangle$ is a $C^\infty$ function on the open set where both sections are defined and has rapid decay along~$\bM$, so that $\langle v,w\rangle\cdot\omega^r$ is a $C^\infty$-form with rapid decay, and hence integrable along~$\sigma$. We have
\[
\Ppairing^{\rrd,\rmod}_r(\sigma\otimes v,\omega^r\otimes w)=\int_\sigma\langle v,w\rangle\cdot\omega^r=\int_{\Delta_r}\sigma^*\langle v,w\rangle\cdot\sigma^*\omega^r.
\]

This interpretation makes the link between the approaches in \cite{B-E04b} and \cite{Hien09}. A~similar interpretation holds for $\Ppairing^{\rmod,\rrd}_r$: now $\langle v,w\rangle$ has only moderate growth, but~$\omega^r$ has rapid decay, so that~$\langle v,w\rangle\cdot\omega^r$ remains a $C^\infty$-form with rapid decay.

\subsubsection*{Existence of a $\kk$-structure}
Let us fix a subfield $\kk$ of $\CC$. The sheaves $\ccC_{M,\bM,\bbullet}$ are now considered with coefficients in $\QQ$. We will make explicit the $\kk$-structure on the various cohomology groups occurring in the de~Rham model, provided that a $\kk$-structure exists on the underlying sheaves. If we are given a $\kk$-structure $(\cV^\nabla)_\kk$ of $\cV^\nabla$, we enrich Assumptions \ref{hyp:hyp} as follows.

\begin{assumption}\mbox{}\label{hyp:k}
There exist subsheaves $\cV_\kk^\rrd\hto\cV_\kk^\rmod\hto j_*(\cV^\nabla)_\kk$ of $\kk$\nobreakdash-vector spaces of $j_*\cV^\nabla$ and a non-degenerate pairing $\gen{\cbbullet,\cbbullet}_\kk$ on $(\cV^\nabla)_\kk$ giving rise to $\cV^\rrd\hto\cV^\rmod\hto j_*\cV^\nabla$ and $\gen{\cbbullet,\cbbullet}$ after tensoring with $\CC$.
\end{assumption}

As a consequence, the perfect pairings $\cH^0\DRpairing_{\rrd,\rmod}$ and $\cH^0\DRpairing_{\rmod,\rrd}$ in Assumption \ref{hyp:hyp}\eqref{hyp:2} are defined over $\kk$:
\[
\cV^\rrd_\kk\otimes\cV^\rmod_\kk\to j_!\kk_{M^\circ},\quad
\cV^\rmod_\kk\otimes\cV^\rrd_\kk\to j_!\kk_{M^\circ}.
\]
In particular, $\cH^0$ of the commutative diagram \eqref{eq:Srdmod} is defined over $\kk$. Furthermore, the de~Rham isomorphism induces a $\kk$-structure on the de~Rham cohomology groups, by setting for example $\coH^r_{\dR,\rrd}(M,\cV)_\kk=\coH^r(M,\cV^\rrd_\kk)$ via the de~Rham isomorphism $\coH^r_{\dR,\rrd}(M,\cV)\isom\coH^r(M,\cV^\rrd)$. The pairings of Corollary \ref{cor:VerdierS} are defined over $\kk$ and correspond to Poincaré-Verdier duality over~$\kk$.

The $\kk$-structures on rapid decay and moderate growth cycles
are obtained by means of the complexes $(\ccC_{M,\bM,\bbullet},\partial)\otimes\cV^{\rrd}_\kk$ and $(\ccC_{M,\bM,\bbullet},\partial)\otimes\cV^{\rmod}_\kk$. In other words, we set
\[
\coH_r^{\rrd}(M,\cV)_\kk=\coH_r(M,\cV^\rrd_\kk),\ \text{etc.}
\]
The Poincaré isomorphisms \eqref{eq:ccPVellB} are then defined over $\kk$.

\begin{prop}\label{prop:kstructDR}
Under Assumptions \ref{hyp:hyp} and \ref{hyp:k}, let $(v_i)_i$ be a basis of $\coH^{m-r}_{\dR,\rrd}(M,\cV)$ and let $(v_i^\vee)_i$ be the $\DRpairing^{m-r}_{\rrd,\rmod}$-dual basis of $\coH^r_{\dR,\rmod}(M,\cV)$. The $\kk$-vector space $\coH^{m-r}_{\dR,\rrd}(M,\cV)_\kk$ is the~$\kk$\nobreakdash-subspace of $\coH^{m-r}_{\dR,\rrd}(M,\cV)$ consisting of the elements
\[
\sum_i\Ppairing^{\rrd,\rmod}_{\dR,r}(\alpha,v_i^\vee)\cdot v_i,
\]
where $\alpha$ runs in $\coH_r(M,\cV^\rrd_\kk)$. A similar assertion holds for $\coH^{m-r}_{\dR,\rmod}(M,\cV)_\kk$.
\end{prop}

\begin{proof}
Let $e\in \coH^{m-r}_{\dR,\rrd}(M,\cV)$ be an element written as
\[
e=\sum_i\DRpairing_{\rrd,\rmod}^{m-r}(e,v_i^\vee)\cdot v_i.
\]
Let us set $\alpha=\ccP_r^\rrd(e)\in\coH_{\dR,r}^\rrd(M,\cV)$. According to \eqref{eq:DPP}, we can rewrite $e$ as
\[
e=\sum_i\Ppairing^{\rrd,\rmod}_{\dR,r}(\alpha,v_i^\vee)\cdot v_i.
\]

The class $e$ belongs to $\coH^{m-r}_{\dR,\rrd}(M,\cV)_\kk$ if and only if,
when regarded as in $\coH^{m-r}(M,\cV^\rrd)$,
it belongs to $\coH^{m-r}(M,\cV^\rrd_\kk)$, equivalently, its image $\alpha$ by the Poincaré isomorphism $\ccP_r^\rrd(\cV)$ belongs to~$\coH_r(M,\cV^\rrd_\kk)$. The assertion of the proposition follows from the above identifications.
\end{proof}

\subsubsection*{Middle quadratic relations}\label{subsec:intersection}
The Betti intersection pairing
\[
\Ipairing^{\rrd,\rmod}_r:\coH_r^\rrd(M,\cV)\otimes \coH_{m-r}^\rmod(M,\cV)\to\CC
\]
is defined from the de~Rham intersection pairing $\Ipairing^{\rrd,\rmod}_{\dR,r}$ (\cf Corollary \ref{cor:VerdierI}) by composing it with the isomorphisms \eqref{eq:isoBdR}. It is thus a \emph{perfect pairing}. A similar result holds for $\Ipairing^{\rmod,\rrd}_r$.

From Corollary \ref{cor:midquadraticreldR} we obtain immediately (and of course similarly for the rapid decay and mod\-erate growth analogues):

\begin{cor}\label{cor:quadraticrelB}
The middle Betti period pairings satisfy the quadratic relations of Corollary \ref{cor:midquadraticreldR}, where the de~Rham intersection pairing $\Ipairing^\rmid_{\dR,r}$ is replaced with the Betti intersection pairing~$\Ipairing^\rmid_r$. That is, in term of matrices,
\[
{}^t\Ipairing^\rmid_{m-r}=\Ppairing^{\rmid}_r\cdot(\DRpairing_\rmid^{m-r})^{-1}\cdot{}^t\Ppairing^{\rmid}_{m-r}.
\]
Under the symmetry assumption of Remark \ref{rem:symmetrymid}, it reads
\[
\pm(-1)^{r(m-r)}\Ipairing^{\rmid}_r=\Ppairing^{\rmid}_r\cdot(\DRpairing_\rmid^{m-r})^{-1}\cdot{}^t\Ppairing^{\rmid}_{m-r}.
\]
\end{cor}

\subsubsection*{Computation of the Betti intersection pairing}
Let us assume that $(M,\bM)$ is endowed with a simplicial decom\-po\-sition $\cT$ such that the sheaves $\cV^\rrd,\cV^\rmod$ satisfy Assump\-tion \ref{ass:simplicial} and are locally constant on $M^\circ$. Then we can \hbox{replace} the complex $(\wt\ccC_{M,\bM,\bbullet}^{\sm}(M,\cV^\rrd)),\partial)$, respec\-tively the complex $(\wt\ccC_{M,\bM,\bbullet}^{\sm}(M,\cV^\rmod)),\partial)$, with the corresponding simplicial chain complex $(\wt\ccC_{M,\bM,\bbullet}^\triup(M,\cV^\rrd)),\partial)$, respectively $(\wt\ccC_{M,\bM,\bbullet}^\smallsquare(M,\cV^\rmod)),\partial)$ (\cf Section \ref{subsec:appdualchain}).

The Betti intersection pairing can easily be computed in the framework of simplicial chain complexes under this assumption. Indeed, choose an orientation for each simplex (it is natural to assume that the maximal-dimensional simplices have the orientation induced by that of $M$) and an orientation for each dual cell $\ov D(\sigma)$. Let $\sigma_r\otimes v$ be a simplex of dimension $r$ with coefficient~$v$ in $\cV^\rrd$ ($\sigma_r\not\subset\bM$) and let $\ov D(\sigma'_r)\otimes w$ be a cell of codimension $r$ with coefficient $w$ in $\cV^\rmod$ for some simplex $\sigma'_r$ of $\cT$ (possibly $\sigma'_r\subset\bM$). We regard them as currents with rapid decay and moderate growth respectively, according to \eqref{eq:singcurrents}.

\begin{prop}[Computation of the Betti intersection pairing]\label{prop:computationBetti}
With these assumptions,
\[
\Ipairing^{\rrd,\rmod}_r(\sigma_r\otimes v,\ov D(\sigma'_r)\otimes w)=
\begin{cases}
0&\text{if }\sigma'_r\neq\sigma_r,\\
\varepsilon(\sigma_r)\cdot\langle v,w\rangle(\wh\sigma_r)&\text{if }\sigma'_r=\sigma_r,
\end{cases}
\]
where $\wh\sigma_r$ is the barycenter of $\sigma_r$ and $\varepsilon(\sigma_r)=\pm1$ is the orientation change between $\sigma_r\times\ov D(\sigma_r)$ and~$M$.
\end{prop}

\begin{proof}
By definition, we have
\begin{equation}\label{eq:Brrdmod}
\Ipairing^{\rrd,\rmod}_r(\sigma_r\otimes v,\ov D(\sigma'_r)\otimes w)=\langle v,w\rangle\cdot \Ipairing_{\dR,r}(\sigma_r,\ov D(\sigma'_r)),
\end{equation}
where $\Ipairing_{\dR,r}$ corresponds to de~Rham's definition of the Kronecker index (\cf\cite[\S20]{deRham84}).

For any $r$-dimensional simplex $\sigma'_r$ of~$\cT$, $\sigma'_r$ is the only $r$-dimensional simplex of~$\cT$ physically intersected by the cell $\ov D(\sigma'_r)$ (\cf Lemma \ref{lem:A7}). Hence, given \hbox{$\sigma_r\!\not\subset\!\bM$}, we have the equality $\Ipairing_{\dR,r}(\sigma_r,\ov D(\sigma'_r))=0$ unless $\sigma'_r=\sigma_r$ and we are reduced to computing \eqref{eq:Brrdmod} in that case. This computation, done by de~Rham (\cf\cite[p.\,85--86]{deRham84}), is recalled in the appendix (Proposition~\ref{prop:Kronecker}). The current $\Ipairing_{\dR,r}(\sigma_r,\ov D(\sigma_r))$ is supported on $\wh\sigma_r=\sigma_r\cap \ov D(\sigma_r)$ and has coefficient $\varepsilon(\sigma_r)$, so~that
\[
\langle v,w\rangle\cdot \Ipairing_{\dR,r}(\sigma_r,\ov D(\sigma_r))=\langle v,w\rangle(\wh\sigma_r)\cdot \Ipairing_{\dR,r}(\sigma_r,\ov D(\sigma_r))=\varepsilon(\sigma_r)\langle v,w\rangle(\wh\sigma_r).\qedhere
\]
\end{proof}

\subsection{Period realizations}
We keep the setting and notation of Section \ref{subsec:Bettiquadratic}. Let $\bK,\kk$ be two subfields of $\CC$. The abelian category $\catPer_{\bK,\kk}$ of \emph{$(\bK,\kk)$-period structures} has objects consisting of pairs of finite-dimensional $\bK$, respectively $\kk$, vector spaces $(V_{\dR},V_\Betti)$ together with an isomorphism $\per:\CC\otimes_\kk V_\Betti\isom \CC\otimes_\bK V_\dR$. The morphisms are the natural ones. It is convenient to consider this category from a dual point of view. Namely, we consider the category $\catPer_{\bK,\kk}^\vee$ whose objects are triples $(V_\dR,V_\Betti,\Ppairing)$, where $\Ppairing:(\CC\otimes V_\dR)\otimes(\CC\otimes V_\Betti)\to\CC$ is a non-degenerate pairing, and whose morphisms are the natural ones. The duality functor $V_\Betti\mto V_\Betti^\vee$ induces an equivalence $T:\catPer_{\bK,\kk}\isom\catPer_{\bK,\kk}^\vee$ by means of the tautological pairing $(\cbbullet\sep\cbbullet):V_\kk\otimes V_\kk^\vee\to\kk$.

Let us consider $(\cV,\nabla,\gen{\cbbullet,\cbbullet})$ satisfying Assumptions \ref{hyp:hyp} and \ref{hyp:k}, and let us restrict to the case where $\bK=\CC$. It defines moderate and rapid decay period structures:
\begin{gather*}
\bigl(\coH^r_{\dR,\rmod}(M,\cV),\coH^r(M,\cV^\rmod_\kk),\per_\rmod\bigr),\\
\bigl(\coH^r_{\dR,\rrd}(M,\cV),\coH^r(M,\cV^\rrd_\kk),\per_\rrd\bigr),
\end{gather*}
where $\per_\rmod:\coH^r(M,\cV^\rmod_\kk)\to\coH^r_{\dR,\rmod}(M,\cV)$ is composed from the natural isomorphism $\coH^r(M,\cV^\rmod_\kk)\otimes\CC\simeq\coH^r(M,\cV^\rmod)$ and the inverse of the de~Rham isomorphism $\coH^r_{\dR,\rmod}(M,\cV)\isom\coH^r(M,\cV^\rmod)$, and $\per_\rrd$ is defined similarly.

\begin{lemma}
There is a natural isomorphism
\[
\bigl(\coH^r_{\dR,\rrd}(M,\cV),\coH^r(M,\cV^\rrd_\kk),\per_\rrd\bigr)
\simeq\bigl(\coH^{m-r}_{\dR,\rmod}(M,\cV),\coH^{m-r}(M,\cV^\rmod_\kk),\per_\rmod\bigr)^\vee.
\]
\end{lemma}

\begin{proof}
This follows from the compatibility of the pairing $\Qpairing$ and the Poincaré-Verdier duality pairing via the de~Rham isomorphism.
\end{proof}

\begin{prop}
There are isomorphisms of $(\CC,\kk)$-period structures
\begin{align*}
T\bigl(\coH^r_{\dR,\rmod}(M,\cV),\coH^r(M,\cV^\rmod_\kk),\per_\rmod\bigr)
&\simeq\bigl(\coH^r_{\dR,\rmod}(M,\cV),\coH_r^\rrd(M,\cV)_\kk,\Ppairing^{\rrd,\rmod}_r\bigr),\\
T\bigl(\coH^r_{\dR,\rrd}(M,\cV),\coH^r(M,\cV^\rrd_\kk),\per_\rrd\bigr)
&\simeq\bigl(\coH^r_{\dR,\rrd}(M,\cV),\coH_r^\rmod(M,\cV)_\kk,\Ppairing^{\rmod,\rrd}_r\bigr).
\end{align*}
\end{prop}

\begin{proof}
We will check the first case for example. The pairing corresponding to the dual presentation of the left-hand side amounts, according to the Poincaré-Verdier duality pairing
\[
\coH^{m-r}(M,\cV^\rrd_\kk)\otimes\coH^r(M,\cV^\rmod_\kk)\to\kk,
\]
to the period pairing
\[
\Qpairing_\kk^{m-r}:\coH^{m-r}(M,\cV^\rrd_\kk)\otimes\coH^r_{\dR,\rmod}(M,\cV)\to\CC,
\]
which corresponds
\begin{itemize}
\item
either to the complexified Poincaré-Verdier pairing via the de~Rham isomorphism
\[
\coH^r_{\dR,\rmod}(M,\cV)\!\isom\!\coH^r(M,\cV^\rmod),
\]
\item
or to the pairing $\Qpairing^{m-r}_{\rrd,\rmod}$ after complexifying the first term and applying the corresponding de~Rham isomorphism.
\end{itemize}
Let us recall that the Poincaré isomorphism \eqref{eq:ccPVell} is defined over $\kk$ (\cf\eqref{eq:ccPVellB}), so that, applying it to the first term, we obtain the period pairing $\Ppairing^{\rrd,\rmod}_r$.
\end{proof}

\let\DRpairing\Spairing
\subsection{Quadratic relations for good meromorphic flat bundles}\label{subsec:comparisonperiods}
We now consider the complex setup. Let $X$ be a compact complex manifold of complex dimension~$n$, let~$\Div$ be a divisor with normal crossings, and let $(\cV,\nabla)$ be a coherent $\cO_X(*\Div)$-module with integrable connection, endowed with a non-degenerate flat pairing $\gen{\cbbullet,\cbbullet}:\cV\otimes\cV\to\cO_X(*\Div)$. Let \hbox{$\varpi:\wt X\to X$} denote the oriented real blow-up
of $X$ along the irreducible components of $\Div$. In a local chart of $X$ with coordinates $x_1,\dots,x_n$ where~$\Div$ is defined by the equation $x_1\cdots x_\ell=0$, $\wt X$ is the space of polar coordinates $(r_1,e^{\sfi\theta_1},\dots,r_\ell,e^{\sfi\theta_\ell},x_{\ell+1},\dots,x_n)$ and the $i$-th component of $\varpi$ ($1\leq i\leq\ell$) sends $(r_i,e^{\sfi\theta_i})$ to $x_i=r_i\cdot e^{\sfi\theta_i}$ (\cf\eg\cite[\S8.2]{Bibi10} for the global construction of $(\wt X,\varpi)$). We~set $\wt\Div=\varpi^{-1}(\Div)=\partial\wt X$, and $(\wt X,\wt\Div)$ will play the role of $(M,\bM)$ in the settings~\ref{subsec:settings}.

There is a natural Cauchy-Riemann operator acting on the sheaf of $C^\infty$ functions on $\wt X$ and one defines the sheaves $\cA_{\wt X},\cA_{\wt X}^\rrd,\cA_{\wt X}^\rmod$ respectively as the subsheaves of $\cC_{\wt X}^\infty,\cC_{\wt X}^{\infty,\rrd},\cC_{\wt X}^{\infty,\rmod}$ annihilated by the Cauchy-Riemann operator (\cf\eg\cite{Bibi97}). Setting $\wt\cV=\cA_{\wt X}\otimes_{\varpi^{-1}\cO_X}\varpi^{-1}\cV$, the connection~$\nabla$ lifts to a connection $\wt\cV\to\varpi^{-1}\Omega^1_X\otimes_{\varpi^{-1}\cO_X}\wt\cV$. There is thus the corresponding notion of holomorphic de~Rham complexes with growth conditions, that we denote by $\DR^\rrd(\wt\cV,\nabla)$ and $\DR^\rmod(\wt\cV,\nabla)$.

\begin{remark}\label{rem:algDR}
By the Dolbeault lemmas (\cf Appendix \ref{app:currents}), the complex~$\DR^\rrd(\wt\cV,\nabla)$ is quasi\nobreakdash-iso\-mor\-phic to its $C^\infty$ analogue, and the same property holds for the complex $\DR^\rmod(\wt\cV,\nabla)$. Hence there is no risk of confusion in the notation. On the other hand, moderate and rapid decay de~Rham hypercohomologies of $\wt\cV$ on $\wt X$ can be identified with analytic de~Rham hypercohomologies on $X$ of~$\cV$ and of the dual $\cD_X$-module $\cV(!\Div)$ respectively (\cf\cite[Chap.\,5]{Mochizuki10}). Under these identifications, the pairing $\wt\DRpairing^r_{\rrd,\rmod}$ defined in Convention \ref{conv:pairing} below induces a pairing $\DRpairing^r$ between $\coH^r(X,\DR(\cV(!\Div))$ and $\coH^{2n-r}(X,\DR(\cV))$. Furthermore, if $X$ is projective, $(\cV,\nabla,\langle\cbbullet,\cbbullet\rangle)$ is the analytification of an $\cO_{X^\alg}(*\Div^\alg)$-module with connection $(\cV^\alg,\nabla^\alg)$ equipped with a non-degenerate pairing $\langle\cbbullet,\cbbullet\rangle$ (this follows from GAGA and the definition of a connection as in \cite[\S I.2.3]{Deligne70}). Setting $U=X^\alg\moins\nobreak \Div^\alg$ and $V=\cV^\alg|_U$, then in each degree $r$ the pairing $\DRpairing^r$ is the analytification of a pairing (for which we use Zariski topology)
\[
\DRpairing^r:\coH_{\dR,\rc}^r(U,V)\otimes_\CC\coH_{\dR}^{2n-r}(U,V)\to\CC,
\]
which is thus non-degenerate. In Section \ref{sec:dimone}, where $X$ is algebraic of dimension one, we provide an algebraic computation of~$\DRpairing^1$.
\end{remark}

\subsubsection*{The goodness property}
Let $x$ be a point in $\Div$, and let $(x_1,\dots,x_n)$ be local complex coordinates of~$X$ centered at $x$ such that the germ $\Div_x$ is given by $x_1\cdots x_\ell=0$ ($\ell\leq n$). Let $\cV_{\wh x}$ be the formal germ of $\cV$ at $x$, which is an $\cO_{\wh x}(*\Div)$-module ($\cO_{\wh x}(*\Div)=\CC\lcr x_1,\dots,x_n\rcr[1/x_1\cdots x_\ell]$) of finite type with the integrable connection induced by $\nabla$. We say that $(\cV,\nabla)$ is \emph{good} at~$x$~if
\begin{enumeratea}
\item
there exists a finite ramification
\[
\rho_x:(x'_1,\dots,x'_\ell,x_{\ell+1},\dots,x_n)\mto(x_1=x^{\prime p}_1,\dots,x_\ell=x^{\prime p}_\ell,x_{\ell+1},\dots,x_n)
\]
such that the pullback $\rho_x^*(\cV_{\wh x},\nabla)$ decomposes as the direct sum of free $\cO_{\wh x'}(*\Div')$-modules of finite rank with an integrable connection of the form $\nabla=\rd\varphi+\nabla^\reg$, where $\nabla^\reg$ has a regular singularities along $\Div'=\rho_x^{-1}\Div$ and $\varphi$ belongs to a finite family $\Phi_x\subset\cO_{X',x}(*\Div')/\cO_{X',x}$,
\item
the finite family $\Phi_x$ is \emph{good}, in the sense that, for any $\varphi\neq\psi$ in $\Phi_x\cup\{0\}$, the divisor of the meromorphic function $\varphi-\psi$ is negative.\footnote{Goodness usually involves $\Phi_x$ rather than $\Phi_x\cup\{0\}$. This stronger condition is needed for Assumptions \ref{hyp:hyp} to be satisfied.}
\end{enumeratea}
We say that $(\cV,\nabla)$ is \emph{good} if it is good at any $x\in X$. In particular, $\cV$ is locally $\cO_X(*\Div)$-free. Two important results should be emphasized.
\begin{enumerate}
\item\label{enum:dimone}
If $n=\dim X=1$, then any $(\cV,\nabla)$ is good. This is known as the Levelt-Turrittin theorem, which has various different proofs (\cf\eg\cite{Wasow65,Levelt75,Robba80,B-V83,D-M-R07,Malgrange91,S-vdP01,A-B01b}).
\item\label{enum:dimn}
If $n\geq2$, goodness may fail, but there exists a finite sequence of complex blow-ups over a sufficiently small analytic neighbourhood of $x$, or over~$X$ if~$X$ is quasi-projective, that are isomorphisms away from $\Div$ and such that the pullback of $(\cV,\nabla)$ is good along the pullback of $\Div$, which can be made with normal crossings. This is the Kedlaya\nobreakdash-Mochizuki theorem (\cf \cite{Bibi97} for $n=2$ and rank $\leq5$, \cite{Mochizuki07b} for $n=2$ in the projective case, \cite{Kedlaya09} for $n=2$ in the local formal and analytic cases, \cite{Mochizuki08} for $n\geq2$ in the projective case, \cite{Kedlaya10} for $n\geq2$ in the local formal and analytic cases, and also in the projective case).
\end{enumerate}

\begin{thm}\label{th:good}
If $(\cV,\nabla)$ is good, then Assumptions \ref{hyp:hyp} and \ref{ass:simplicial} hold for the rapid decay and moderate growth de~Rham complexes of $(\wt\cV,\nabla)$.
\end{thm}

\begin{convention}\label{conv:pairing}
In the case of complex manifolds, in order to make formulas independent of the choice of a square root $\sfi$ of $-1$,
we replace integration $\int_X$ of a form of maximal degree with the trace $\tr_X=(1/2\pii)^n\cdot\int_X$. The global de~Rham pairing $\Qpairing^r_\rmid$ in degree $r$ on $\wt X$ will thus be replaced with the normalized de~Rham pairing, \eg
\[
\wt\DRpairing^r_{\rrd,\rmod}=(1/2\pii)^n\Qpairing^r_{\rrd,\rmod},
\]
which pairs $\coH^r(\wt X,\DR^\rrd(\wt\cV))$ and $\coH^{2n-r}(\wt X,\DR^{\rmod}(\wt\cV))$ and is thus perfect. We~define its ``middle'' version $\wt\DRpairing^r_\rmid$ as in Section \ref{subsec:middlequadrel}, that we can interpret in terms of hypercohomologies computed on $X$ as in Remark \ref{rem:algDR}, in which case we denote it by $\DRpairing^r_\rmid$.
\end{convention}

Let us for example consider the quadratic relations in the middle dimension~$n$ if $\gen{\cbbullet,\cbbullet}$ is $\pm$-symmetric.

\begin{cor}\label{cor:good}
Assume that $X$ is compact of dimension $n$. The quadratic relations for the matrices of the Betti period pairings
\[
\pm(-2\pii)^n\,\Bpairing^\rmid_n=\Ppairing^{\rmid}_n\cdot(\DRpairing_\rmid^n)^{-1}\cdot{}^t\Ppairing^{\rmid}_n
\]
hold for a good meromorphic flat bundle on $(X,D)$, and the Betti intersection pairing can be computed with a suitable simplicial decomposition of $(\wt X,\wt D)$ as in Proposition \ref{prop:computationBetti}.
\end{cor}

\begin{proof}[Proof of Theorem \ref{th:good}]
Let us start with Assumption \ref{hyp:hyp}\eqref{hyp:1}. In dimension one, the rapid decay case is the Hukuhara-Turrittin asymptotic expansion theorem (\cf\cite{Wasow65} and \cite[App.\,1]{Malgrange91}). In~dimen\-sion $\geq2$, the proof is essentially a consequence of the work of Majima \cite{Majima84} on asymp\-totic analysis (\cf also \cite[App.]{Bibi93}, \cite[Chap.\,20]{Mochizuki08}). The rapid decay case is proved in \cite{Bibi93} and the moderate case in \cite[App.]{Hien07} with the assumption $n\leq2$, but the argument extends to any $n\geq2$ (\cf\cite[Prop.\,1]{Hien09}); both cases are also proved in \cite[Prop.\,5.1.3]{Mochizuki10}. The proof that Assumption \ref{hyp:hyp}\eqref{hyp:2} holds is obtained by means of the local description below (\cf \eg the proof in dimension one of Proposition~\ref{prop:tildeduality}).

Due to the aforementioned theorems, the local structure of $\cH^0\DR^\rrd(\wt\cV)=\cV^\rrd$ and of $\cH^0\DR^\rmod(\wt\cV)=\cV^\rmod$ is easy to understand. Let $\Delta(\varepsilon)$ denote a disc of radius $\varepsilon$. Near $x\in \Div$ as above, the real blow-up $\wt X$ is diffeomorphic to the product $(S^1)^\ell\times[0,\varepsilon)^\ell\times\Delta(\varepsilon)^{n-\ell}$ for some $\varepsilon>0$. In this model, denoting by $(r_1,\dots,r_\ell)$ the coordinates on $[0,\varepsilon)^\ell$, $\wt\Div$ is described as~$r_1\cdots r_\ell=0$, giving it the structure of a real analytic set. The structure of manifold with corners is then clear.

From the set~$\Phi_x$ of exponential factors defined after some ramification
\[
\rho_x:(x',x'')=(x'_1,\dots,x'_\ell,x_{\ell+1},\dots,x_n)\mto(x^{\prime a_1}_1,\dots,x^{\prime a_\ell}_\ell,x_{\ell+1},\dots,x_n),
\]
one constructs a local real semi-analytic stratification of $\wt\Div$: any meromorphic function $\varphi\neq0$ in $\Phi_x$ can be written in the ramified coordinates $(x',x'')$
as
$\varphi=u_\varphi\prod_{i=1}^\ell x_i^{\prime -q_i(\varphi)}(1+h(x',x''))$,
with $q_i(\varphi)$ non-negative integers,
$u_\varphi=u_\varphi(x'')$ holomorphic and invertible, and $h(x',x'')$ holomorphic.
Then we can write $u_\varphi(x'')=r_\varphi(x'')e^{\sfi a_\varphi(x'')}$, where $r_\varphi,a_\varphi$ are real analytic functions of their arguments, and
\[
\varphi\circ\rho_x=r_\varphi\prod_{i=1}^\ell r_i^{\prime -q_i(\varphi)}\cdot
\exp\bigl[\sfi(a_\varphi-\textstyle\sum_iq_i(\varphi)\theta'_i)\bigr]\cdot(1+h(r'e^{\sfi\theta'},x'')).
\]
We consider the sets defined by
\[
\cos\bigl(a_\varphi-\textstyle\sum_iq_i(\varphi)\theta'_i\bigr)=0,
\]
and their projections $S_\varphi$ to $\wt\Div$, which are real semi-analytic subsets of $\wt\Div$ (the multi-dimensional Stokes directions); the desired stratification is any real semi-analytic stratification compatible with all these subsets. We also denote by $S_\varphi^{\leq0}$ (\resp $S_\varphi^{>0}$) the corresponding subsets defined with $\leq0$ (\resp $>0$).

In the neighbourhood of any $\wt x\in\varpi^{-1}(x)$, $\cV^\rrd|_{\wt\Div}$ decomposes as $\bigoplus_\varphi\cV^\rrd_\varphi$, where $\cV^\rrd_\varphi{}$ is zero if $\varphi=0$ and otherwise constant on $S_\varphi^{>0}$ and zero on $S_\varphi^{\leq0}$. A~similar property holds for~$\cV^\rmod|_{\wt\Div}$, with the only change that $\cV^\rmod_\varphi$ is constant if $\varphi=0$. In particular, $\cV^\rrd$ and $\cV^\rmod$ are locally~$\RR$\nobreakdash-constructible in the sense of \cite[Def.\,8.4.3]{K-S90} (we can use a local stratification as above to check local $\RR$\nobreakdash-constructibility).
In this local setting, $\cV^\rrd$ and $\cV^\rmod$, together with any simplicial structure compatible with any real semi-analytic stratification with respect to which they are $\RR$-constructible,\footnote{It is standard that such a simplicial structure exists, \cf\eg\cite[Prop.\,8.2.5]{K-S90}.} satisfy thus Assumption \ref{ass:simplicial}.

In order to conclude that Assumption \ref{ass:simplicial} holds in the global setting, we note that the pair $(\wt X,\wt D)$ comes equipped with a real semi-analytic structure which induces the previous one in each local chart adapted to $\Div$, and that $\cV^\rrd$ and $\cV^\rmod$ are $\RR$-constructible, since $\RR$\nobreakdash-constructibility is a local property on $\wt\Div$, according to \cite[Th.\,8.4.2]{K-S90}. Then the same conclusion as in the local setting holds.
\end{proof}

\section{Algebraic computation of de Rham duality in dimension one}\label{sec:dimone}

In this section, we restrict to the case of complex dimension one and we make explicit in algebraic terms the results obtained in Section \ref{subsec:comparisonperiods}.

\subsection{Setting, notation, and objectives}\label{subsec:settings-one}
Let $X$ be a connected smooth complex projective curve, let $\Div$ be a \emph{non-empty} finite set of points in~$X$,
and let us define \hbox{$j:U=X\moins\Div\hto X$} as the complementary inclusion,
so that $U$ is affine. Depending on the context, we work in the Zariski topology or the analytic topology on $U$ and $X$, and hence $\cO_U$ will have the corresponding meaning. We denote by $\cO_{\wh\Div}$ the structure sheaf of the formal neighbourhood of~$\Div$ in~$X$, so~that $\cO_{\wh\Div}=\bigoplus_{x\in\Div}\cO_{\wh x}$ and $\cO_{\wh x}\simeq\CC\lcr z\rcr$ for some local coordinate $z$ centered at $x$. We~fix an affine neighbourhood $U_\Div$ of~$\Div$ in $X$. We set $U_\Div^\circ=U_\Div\moins\Div=U\cap U_\Div$ and we denote by~$\iota_\Div$ the restriction functor $\mathsf{Mod}(\cO(U))\to\mathsf{Mod}(\cO(U_\Div^\circ))$ attached to the inclusion $U_\Div^\circ\hto U$. We~set~$U_{\wh\Div}$ to be the formal completion of $U_\Div$ at $D$ and we define $U_{\wh\Div}^\circ$ by $\cO(U_{\wh\Div}^\circ)=\cO(U_D^\circ)\otimes_{\cO(U_D)}\cO(U_{\wh\Div})$. We~similarly denote by $\iota_{\wh\Div}$ the restriction functor $\mathsf{Mod}(\cO(U))\to\mathsf{Mod}(\cO(U_{\wh\Div}^\circ))$.

We consider a locally free $\cO_U$-module $V$ endowed with a connection~$\nabla$ together with the associated $\cO_X(*\Div)$-module with connection $j_*V$ (one also finds the notation $j_+V$ in the literature) that we denote by $\cV$.
It is thus a left module over the sheaf $\cD_X$ of algebraic differential operators on $X$, and thereby endowed with an action of meromorphic vector fields $\Theta_X(*\Div)$ with poles along $\Div$.

On the other hand,
let $(V^\vee,\nabla)$ be the dual bundle with connection and set $\cV^\vee=j_*V^\vee$. The dual $\cD_X$-module $\bD(\cV^\vee)$ is a holonomic $\cD_X$-module that we denote by $\cV(!\Div)$ (one also finds the notation $j_!V$ or $j_\dag V$ in the literature).
There exists a natural morphism $\cV(!\Div)\to \cV$
whose kernel~$K$ and cokernel $C$ are supported on~$\Div$ and whose image is denoted by $\cV(!{*}\Div)$.

We deduce natural morphisms between the associated algebraic de~Rham complexes\enlargethispage{2\baselineskip}%
\[
\DR \cV(!\Div)\to\DR \cV(!{*}\Div)\quad\text{and}\quad\DR \cV(!{*}\Div)\to\DR \cV.
\]
We denote the hypercohomologies on $X$ of $\DR \cV(!\Div)$, $\DR \cV(!{*}\Div)$, and $\DR \cV$ respectively by $\coH^*_{\dR,\rc}(U,V)$, $\coH^*_{\dR,\rmid}(U,V)$, and $\coH^*_{\dR}(U,V)$.

We note that $\DR K$ and $\DR C$ have non-zero cohomology in degree one at most and are supported on $\Div$, whence exact sequences
\[
\begin{gathered}
\bH^1(\Div,\DR K)\to \coH^1_{\dR,\rc}(U,V)\to \coH^1_{\dR,\rmid}(U,V)\to0\\
0\to \coH^1_{\dR,\rmid}(U,V)\to \coH^1_{\dR}(U,V)\to\bH^1(\Div,\DR C).
\end{gathered}
\]
In particular, we obtain an identification
\begin{equation}\label{eq:Hmid}
\coH^1_{\dR,\rmid}(U,V)=\image[\coH^1_{\dR,\rc}(U,V)\rightarrow \coH^1_{\dR}(U,V)].
\end{equation}

Moreover, when working in the analytic topology, one can define the analytic de~Rham complex $\DR^\an\cV$, which has constructible cohomology. Since $X$ is compact, it can be deduced from the GAGA theorem that the natural morphism $\DR\cV\to\DR^\an\cV$ induces an isomorphism in hypercohomology (and similarly for $\cV(!\Div)$). This enables us to identify the algebraic de Rham cohomology with the analytic one. We will not use the exponent `$\an$' when the context is clear and we will use this GAGA theorem without mentioning it explicitly.

If $(V,\nabla)=(\cO_U,\rd)$, one knows various forms of $\DR (\cO_X(!\Div))$. Namely, $\DR (\cO_X(!\Div))$ is quasi\nobreakdash-isomorphic to one of the following complexes:
\begin{itemize}
\item
$\mathrm{Cone}\bigl(\DR\cO_X(*\Div)\to\iota_{\Div}\DR\cO_{\wh\Div}\bigr)[-1]$,
\item
$\cO_X(-\Div)\to\Omega^1_X$,
\end{itemize}
and, for the sake of computing cohomology, one can replace these complexes with their analytic counterparts. Another form proves useful:
it is obtained by means of the rapid decay de Rham complex defined on the real-oriented blow-up $\wt X$ of $X$ at each point of $\Div$ (the one-dimensional version of the rapid decay de~Rham complex considered in Section \ref{subsec:comparisonperiods}, \cf Section~\ref{subsec:sheavesrealblowup} for details).

The first goal of this section is to explain similar presentations for $\DR\cV(!\Div)$ and, correspondingly, the presentation of $\DR\cV$ in terms of the moderate de~Rham complex on~$\wt X$. Furthermore, we express the natural de~Rham pairing between these complexes in terms of the analytic de~Rham pairing $\wt\Qpairing_{\rrd,\rmod}$ on the real blow-up space $\wt X$ at $\Div$. This approach is reminiscent of that of \cite{Mochizuki10} in any dimension (\cf Cor.\,5.2.7 from \loccit\ and Remark \ref{rem:algDR}).

Theorem \ref{th:duality} gives a residue formula for this de Rham pairing~$\DRpairing$. It is a generalization to higher rank of a formula already obtained by Deligne \cite{Deligne84b} (\cf also \cite{Deligne06}) in rank one.

Once this material is settled, we translate into this setting the general results on quadratic relations obtained in Section \ref{subsec:comparisonperiods}, with the de~Rham cohomology (\resp with compact support) instead of moderate (\resp rapid decay) cohomology. Also, we focus here our attention on the middle extension (co)homology.

\subsection{\texorpdfstring{\v{C}ech}{Cech} computation of de~Rham cohomologies}
Recall that, since $U$ is affine, the de~Rham cohomology of $(V,\nabla)$ is computed as the cohomology of the complex
\[
\Gamma(U,V)\To{\nabla}\Gamma(U,\Omega^1_U\otimes V).
\]
Hence, any element of $\coH^1_{\dR}(U,V)$ is represented by an element of \hbox{$\Gamma(U,\Omega^1_U\otimes V)$} modulo the image of $\nabla$.

We will compute the de~Rham cohomology
in terms of a \v{C}ech complex relative to the covering $(U_\Div,U)$,
whose differential is denoted by $\delta$.
So for a sheaf $F$ of $\cO_X$-modules
and $(f,\varphi)\in F(U_\Div)\oplus F(U)$,
we have $\delta(f,\varphi)=\iota_\Div\varphi-f\in F(U_\Div^\circ)$. We will write $\cV_\Div=\Gamma(U_\Div,\cV)$ and
implicitly identify $\Gamma(U_\Div^\circ,\iota_\Div V)$ with $\cV_\Div$.

Replacing $X$ with the formal neighbourhood $\wh\Div$ of $\Div$ in $X$ gives rise to \hbox{$\cV_{\wh\Div}=\cO_{\wh\Div}\otimes_{\cO_X}\cV$}, which is an $\cO_{\wh\Div}(*\Div)$-module with connection, and hence endowed with an action of the formal vector fields $\Theta_{\wh\Div}$.

\begin{prop}\label{prop:Cech!}
The complex $\DR \cV(!\Div)$ is quasi-isomorphic
to the $(-1)$-shifted cone of the natural morphism $\DR\cV\to\DR\cV_{\wh\Div}$,
that is, to the simple complex associated with the double complex
\[
\xymatrix@R=.7cm{
j_*V\ar[r]^-{\iota_{\wh\Div}}\ar[d]_\nabla&\cV_{\wh\Div}\ar[d]^{\wh\nabla}\\
j_*(\Omega^1_U\otimes V)\ar[r]^-{\iota_{\wh\Div}}&\Omega^1_{\wh\Div}\otimes \cV_{\wh\Div}.
}
\]
The complex $\bR\Gamma(X,\DR \cV(!\Div))$
is quasi-isomorphic to the simple complex associated with the double complex
\[
\xymatrix@R=.7cm{
\Gamma(U,V)\ar[r]^-{\iota_{\wh\Div}}\ar[d]_\nabla&\cV_{\wh\Div}\ar[d]^{\wh\nabla}\\
\Gamma(U,\Omega^1_U\otimes V)\ar[r]^-{\iota_{\wh\Div}}&\Omega^1_{\wh\Div}\otimes \cV_{\wh\Div}.
}
\]
\end{prop}

\begin{proof}
Any holonomic $\cD_X$\nobreakdash-module~$N$ is endowed with a canonical exhaustive increasing filtration by coherent $\cD_X(\log\Div)$-submodules $N_\bbullet$ indexed by~$\ZZ$, called the \emph{Kashiwara-Malgrange filtration} (\cf\eg\cite[Chap.\,I,\,Prop.\,6.1.2]{Bibi90b}) such that
\begin{itemize}
\item
$N_{-k}=\cO(-k\Div)N_0$ for any $k\geq0$, and $N_{k+1}=N_k+\Theta_XN_k$ for any $k\geq1$,
\item
for each $k\in\ZZ$,
the eigenvalues of $\Res\nabla$ on $N_k/N_{k-1}$ at each point of $\Div$ have their real part in $[k,k+1)$.
\end{itemize}
For each $k$, the connection $\nabla$ defines a connection $\nabla:N_k\to\Omega^1_X\otimes N_{k+1}$. The above properties imply that, for $k\neq0$, the induced morphism on graded spaces $\nabla:\gr_kN\to\Omega^1_X\otimes\gr_{k+1}N$ is bijective, so in particular the natural inclusion morphism is a quasi-isomorphism:
\begin{equation}\label{eq:VN}
\{N_0\To{\nabla}\Omega^1_X\otimes N_1\}\overset\sim\hto\DR N.
\end{equation}
For example, the following also holds:
\begin{itemize}
\item
If $N_{|U}=V$, then the natural morphism $N_k\to \cV_k$ is an isomorphism for any $k\leq0$.
\item
The formation of the Kashiwara-Malgrange filtration is compatible with tensoring with~$\cO_\Div$ and $\cO_{\wh\Div}$ and, for any $k$, $N/N_k\simeq N_{\Div}/N_{\Div,k}\simeq N_{\wh\Div}/N_{\wh\Div,k}$.
\item
For $\cV=j_*V$, we have $\cV_k=\cO(k\Div)\cV_0$ for any $k\in\ZZ$.
\item
Among all $N$'s such that $N_{|U}=V$, the $\cD_X$-module $\cV(!\Div)$ is characterized by the property that $\nabla:\gr_0N\to\Omega^1_X\otimes\gr_1N$ is bijective.
\end{itemize}
By the third point, the left-hand complex in \eqref{eq:VN} for $\cV$ reads\vspace*{-3pt}
\[
\{\cV_0\To{\nabla}\Omega^1_X(\Div)\otimes \cV_0\},
\]
so~\eqref{eq:VN} amounts to the quasi-isomorphism
\begin{equation*}
\{\cV_0\To{\nabla}\Omega^1_X(\Div)\otimes \cV_0\}\overset\sim\hto\DR\cV.
\end{equation*}

On the other hand, the last point implies that the inclusion of complexes\vspace*{-3pt}
\[
\{\cV(!\Div)_{-1}\To{\nabla}\Omega^1_X\otimes \cV(!\Div)_0\}\hto\DR \cV(!\Div)
\]
is a quasi-isomorphism. The left-hand complex reads $\{\cV_{-1}\To{\nabla}\Omega^1_X\otimes \cV_0\}$ according to the second point, so we obtain a quasi-isomorphism\vspace*{-3pt}
\begin{equation*}
\{\cV_{-1}\To{\nabla}\Omega^1_X(\Div)\otimes \cV_{-1}\}\isom\DR \cV(!\Div).
\end{equation*}
The \v{C}ech resolution of the complex on the left
is the simple complex associated with the double complex\vspace*{-3pt}
\[
\xymatrix{
\cV_{\Div,-1}\oplus j_*V\ar[r]^-{\delta}\ar[d]_\nabla&\cV_{\Div}\ar[d]^{\nabla}\\
\bigl(\Omega^1_{U_\Div}(\Div)\otimes \cV_{\Div,-1}\bigr)\oplus j_*\bigl(\Omega^1_U\otimes V\bigr)\ar[r]^-{\delta}&\Omega^1_{U_\Div}(D)\otimes \cV_{\Div}.
}
\]
Since $\cV_{\Div,-1}\subset \cV_\Div$, this complex is quasi-isomorphic to
\[
\begin{array}{c}
\xymatrix@C=.6cm{
j_*V\ar[r]^-{\iota_\Div}\ar[d]_\nabla&\cV_{\Div}/\cV_{\Div,-1}\ar[d]^{\nabla}\\
j_*(\Omega^1_U\otimes V)\ar[r]^-{\iota_\Div}&\Omega^1_{U_\Div}(\Div)\otimes (\cV_{\Div}/\cV_{\Div,-1})
}
\end{array}
=
\begin{array}{c}
\xymatrix@C=.6cm{
j_*V\ar[r]^-{\iota_{\wh\Div}}\ar[d]_\nabla&\cV_{\wh\Div}/\cV_{\wh\Div,-1}\ar[d]^{\nabla}\\
j_*(\Omega^1_U\otimes V)\ar[r]^-{\iota_{\wh\Div}}&\Omega^1_{\wh\Div}(\Div)\otimes (\cV_{\wh\Div}/\cV_{\wh\Div,-1}).
}
\end{array}
\]
Reversing the reasoning, it is also isomorphic to the \v{C}ech complex
\begin{equation*}
\begin{array}{c}
\xymatrix{
\cV_{\wh\Div,-1}\oplus j_*V\ar[r]^-{\delta}\ar[d]_\nabla&\cV_{\wh\Div}\ar[d]^{\wh\nabla}\\
\bigl(\Omega^1_{\wh\Div}(\Div)\otimes \cV_{\wh\Div,-1}\bigr)\oplus j_*\bigl(\Omega^1_U\otimes V\bigr)\ar[r]^-{\delta}&\Omega^1_{\wh\Div}\otimes \cV_{\wh\Div}.
}
\end{array}
\end{equation*}
We claim that the morphism $\nabla:\cV_{\wh\Div,-1}\to\Omega^1_{\wh\Div}(\Div)\otimes \cV_{\wh\Div,-1}$ is bijective. Indeed, recall that $\cV_{\wh\Div}$ has a canonical decomposition (\cf \cite[Th.\,(2.3),\,p.\,51]{Malgrange91})
\begin{equation*}
\cV_{\wh\Div}=\cV_{\wh\Div}^\reg\oplus \cV_{\wh\Div}^\irr
\end{equation*}
into its regular and irregular components and that the Kashiwara-Malgrange filtration $\cV_{\wh\Div,\bbullet}$ decom\-poses accordingly. Moreover, the filtration $\cV_{\wh\Div,\bbullet}^\irr$ is constant and equal to $\cV_{\wh\Div}^\irr$, and we know that
\[
\wh\nabla:\cV_{\wh\Div}^\irr\to\Omega^1_{\wh\Div}\otimes \cV_{\wh\Div}^\irr
\]
is bijective (use for example ii) in \loccit\ with $M'\!=\!\cO_{\wh\Div}$ and $M''\!=\!\cV_{\wh\Div}^\irr$). On the other hand, for the regular holonomic part,
\[
\wh\nabla:\cV_{\wh\Div,-1}^\reg\to\Omega^1_{\wh\Div}(\Div)\otimes \cV_{\wh\Div,-1}^\reg
\]
is bijective: indeed, by the very definition of the Kashiwara-Malgrange filtration, this property holds true for the morphism induced by $\wh\nabla$ on $\cV_{\wh\Div,-1}/\cV_{\wh\Div,-1-k}$ for each $k\geq1$, hence on the projective limit; for $\cV_{\wh\Div,-1}^\reg$, this projective limit is isomorphic to $\cV_{\wh\Div,-1}^\reg$ itself, as follows \eg from \cite[Chap.\,I,\,Lem.\,6.2.6]{Bibi90b}.

It follows from this claim that $\DR \cV(!\Div)$ is quasi-isomorphic to the complex given in the proposition, and the statement for $\bR\Gamma(X,\DR \cV(!\Div))$ results immediately.
\end{proof}

\begin{cor}
The formal de Rham complex $\DR_{\wh D}\cV(!\Div)$ is quasi-isomorphic to zero.
\end{cor}

\begin{proof}
Indeed, $\DR_{\wh D}\cV(!\Div)$ is the simple complex
associated with the double complex
$\DR_{\wh D}\cV\To{\id}\DR_{\wh D}\cV$.
\end{proof}

\begin{cor}\label{cor:represHdRc}
The de~Rham cohomology space with compact support $\coH^1_{\dR,\rc}(U,V)$ consists of pairs
\[
(\wh m,\omega)\in \cV_{\wh\Div}\oplus\Gamma(U,\Omega^1_U\otimes V)
\]
satisfying $\wh\nabla\wh m=\iota_{\wh\Div}\omega$, modulo pairs of the form $(\iota_{\wh\Div}(v),\nabla v)$ for $v\in\Gamma(U,V)$.\qed
\end{cor}

\begin{remark}\label{rem:H1c}
The natural morphism $\coH^1_{\dR,\rc}(U,V)\to \coH^1_{\dR}(U,V)$ is described in terms of the previous representatives as $(\wh m,\omega)\mto\omega$.

\step\label{rem:H1c2}
The image $\coH^1_{\dR,\rmid}(U,V)$ of this morphism (\cf \eqref{eq:Hmid}) consists of classes of sections \hbox{$\omega\in\Gamma(U,V)$} such that $\iota_{\wh\Div}\omega\in\image(\wh\nabla:\cV_{\wh\Div}\to\cV_{\wh\Div})$. Therefore, any family consisting of $\dim \coH^1_{\dR,\rmid}(U,V)$ linearly independent classes $[\omega]$ in $\coH^1_{\dR}(U,V)$ for which there exists $\wh m\in \cV_{\wh\Div}$ satisfying $\wh\nabla\wh m=\iota_{\wh\Div}\omega$ is a basis of $\coH^1_{\dR,\rmid}(U,V)$.

\step\label{rem:H1c1}
The kernel of this morphism consists of pairs of the form $(\wh m,\nabla w)$ for some $w\in\Gamma(U,V)$, modulo pairs $(\iota_{\wh\Div}(v),\nabla v)$. Each class has thus a representative of the form $(\wh m,0)$ with~$\wh\nabla\wh m=0$, and there is a surjective morphism
\[
\ker\bigl[\wh\nabla:\cV_{\wh\Div}\ra \cV_{\wh\Div}\bigr]\to\hspace*{-6.3mm}\to\ker\Bigl[\coH^1_{\dR,\rc}(U,V)\ra \coH^1_{\dR}(U,V)\Bigr].
\]

\step\label{rem:H1c3}
If $V$ has no constant subbundle with connection (\eg if $V$ is irreducible and non-constant), then a representative $(\wh m,0)$ with $\wh\nabla\wh m=0$ is unique. We thus conclude in such a case that
\[
\ker\Bigl[\wh\nabla:\cV_{\wh\Div}\ra \cV_{\wh\Div}\Bigr]\simeq\ker\Bigl[\coH^1_{\dR,\rc}(U,V)\ra \coH^1_{\dR}(U,V)\Bigr].
\]

\step\label{rem:H1c4}
With the assumption in \eqref{rem:H1c3}, a basis of $\coH^1_{\dR,\rc}(U,V)$ can thus be obtained as the union of the following sets:
\begin{itemize}
\item
a basis of $\ker\wh\nabla$ in $\cV_{\wh\Div}$,
\item
a family of $\dim \coH^1_{\dR,\rmid}(U,V)$
representatives $(\wh m,\omega)$ for which the classes of~$\omega$ in $\coH^1_{\dR}(U,V)$ are linearly independent (and thus form a basis of $\coH^1_{\dR,\rmid}(U,V)$).
\end{itemize}
\end{remark}

\begin{remark}[$C^\infty$ computation]\label{rem:DRinfty}
We consider the sheaves
\[
\cV_\infty=\cC_X^\infty\otimes_{\cO_X}\cV\quad\text{and}\quad\cV(!\Div)_\infty=\cC_X^\infty\otimes_{\cO_X}\cV(!\Div)
\]
and the corresponding $C^\infty$ de~Rham complexes $\DR^\infty(\cV)$ and $\DR^\infty(\cV(!\Div))$ with differential $\nabla+\ov\partial$.
These complexes are fine resolutions of
$\DR^\an(\cV)$ and $\DR^\an(\cV(!\Div))$. We also consider the sheaves of $C^\infty$ functions on $U$ having respectively moderate growth and rapid decay at the points of $\Div$, defined in a way similar to that of Definition \ref{def:rdmod}.
The Dolbeault complex $(\cE_X^{\rmod,\cbbullet},\ov\partial)$ (\cf Appendix~\ref{app:currents}) is a resolution of $\cO_X(*\Div)$.
Hence, the complex
\[
\DR^{\infty,\rmod}(\cV)=(\cE_X^{\rmod,\cbbullet}\otimes\cV,\nabla+\ov\partial)
\]
is a fine resolution of $\DR^\an\cV$. Therefore, there is an isomorphism
\[
\coH^1_{\dR}(U,V)\simeq \coH^1(\Gamma( X,\DR^{\infty,\rmod}(\cV)))
\]
induced by
\begin{multline*}
\Gamma(U,\Omega^1_U\otimes V)=\Gamma(X,\Omega^1_X(*\Div)\otimes\cV)\ni\omega\\
\mto\omega_\infty=\omega\in\Gamma(X,\cE_X^{\rmod,(1,0)}\otimes\cV)\subset\Gamma(X,\cE_X^{\rmod,1}\otimes\cV).
\end{multline*}

On the other hand, $\DR^\infty(\cV(!\Div))$ is quasi-isomorphic to the double complex $\DR^\infty(\cV)\to\DR^\infty_{\wh\Div}(\cV)$. By Borel's lemma (applied to the coefficients in a basis of $\cE_X^k\otimes\cV$), this morphism is termwise surjective, hence the double complex is isomorphic to the $C^\infty$ rapid decay complex $\DR^{\infty,\rrd}(\cV)=(\cE_X^{\rrd,\cbbullet}\otimes\cV,\nabla+\ov\partial)$.
Therefore, there is an isomorphism
\begin{equation}\label{eq:represinfty}
\coH^1_{\dR,\rc}(U,V)\simeq \coH^1(\Gamma( X,\DR^{\infty,\rrd}(\cV)))
\end{equation}
described as follows. For a class $(\wh m,\omega)$ with $\nabla\wh m=\iota_{\wh\Div}\omega$, let us choose $m_\infty\in\Gamma(X,\cV_\infty)$ whose Taylor series at $\Div$ is equal to $\wh m$. As above, we also regard $\omega$ as belonging to $\Gamma(X,\cE_X^{\rmod,1}\otimes\cV)$ and we denote it by $\omega_\infty$. Then
\[
\eta_\infty=\omega_\infty-(\nabla+\ov\partial)m_\infty
\]
has rapid decay at $\Div$. Moreover, it is $(\nabla+\ov\partial)$-closed. The image of the class of $(\wh m,\omega)$ by \eqref{eq:represinfty} is the class of $\eta_\infty$.
\end{remark}

\begin{example}[The case $(V,\nabla)=(\cO_U,\rd)$]\label{ex:O!}
We are interested in $\coH^2_{\dR,\rc}(U)$. This space is identified with $\coH^2(X,\DR\cO_X(!\Div))$. By Proposition \ref{prop:Cech!} applied to \hbox{$V=\cO_U$}, we identify $\coH^2_{\dR,\rc}(U)$ with the quotient of $\Omega^1_{\wh\Div}(*\Div)$ by the subspace consisting of elements of the form $\rd\wh\varphi-\iota_{\wh\Div}\eta$, with $\wh\varphi\in\cO_{\wh\Div}(*\Div)$ and $\eta\in\Omega^1(U)$. Let us make explicit the canonical isomorphism
\[
\can:\coH^2_{\dR,\rc}(U)\isom \coH^2(\Gamma(X,\cE_X^{\rrd,\cbbullet})).
\]
For $\wh\mu\in\Omega^1_{\wh\Div}(*\Div)$, one can choose a lift $\mu_\infty\in\Gamma(X,\cE_X^{1,0}(*\Div))$. Since $\rd\wh\mu=0$, the $2$-form $\rd\mu_\infty$ has rapid decay. Then, $\can[\wh\mu]$ is the class of $-\rd\mu_\infty$.

On the one hand, let $\res_\Div:\coH^2_{\dR,\rc}(U)\to\CC$ be induced by
\[
\sum_{x\in \Div}\res_x:\Omega^1_{\wh\Div}(*\Div)\to\CC.
\]
On the other hand, we consider the trace morphism
\[
\tr_X:\frac1{2\pii}\int_X\colon\coH^2(\Gamma(X,\cE_X^{\rrd,\cbbullet}))\to\CC.
\]

The following result is standard: the morphisms $\tr_X$ and $\res_\Div$ are isomorphisms and we have
\begin{equation}\label{eq:O!}
\tr_X\circ\can=\res_\Div.
\end{equation}
To see the latter equality, we write, for a union of discs $\Delta_r$ of radius $r\!>\!0$ centered at the points of~$\Div$,
\begin{align*}
-\frac1{2\pii}\int_X\rd\mu_\infty&=-\frac1{2\pii}\int_{X\moins\Delta_r}\rd\mu_\infty-\frac1{2\pii}\int_{\Delta_r}\rd\mu_\infty\\
&=\frac1{2\pii}\int_{\partial\Delta_r}\mu_\infty-\frac1{2\pii}\int_{\Delta_r}\rd\mu_\infty.
\end{align*}
We can choose $\mu_\infty$ so that, for some $R>0$, $\mu_\infty=\wh\mu^{<0}+\nu_\infty$, where $\wh\mu^{<0}$ is the polar part of $\wh\mu$ and~$\nu_\infty$ is $C^\infty$ on $\Delta_R$.
Then, since $\de\mu_\infty$ has rapid decay, the above integral tends to $\res_D\wh\mu$
as $r\to 0$.
\end{example}

\subsection{Pairings}
Starting from the tautological pairing
\[
(\cbbullet\sep\cbbullet):V\otimes V^\vee\to\cO_U,
\]
one defines a pairing
\[
\lpr\cbbullet\sep\cbbullet\rpr:\coH^1_{\dR,\rc}(U,V)\otimes \coH^1_{\dR}(U,V^\vee)\to\coH^2_{\dR,\rc}(U)
\]
in the following way.
For representatives $(\wh m,\omega)\in\cV_{\wh\Div}\oplus\Gamma(U,\Omega^1_U\otimes V)$ and $\omega^\vee\in\Gamma(U,\Omega^1_U\otimes V^\vee)$ (\cf Corollary \ref{cor:represHdRc}), we set
\[
\lpr(\wh m,\omega)|\omega^\vee\rpr=(\wh m\sep\iota_{\wh\Div}\omega^\vee)\mod\{\rd\wh\varphi-\iota_{\wh\Div}\eta\}.
\]
Indeed, let us check that the formula does not depend on the choice of representatives of the de~Rham cohomology classes.
\begin{itemize}
\item
If $\omega^\vee=\nabla v^\vee$, we have
\[
(\wh m\sep\iota_{\wh\Div}\nabla v^\vee)=\rd(\wh m\sep\iota_{\wh\Div}v^\vee)-\iota_{\wh\Div}(\omega\sep v^\vee),
\]
which is of the form $\rd\wh\varphi-\iota_{\wh\Div}\eta$.
\item
Similarly, if $(\wh m,\omega)=(\iota_{\wh\Div}(v),\nabla v)$, then $(\iota_{\wh\Div}(v)\sep\iota_{\wh\Div}\omega^\vee)=\iota_{\wh\Div}(v\sep\omega^\vee)$.
\end{itemize}

The residue pairing $\Rpairing^1$ is defined as the composition (\cf Example \ref{ex:O!} for the isomorphism $\res_\Div$)
\[
\Rpairing^1:\coH^1_{\dR,\rc}(U,V)\otimes \coH^1_{\dR}(U,V^\vee)\To{\lpr\cbbullet\sep\cbbullet\rpr}\coH^2_{\dR,\rc}(U,\CC)\xrightarrow[\sim]{~\textstyle\res_\Div~}\CC.
\]
In other words, on representatives $(\wh m,\omega)$ and $\omega^\vee$ it is expressed by summing up the residues:
\begin{equation*}
\Rpairing^1\bigl((\wh m,\omega),\omega^\vee\bigr)=\res_\Div(\wh m\sep\iota_{\wh\Div}\omega^\vee)=\sum_{x\in\Div}\res_x(\wh m\sep\iota_{\wh x}\omega^\vee).
\end{equation*}

On the other hand, by working with the $C^\infty$ realizations of Remark \ref{rem:DRinfty}, we consider the de~Rham pairing
\begin{equation}\label{eq:Qpairing}
\Qpairing:\DR^{\infty,\rrd}(\cV)\otimes\DR^{\infty,\rmod}(\cV^\vee)\to(\cE_X^{\rrd,\cbbullet},\rd),
\end{equation}
which gives rise at the cohomology level to the pairing
\[
\DRpairing^1:\coH^1_{\dR,\rc}(U,V)\otimes \coH^1_{\dR}(U,V^\vee)\To{\Qpairing^1}\coH^2(\Gamma(X,\cE_X^{\rrd,\cbbullet}))\xrightarrow[\sim]{~\tr_{U^\an}~}\CC.
\]

\begin{prop}\label{prop:duality}
The de~Rham pairing $\DRpairing^1$ and the residue pairing $\Rpairing^1$ are equal.
\end{prop}

\begin{proof}
As explained in Remark \ref{rem:DRinfty}, the elements
\[
(\wh m,\omega)\in\cV_{\wh\Div}\oplus\Gamma(U,\Omega^1_U\otimes V)\quad\text{and}\quad\omega^\vee\in\Gamma(U,\Omega^1_U\otimes V^\vee)
\]
are sent in a natural way, by choosing a $C^\infty$ lift $m_\infty$ of~$\wh m$, respectively to
\[
(m_\infty,\omega_\infty)\in\Gamma(X,\cV_\infty)\oplus\Gamma(X,\cE_X^{\rmod
,(1,0)
}\otimes V)\quad \text{and}\quad \omega_\infty^\vee\in\Gamma(X,\cE^{\rmod
,(1,0)
}\otimes\cV^\vee),
\]
and $\omega_\infty,\omega_\infty^\vee$ are $(\nabla+\ov\partial)$-closed. Set $\eta_\infty=\omega_\infty-(\nabla+\ov\partial)m_\infty$ as before. By reasons of type we have
\[
\Qpairing(\eta_\infty,\omega_\infty^\vee)=\bigl(-(\nabla+\ov\partial)m_\infty\sep\omega_\infty^\vee\bigr)=-\rd(m_\infty\sep\omega_\infty^\vee).
\]
With the notation of Example \ref{ex:O!}, this equality reads
\[
\Qpairing(\eta_\infty,\omega_\infty^\vee)=\can(\wh m\sep\omega^\vee)=\can\lpr(\wh m,\omega)\sep\omega^\vee\rpr.
\]
The proposition now follows from \eqref{eq:O!}.
\end{proof}

By working on the real blow-up space we will also show:

\begin{thm}\label{th:duality}
The de~Rham pairing $\DRpairing^1$ is a perfect pairing.
\end{thm}

In particular, the residue pairing $\Rpairing^1$ is perfect. Before starting the proof of the theorem which will be done in Section \ref{subsec:sheavesrealblowup}, let us state its ``middle'' consequence.

\begin{cor}\label{cor:dualitymiddle}
The residue pairing $\Rpairing^1$ vanishes on $\ker\wh\nabla\otimes \coH^1_{\dR,\rmid}(U,V^\vee)$ and induces a non\nobreakdash-degenerate pairing
\[
\Rpairing^1_\rmid:\coH^1_{\dR,\rmid}(U,V)\otimes \coH^1_{\dR,\rmid}(U,V^\vee)\to\CC.
\]
Moreover, if $V\isom V^\vee$ is a $\pm$-symmetric isomorphism, it induces a $\mp$-symmetric non-degenerate pairing
\[
\coH^1_{\dR,\rmid}(U,V)\otimes \coH^1_{\dR,\rmid}(U,V)\to\CC.
\]
\end{cor}

\begin{proof}[Proof of Corollary \ref{cor:dualitymiddle}]
If $\wh\nabla\wh m\!=\!0$ and $\omega^\vee$ satisfies $\iota_{\wh\Div}\omega^\vee\!=\!\wh\nabla\wh m^\vee$ (see Remark \ref{rem:H1c}\eqref{rem:H1c2}), we~have
\[
\res_\Div(\wh m\sep\iota_{\wh\Div}\omega^\vee)=\res_\Div\rd(\wh m\sep\wh m^\vee)=0,
\]
hence the first assertion. That $\Rpairing^1_\rmid$ is non-degenerate follows from the same property for~$\Rpairing^1$. Assume now that $\langle\cbbullet,\cbbullet\rangle$ is a $\pm$-symmetric pairing on $V$ and let $\omega,\omega'$ represent classes in $\coH^1_{\dR,\rmid}(U,V)$, so that there exist $\wh m,\wh m'$ such that $\nabla\wh m=\iota_{\wh\Div}\omega$ and $\nabla\wh m'=\iota_{\wh\Div}\omega'$. Then
\begin{align*}
\res_\Div\langle\wh m,\iota_{\wh\Div}\omega'\rangle&=\res_\Div\langle\wh m,\nabla\wh m'\rangle=-\res_\Div\langle\nabla\wh m,\wh m'\rangle\\
&=\mp\res_\Div\langle\wh m',\nabla\wh m\rangle=\mp\res_\Div\langle\wh m',\iota_{\wh\Div}\omega\rangle.\qedhere
\end{align*}
\end{proof}

\begin{remark}
Theorem \ref{th:duality} is of course a variant of the theorem asserting compatibility between duality of $\cD$-modules and proper push-forward (\cf\cite{Mebkhout87}), but more in the spirit of \cite[Cor.\,5.2.7]{Mochizuki10}. The presentation given here owes much to \cite[App.\,2]{Malgrange91}. Nevertheless, the present formulation is more precise and can lead to explicit computations.

In the present form, this theorem has a long history, starting (as far as we know) with~\cite{Deligne84b} (\cf also \cite{Deligne06}) in rank one. The computation that we perform with \v{C}ech complexes by using a formal neighbourhood of $\Div$ is inspired from \loccit\ The result was used by Deligne for showing compatibility between duality and the irregular Hodge filtration introduced in \loccit

In a series of papers mentioned in the introduction, Matsumoto et~al.\ also gave duality results of this kind, either in the case of regular singularities or in rank one cases with irregular singularities (possibly in higher dimensions). An instance of Proposition \ref{prop:duality} was already obtained in \cite{Matsumoto98}.
\end{remark}

\subsection{Computations on the real blow-up space}\label{subsec:sheavesrealblowup}
In this section, we revisit the results of Section~\ref{subsec:comparisonperiods} in the simpler case where $\dim X=1$ and we give details on the proof that Assumption~\ref{hyp:hyp}\eqref{hyp:2} holds for $\wt\cV$ with respect to the tautological pairing $(\cbbullet\sep\cbbullet)$. \emph{In this section, we use the analytic topology on $U$ and~$X$}.

Let $\varpi:\wt X\to X$ denote the real oriented blow-up of $X$ at each point of~$\Div$. Recall that $\wt X$ is endowed with holomorphic and $C^\infty$ sheaves of functions with growth conditions. Working on~$\wt X$ makes proofs of local duality easier since all involved de~Rham complexes are sheaves, \ie have cohomology in degree zero only. We denote by $\wtj:U\!\hto\!\wt X$ the open inclusion and by \hbox{$\wti:\varpi^{-1}(\Div)\!\hto\!\wt X$} the complementary inclusion.

For $\cV$ as in Section \ref{subsec:settings-one}, we~set (\cf Section~\ref{subsec:comparisonperiods} for $\cA_{\wt X}$)
\[
\wt\cV=\cA_{\wt X}\otimes_{\varpi^{-1}\cO_X}\varpi^{-1}\cV=\cA_{\wt X}(*\Div)\otimes_{\varpi^{-1}\cO_X}\varpi^{-1}\cV
\]
(the second equality holds since $\cV$ is an $\cO_X(*\Div)$-module), which is a locally free $\cA_{\wt X}(*\Div)$-module of finite rank. The $C^\infty$ de~Rham complexes with growth conditions $\DR^{\infty,\rrd}(\wt\cV)$ and $\DR^{\infty,\rmod}(\wt\cV)$ have been defined in Section \ref{subsec:comparisonperiods} and satisfy Assumptions \ref{hyp:hyp}\eqref{hyp:1} (with respect to the tautological pairing $(\cbbullet\sep\cbbullet)$) and \ref{ass:simplicial}, according to Theorem \ref{th:good}. We will check in Proposition~\ref{prop:tildeduality}, as asserted in Theorem \ref{th:good}, that they also satisfy Assumption \ref{hyp:hyp}\eqref{hyp:2} with respect to the tautological pairing $(\cbbullet\sep\cbbullet)$ on $V\otimes V^\vee$ (as already mentioned, the results of Section \ref{subsec:dRquadratic} can easily be adapted to this setting).

The pushforwards by $\varpi$ of these $C^\infty$ de~Rham complexes are isomorphic to the corresponding complexes on~$X$, allowing for a computation on~$\wt X$ of the various algebraic de~Rham cohomologies, according to Remark \ref{rem:DRinfty}.

\begin{cor}
We have natural isomorphisms
\[
\coH^k_{\dR}(U,V)\simeq \coH^k(\wt X,\DR^{\infty,\rmod}(\wt\cV)),\quad \coH^k_{\dR,\rc}(U,V)\simeq \coH^k(\wt X,\DR^{\infty,\rrd}(\wt\cV)).\eqno\qed
\]
\end{cor}

Recall that we set $\wt\cV^\rrd=\cH^0\DR^{\infty,\rrd}(\wt\cV)$ and similarly with $\rmod$.

\begin{prop}\label{prop:tildeduality}
The natural pairing
\[
\wt\Qpairing_{\rrd,\rmod}^\infty:\DR^{\infty,\rrd}(\wt \cV)\otimes\DR^{\infty,\rmod}(\wt\cV^\vee)\to(\cE_{\wt X}^{\rrd,\cbbullet},\rd)\simeq\wtj_!\CC_U
\]
is perfect, that is, the pairing induced on the $\cH^0$ is perfect:
\[
\wt\cV^\rrd\otimes\wt\cV^{\vee\rmod}\to\wtj_!\CC_U.
\]
\end{prop}

\begin{proof}
As the statement is local on $\wt X$, we work on an analytic neighbourhood~$\Delta$ of $x\in\Div$ with coordinate $z$. The Levelt-Turrittin decomposition (\cf \eqref{enum:dimone} in Section~\ref{subsec:comparisonperiods}) can be lifted locally on $S^1\times\{0\}$ with coefficients in $\cA_{\wt \Delta}$, so that the ramification~$\rho$ can be neglected by considering a local determination of $z^{1/r}$ and $\log z$, where~$r$ is the order of ramification. We are in this way reduced to proving the assertions in the case of an elementary model entering the Levelt\nobreakdash-Turrittin decomposition.

Let $\theta_o\in S^1$ and let $\nb(\theta_o)=(\theta_o-\delta,\theta_o+\delta)\times[0,\varepsilon)$ be an open neighbourhood of $\theta_o$ in $\wt X$. We set $\nb(\theta_o)^\circ=\nb(\theta_o)\cap U=(\theta_o-\delta,\theta_o+\delta)\times(0,\varepsilon)$ and consider the decompositions of $\wtj$ as the composition of the open inclusions
\[
\nb(\theta_o)^\circ\Hto{\alpha^+}\nb(\theta_o)^\circ\cup\bigl((\theta_o,\theta_o+\delta)\times\{0\}\bigr)
\Hto{\beta^+}\nb(\theta_o)=\nb(\theta_o)^\circ\cup\bigl((\theta_o-\delta,\theta_o+\delta)\times\{0\}\bigr)
\]
and
\[
\nb(\theta_o)^\circ\Hto{\alpha^-}\nb(\theta_o)^\circ\cup\bigl((\theta_o-\delta,\theta_o)\times\{0\}\bigr)\Hto{\beta^-}\nb(\theta_o).
\]
For an elementary model, three kinds of cases can occur for $\wt\cV^\rrd$ and $\wt\cV^{\vee\rmod}$ on such a sufficiently small neighbourhood:
\begin{itemize}
\item
$\wt\cV^\rrd=\wtj_*V^\nabla$ and $\wt\cV^{\vee\rmod}=\wtj_!V^{\vee\nabla}$,
\item
$\wt\cV^\rrd=\wtj_!V^\nabla$ and $\wt\cV^{\vee\rmod}=\wtj_*V^{\vee\nabla}$,
\item
$\wt\cV^\rrd=\beta^+_!\alpha^+_*V^\nabla$ and $\wt\cV^{\vee\rmod}=\beta^-_!\alpha^-_*V^{\vee\nabla}$.
\end{itemize}
Perfectness being clear in the first two cases, let us analyze the third one. Poincaré-Verdier duality shows perfectness of the natural pairing
\[
(\beta^-_*\alpha^-_!V^\nabla)\otimes(\beta^-_!\alpha^-_*V^{\vee\nabla})\to\wtj_!\CC_{\nb(\theta_o)^\circ},
\]
so one only needs to show the equality of subsheaves of $\wtj_*V^\nabla$:
\[
\beta^-_*\alpha^-_!V^\nabla=\beta^+_!\alpha^+_*V^\nabla.
\]
These subsheaves clearly coincide away from $\theta_o$, and both are zero at $\theta_o$, hence they coincide.
\end{proof}

\begin{proof}[Proof of Theorem \ref{th:duality}]
By applying $\bR\varpi_*$ (which is $\varpi_*$ here), the pairing $\wt\Qpairing_{\rrd,\rmod}^\infty$ induces the pairing~$\Qpairing$ of \eqref{eq:Qpairing}. From the compatibility between Verdier duality and proper pushforward by~$\varpi$, together with Proposition \ref{prop:accperfect}, we deduce that $\DRpairing^1$ is non-degenerate.
\end{proof}

\subsection{Computation of Betti period pairings}
The formula for computing the period pairing
\[
\Ppairing^{\rrd,\rmod}_1\colon
\coH_1^{\rrd}(\wt X,\wt\cV)\otimes \coH^1_{\dR,\rmod}(\wt X,\wt\cV^\vee)\to\CC
\]
(\cf Section~\ref{subsec:Bettiquadratic}) is the natural one described as follows. Let $\sigma\in \coH_1^{\rrd}(\wt X,\wt\cV)$
be represented by a finite sum $\sum_\ell c_\ell\otimes v_\ell$ of twisted cycles where each $c_\ell\colon [0,1]\to\wt X$
is a piecewise smooth simplex satisfying $c_\ell((0,1))\subset U$
and $v_\ell$ is a horizontal section of $\cA_{\wt X}^\rrd\otimes\wt\cV$
on a neighbourhood of the support of $c_\ell$.
(The boundary $c_\ell(o)$, $o\in\{0,1\}$, may lie in $\wt\Div$.) Let $\omega^\vee\in\Gamma(U,\Omega^1_U\otimes V^\vee)$ be a representative of a de~Rham class in $\coH^1_{\dR,\rmod}(\wt X,\wt\cV^\vee)=\coH^1_{\dR}(U,V^\vee)$.
Then
\[
\Ppairing^{\rrd,\rmod}_1(\sigma,\omega^\vee)=
\sum_\ell \int_0^1\hspace*{-2mm}c_\ell^*(( v\sep\omega^\vee)).
\]

We will make explicit the formula for computing the period pairing
\[
\Ppairing^{\rmod,\rrd}_1:\coH_1^{\rmod}(\wt X,\wt\cV)\otimes \coH^1_{\dR,\rrd}(\wt X,\wt\cV^\vee)\to\CC
\]
starting from a representative $(\wh m,\omega)$ of a de~Rham class in $\coH^1_{\dR,\rrd}(\wt X,\wt\cV^\vee)=\coH^1_{\dR,\rc}(U,V^\vee)$. With respect to the above formula, there is a supplementary regularization procedure to be performed.

\begin{prop}\label{prop:calculP}\mbox{}
\begin{itemize}
\item
Let $\sigma\in \coH_1^{\rmod}(\wt X,\wt\cV)$
be represented by a finite sum $\sum_\ell c_\ell\otimes v_\ell$ of twisted cycles
where each $c_\ell\colon [0,1]\to\wt X$
is a piecewise smooth simplex satisfying $c_\ell((0,1))\subset U$
and $v_\ell$ is a horizontal section of $\cA_{\wt X}^\rmod\otimes\wt\cV$
on a neighbourhood of the support of $c_\ell$.
\item
Let
$(\wh m,\omega)\in\cV_{\wh\Div}^\vee\oplus\Gamma(U,\Omega^1_U\otimes V^\vee)$
be a representative of $[(\wh m,\omega)]\in \coH^1_{\dR,\rc}(U,V^\vee)$.
\item
For $o\in\{0,1\}$,
if $c_\ell(o) \in \wt{D}$,
let $\wt m_{\ell,o}$ be a germ in $\wt\cV^\vee_{c_\ell(o)}$
having $\wh m_{\ell,o}$ as asymptotic expansion at $c_\ell(o)$.
Otherwise let $\wt{m}_{\ell,o}=0$.
\end{itemize}
Then the following equality holds:
\[
\Ppairing^{\rmod,\rrd}_1(\sigma,[(\wh m,\omega)])
=\sum_\ell \lim_{\varepsilon\to0} \biggl[\biggl(\int_\varepsilon^{1-\varepsilon}\hspace*{-2mm}c_\ell^*(( v\sep\omega))\biggr)+(v\sep\wt m_{\ell,0})(c_\ell(\varepsilon))-(v\sep\wt m_{\ell,1})(c_\ell(1-\varepsilon)) \biggr].
\]
\end{prop}

\begin{remark}
In practice, one only needs to know the first few terms of the asymptotic expansion of~$\wt m_{\ell,o}$ in order to compute the limit.
The knowledge of $\wh m_{\ell,o}$ up to a finite order is therefore enough.
\end{remark}

\begin{proof}
We start with the $C^\infty$ setting.

\begin{lemma}\label{lem:calculPCinfty}
Let $\sigma$ be as in the proposition and let $\eta_\infty=\omega_\infty-(\nabla+\ov\partial)m_\infty$ be a representative of a class in~$\coH^1_{\dR,\rc}(U,V^\vee)$ as in Remark \ref{rem:DRinfty}. Then
\[
\Ppairing^{\rmod,\rrd}_1(\sigma,[\eta_\infty])
=\sum_\ell \lim_{\varepsilon\to0} \biggl[\biggl(\int_\varepsilon^{1-\varepsilon}\hspace*{-2mm}c_\ell^*((v\sep\omega_\infty))\biggr)+(v\sep m_\infty)(c_\ell(\varepsilon))-(v\sep m_\infty)(c_\ell(1-\varepsilon)) \biggr].
\]
\end{lemma}

\begin{proof}
Since $\eta_\infty$ has rapid decay, we have
\[
\Ppairing^{\rmod,\rrd}_1(\sigma,[\eta_\infty])
=\sum_\ell \int_0^1\hspace*{-2mm}c_\ell^*((v\sep\eta_\infty))
=\sum_\ell\lim_{\varepsilon\to0}\int_\varepsilon^{1-\varepsilon}\hspace*{-2mm}c_\ell^*((v\sep\eta_\infty)),
\]
and we write
\begin{align*}
\int_\varepsilon^{1-\varepsilon}\hspace*{-2mm}c_\ell^*((v\sep-(\nabla+\ov\partial)m_\infty))
&=-\int_\varepsilon^{1-\varepsilon}\hspace*{-2mm}\rd c_\ell^*((v\sep m_\infty))\\
&=(v\sep m_\infty)(c_\ell(\varepsilon))-(v\sep m_\infty)(c_\ell(1-\varepsilon)).\qedhere
\end{align*}
\end{proof}

\subsubsection*{End of the proof of Proposition \ref{prop:calculP}}
Let $\wt m_{\ell,o}$ be as in the proposition.
Then \hbox{$\wt m_{\ell,o}-m_\infty$} has rapid decay
at $c_\ell(0)$ and $c_\ell(1)$ along the path $c_\ell$,
so that we can replace $m_\infty$ with $\wt m_{\ell,o}$
in the formula of Lemma \ref{lem:calculPCinfty} without changing the limit.
\end{proof}

\subsection{Quadratic relations in dimension one}
We summarize the consequences of the identifications previously obtained in this section to the form of quadratic relations in dimension one. The setting is as in Section \ref{subsec:settings-one}. We assume that $(V,\nabla)$ is endowed with a non-degenerate pairing $\langle\cbbullet,\cbbullet\rangle:(V,\nabla)\otimes(V,\nabla)\to(\cO_U,\rd)$. We~assume that it is symmetric or skew-symmetric, that we denote by $\pm$-symmetric. We will make explicit the way to express middle quadratic relations (Corollaries \ref{cor:good}, \ref{cor:midquadraticreldR}, and Remark \ref{rem:symmetrymid}) in the present setting.

\subsubsection*{Middle de~Rham pairing}
Let $d$ be the dimension of
\[
\coH^1_{\dR,\rmid}(U,V)=\image\bigl[\coH^1_{\dR,\rc}(U,V)\!\to\!\coH^1_{\dR}(U,V)\bigr].
\]
\begin{itemize}
\item
We choose $d$ elements $\omega_1,\dots,\omega_d$ in $\coH^1_{\dR}(U,V)$ and we solve the equation $\nabla\wh m_j=\iota_{\wh D}\omega_j$ for each~$j$ and at each point of $\Div$, with $\wh m_j\in V_{\wh\Div}$. We~choose such solutions~$\wh m_j$.
\item
The matrix $\DRpairing^1_\rmid$ of the de~Rham pairing with respect to these families has $(i,j)$-entries $\sum_{x\in\Div}\res_x\langle\wh m_i,\iota_{\wh\Div}\omega_j\rangle$. It is $\mp$-symmetric.
\item
If $\det\DRpairing^1_\rmid\neq0$, then the family $(\omega_j)$ is a basis of $\coH^1_{\dR,\rmid}(U,V)$.
\end{itemize}

\subsubsection*{Middle Betti pairing}
The space $\coH_1^\rmid(U,V)=\image\bigl[\coH_1^\rrd(U,V)\to\coH_1^\rmod(U,V)\bigr]$ has also dimension~$d$.
\begin{itemize}
\item
We choose $d$ elements $\beta_1,\dots,\beta_d$ in $\coH_1^\rmod(U,V)$ which are the images of elements $\alpha_1,\dots,\alpha_d$ in $\coH_1^\rrd(U,V)$.
\item
The matrix $\Bpairing_1^\rmid$ of the Betti pairing with respect to these families is computed for example by means of Proposition \ref{prop:computationBetti}. It is $\mp$-symmetric.
\item
If $\det\Bpairing_1^\rmid\neq0$, then the family $(\beta_i)$ is a basis of $\coH_1^\rmid(U,V)$.
\end{itemize}

\subsubsection*{Middle period pairing}
Let us fix a triangulation of $(X,\Div)$ that is induced by a triangulation of~$\wt X$ such that the $\RR$-constructible sheaves $\cH^0\DR^\rrd\cV$ and $\cH^0\DR^\rmod\cV$ are constant on each open $1$-simplex (\cf Theorem \ref{th:good} or the proof of Proposition~\ref{prop:tildeduality}). We assume the simplices are given an orientation. Let us write each $\beta_i$ as~$\sum_\ell c_\ell\otimes v_{i,\ell}$, where $c_\ell:[0,1]\to X$ runs among the oriented $1$-simplices of the chosen triangulation, and~$v_{i,\ell}$ is a (possibly zero) section of $\cV^\an$ with moderate growth in the neighbourhood of $c_\ell$. The cycle condition reads:
\[
\text{for each vertex }x\in U^\an,\quad\sum_{\ell\mid c_\ell(1)=x}v_{i,\ell}(x)-\sum_{\ell\mid c_\ell(0)=x}v_{i,\ell}(x)=0.
\]

The middle period matrix $\Ppairing^1_\rmid$ is the $d\times d$ matrix with entries
\[
\mathrm{Pf}\,\sum_\ell\int_{c_\ell}\langle v_{i,\ell},\omega_j\rangle,
\]
where the ``finite part'' $\mathrm{Pf}$ means that, equivalently,
\begin{itemize}
\item
either we replace $\beta_i$ with $\alpha_i$ that we realize as $\sum_\ell c_\ell\otimes v'_{i,\ell}$, where $v'_{i,\ell}$ has rapid decay near the boundary points of $c_\ell$ that are contained in $\Div$ and satisfy the cycle condition, and
\[
\mathrm{Pf}\int_{c_\ell}\langle v_{i,\ell},\omega_j\rangle=\int_{c_\ell}\langle v'_{i,\ell},\omega_j\rangle,
\]
\item
or we replace $\omega_j$ with $(\wh m_j,\omega_j)$, for each $\ell$ such that $c_\ell(o)\in\Div$ ($o=0,1$), we choose a germ~$\wt m_{j,\ell}^o$ as in Proposition \ref{prop:calculP}, and we set
\[
\mathrm{Pf}\int_{c_\ell}\langle v_{i,\ell},\omega_j\rangle=\lim_{\varepsilon\to0} \biggl[\int_\varepsilon^{1-\varepsilon}\hspace*{-4mm}c_\ell^*(\langle v_{i,\ell},\omega_j\rangle)+\langle v_{i,\ell},\wt m_{j,\ell}^0\rangle(c_\ell(0))-\langle v_{i,\ell},\wt m_{j,\ell}^1\rangle(c_\ell(1)) \biggr].
\]
\end{itemize}

\subsubsection*{Middle quadratic relations}
They now read, in terms of matrices in the chosen bases (\cf Corol\-lary \ref{cor:good}),
\begin{equation}\label{eq:quadraticone}
\mp(2\pii)\,\Bpairing_1^\rmid=\Ppairing_1^\rmid\cdot(\DRpairing^1_\rmid)^{-1}\cdot{}^t\Ppairing_1^\rmid.
\end{equation}

\appendix
\section{Twisted singular chains}\label{sec:chains}
In this section, we recall classical results from the Cartan Seminars \cite{Cartan48,Cartan50}. However, we do not use homology of cosheaves as in \cite{Bredon97}. We assume that $X$ is a compact topological space, so that all locally finite open coverings are finite. All sheaves are sheaves of vector spaces over some field $\kk$, in order to avoid any problem with torsion.

\subsection{Presheaves and sheaves}

Let $\wt\cF$ be a presheaf on $X$ and let $\cF$ be the associated sheaf. We have $\cF_x=\varinjlim_{U\ni x}\wt\cF(U)$. For every open set $U\subset X$, there is a natural morphism \hbox{$\wt\cF(U)\to \coprod_{x\in U}\cF_x$.} We denote by $\wt s_x\in\cF_x$ the germ of a section $\wt s\in\wt\cF(X)$. If $\wt s_x=0$, there exists $U_x$ such that the image of $\wt s$ in $\wt\cF(U_x)$ vanishes. It follows that the support of~$\wt s$, that is, the subset $\{x\in X\mid \wt s_x=0\}$ is closed.

\begin{itemize}
\item
We say that $\wt\cF$ has the \emph{surjectivity property} if
\begin{equation}
\label{eq:surj}
\text{for all }x\in X,\quad\wt\cF(X)\to \cF_x\text{ \ is surjective}.\end{equation}
This property is also called $\Phi$-softness, for $\Phi$ the family of closed subsets of $X$ consisting of points only.
\item
We say that $\wt\cF$ has the \emph{injectivity property} if
\begin{equation}
\label{eq:inj}
\wt\cF(X)\to \coprod_{x\in X}\cF_x\text{ \ is injective}.
\end{equation}
The latter condition is equivalent to asking that the natural morphism $\wt\cF(X)\to\Gamma(X,\cF)$ is injective, since \eqref{eq:inj} factorizes through the latter.
\item
We say that $\wt\cF$ is a \emph{fine presheaf} if, for any (locally) finite open covering $\cU=(U_i)$ of~$X$, there exist closed subsets $F_i\subset U_i$ and endomorphisms $\wt\ell_i$ of $\wt\cF$ (\ie each $\wt\ell_i$ is a family, indexed by open sets $U$, of endomorphisms $\wt\ell_i:\wt\cF(U)\to\wt\cF(U)$ compatible with restrictions $U'\subset U$ in an obvious way) such that
\begin{enumeratea}
\item\label{enum:a}
$\wt\ell_i=0$ on $\wt\cF(U)$
if $U\cap F_i=\emptyset$,
\item
for any open set $U$ and each $\wt s\in\wt\cF(U)$, the sum $\sum_i\wt\ell_i(\wt s)$ exists in $\wt\cF(U)$ and is equal to $\wt s$. Let $\ell_i:\cF\to\cF$ be the endomorphism associated with $\wt\ell_i$. If $x\notin F_i$, there exist $U\ni x$ such that $U\cap F_i=\emptyset$, and hence $\ell_{i,x}=0:\cF_x\to\cF_x$. Moreover, for each $U$ and $s\in\Gamma(U,\cF)$, we have, for any $x\in U$, the relation $\sum_i\ell_{i,x}(s_x)=s_x$ (the sum is finite since the covering is finite).
\end{enumeratea}
\end{itemize}

\begin{prop}\label{prop:fine}
Assume that $\wt\cF$ is a fine presheaf that satisfies the surjectivity property \eqref{eq:surj}. Then the natural morphism\vspace*{-3pt}
\[
\wt\cF(X)\to\Gamma(X,\cF)
\]
is onto. If it moreover satisfies the injectivity property \eqref{eq:inj}, then it is an isomorphism.
\end{prop}

\begin{proof}
Once the first statement is proved, the second statement is obvious. Let $s\in\Gamma(X,\cF)$.
By~\eqref{eq:surj}, for any $y\in X$,
there exists a section $\wt s^y\in\wt\cF(X)$ whose germ at $y$ is~$s_y$.
Then there exists a neighbourhood $U_y$ of $y$
such that $\wt s^y_x=s_x$ for any $x\in U_y$.
Let $\cU$ be a finite open covering of~$X$ such that each open set $U_i$ is of the form $U_y$ as above. Let $\wt s^i\in\wt\cF(X)$ be such that $\wt s^i_x=s_x$ for any $x\in U_i$. On the one hand, we have $\ell_{i,x}(s_x)=(\wt\ell_i(\wt s^i))_x$ for any $x\in U_i$, and on the other hand $s_x=\sum_{i\mid U_i\ni x}\ell_{i,x}(s_x)$, whence\vspace*{-3pt}
\[
s_x=\sum_{i\mid U_i\ni x}(\wt\ell_i(\wt s^i))_x=\sum_i(\wt\ell_i(\wt s^i))_x,
\]
since $(\wt\ell_i(\wt s^i))_x=0$ for $x\notin U_i$ according to \eqref{enum:a} above. We conclude, since the covering is finite,
\[
s=\sum_i\image(\wt\ell_i(\wt s^i))=\image\Bigl(\sum_i\wt\ell_i(\wt s^i)\Bigr),
\]
where we consider the image of the morphism $\wt\cF(X)\to\Gamma(X,\cF)$. It follows that the natural morphism $\wt\cF(X)\to\Gamma(X,\cF)$ is surjective.
\end{proof}

\subsection{Homotopically fine sheaves}
See \cite[\S8]{Cartan50-18}.

\begin{defi}\label{def:homfine}
Let $(\wt\cF,\rd)$ be a $\ZZ$-graded differential presheaf on $X$. We say that $\wt\cF$ is homotopically fine if, for any (locally) finite open covering $\cU=(U_i)$ of $X$, there exist closed subsets~$F_i$ contained in~$U_i$ and endomorphisms $\wt\ell_i$ and $\wt k$ of $\wt\cF$ such that
\begin{enumerate}
\item\label{def:homfine1}
$\wt\ell_i=0:\wt\cF(U)\to\wt\cF(U)$ if $U\cap F_i=\emptyset$,
\item\label{def:homfine2}
$\sum_i\wt\ell_i=\id+\wt k\rd+\rd \wt k$.
\end{enumerate}
\end{defi}

We have a similar definition for sheaves. We note the following properties:
\begin{enumeratea}
\item\label{enum:homfinea}
If a graded differential presheaf is homotopically fine, the associated graded differential sheaf is also homotopically fine.
\item\label{enum:homfineb}
If $(\wt\cF,\rd)$ is homotopically fine and if $\wt\cG$ is any presheaf, then the graded differential presheaf $(\wt\cF\otimes\wt\cG,\rd\otimes\id)$ is also homotopically fine. Recall that $\wt\cF\otimes\wt\cG$ is the presheaf \hbox{$U\mto\wt\cF(U)\otimes\wt\cG(U)$.} The same property holds for graded differential sheaves.
\end{enumeratea}

\begin{thm}[{\cite[\S3, Corollaire]{Cartan50-19}}]\label{th:isofine}
Let $(\cF,\rd)$ be a homotopically fine $\ZZ$-graded differential sheaf. For each $p$, the natural morphism\vspace*{-3pt}
\[
\coH^p(\Gamma(X,\cF),\rd)\to\bH^p(X,(\cF,\rd))
\]
is an isomorphism. 
\end{thm}

The proof is obtained by choosing a resolution $C^\cbbullet$ of $\kk_X$ by fine sheaves and realizing $\bH^p(X,(\cF,\rd))$ as $\coH^p(\Gamma(X,(C\otimes\cF)^\cbbullet)$, where $(C\otimes\cF)^\cbbullet$ stands for the simple complex associated with the double complex $C^p\otimes\cF^q$. The assertion follows from the degeneration at $E_1$ of the second spectral sequence, since the second filtration is regular. We can complete this theorem in terms of presheaves as follows.

\begin{prop}\label{prop:homotopfine}
Let $(\wt\cF,\rd)$ be a homotopically fine $\ZZ$-graded differential pre\-sheaf satisfying~\eqref{eq:surj} and \eqref{eq:inj}. Then the natural morphism
\[
\coH^p(\wt\cF(X),\rd)\to\bH^p(X,(\cF,\rd))
\]
is an isomorphism.
\end{prop}

\begin{proof}
According to Theorem \ref{th:isofine}, it is enough to prove that $(\wt\cF(X),\rd)\to(\Gamma(X,\cF),\rd)$ induces an isomorphism in cohomology.
\begin{itemize}
\item
Let $s\in\Gamma(X,\cF)$ be such that $\rd s=0$, and let $\wt s^i$ be as in the proof of Proposition \ref{prop:fine}. Recall that $\ell_{i,x}(s_x)=(\wt\ell_i(\wt s^i))_x$ for any $x\in U_i$. Besides, $s_x=\sum_{i\mid U_i\ni x}\ell_{i,x}(s_x)-\rd k s_x$ since $s$ is $\rd$-closed. Hence
\[
s_x+\rd ks_x=\sum_{i\mid U_i\ni x}(\wt\ell_i(\wt s^i))_x=\sum_i(\wt\ell_i(\wt s^i))_x,
\]
since $(\wt\ell_i(\wt s^i))_x=0$ for $x\notin U_i$ according to \eqref{enum:a} above. Since the covering is finite, we get
\[
s+\rd ks=\sum_i\image(\wt\ell_i(\wt s^i))=\image\Bigl(\sum_i\wt\ell_i(\wt s^i)\Bigr).
\]
Since $\rd(s-\rd ks)=0$, \eqref{eq:inj} implies that $\sum_i\wt\ell_i(\wt s^i)$ is closed, which gives the surjectivity of the cohomological map.
\item
Let $\wt s$ be a $\rd$-closed section of $\wt\cF(X)$ whose image $s\in\Gamma(X,\cF)$ satisfies $s=\rd\sigma$ for some $\sigma\in\Gamma(X,\cF)$. Up to replacing $\sigma$ with $\sigma+(\rd k+k \rd)\sigma$ (and $\wt s,s$ with $\wt s+\rd \wt k\wt s$, $s+\rd ks$) we can assume that $\sigma$ is the image of some $\wt\sigma\in\wt\cF(X)$ as argued above. The image of~$\wt s-\rd\wt\sigma$ in $\Gamma(X,\cF)$ is zero, and hence $\wt s=\rd\wt\sigma$, according to \eqref{eq:inj}. This proves the injectivity of the cohomological map.\qedhere
\end{itemize}
\end{proof}

\subsection{Homotopy operator for singular chains}
See \cite[\S3]{Cartan48-08}.

Let ($S_\bbullet(X),\partial)$ be the complex of singular chains. Recall that $S_\bbullet(X)$ is the infinite-dimensional vector space having as a basis the set of singular simplices (continuous maps from a simplex to~$X$). The support of a singular simplex is its image, which is a compact subset of $X$. For a closed subset $Z\subset X$, we regard $S_\bbullet(Z)$ as the subspace of $S_\bbullet(X)$ generated by simplices with support in $Z$ and $S_\bbullet(X)/S_\bbullet(Z)$ as the subspace generated by simplices whose support is not contained in $Z$. For any (locally) finite open covering $\cU=(U_i)$, there exist (\cf \loccit) endomorphisms $\ell$ and $k$ of $S_\bbullet(X)$ (where $\ell$ preserves the grading and $k$ has degree one) such that
\begin{equation}\label{eq:lk}
\begin{cases}
\bbullet
\text{ for any simplex $\sigma\in S_\bbullet(X)$, $\ell(\sigma)$ is a sum of simplices,}\\
\hphantom{\bbullet}
\text{each one being contained in some~$U_i\cap|\sigma|$,}\\
\bbullet
\text{ if $|\sigma|$ is contained in some $U_i$, then $\ell(\sigma)=\sigma$,}\\
\bbullet
\text{ $\ell=\id+k\partial+\partial k$ (in particular, $\ell\partial=\partial\ell$).}
\end{cases}
\end{equation}
For $Z$ closed in $X$, let $(\wt\ccC_{X,Z,\bbullet},\partial)$ be the chain presheaf
\[
U\mto \wt\ccC_{X,Z,\bbullet}(U)=S_\bbullet(X)/S_\bbullet((X\moins U)\cup Z)
\]
and let $(\ccC_{X,Z,\bbullet},\partial)$ be the associated differential sheaf (with the usual convention $(\ccC_{X,Z}^{-\cbbullet},\rd)=(\ccC_{X,Z,\bbullet},\partial)$, we obtain a $(-\NN)$-graded differential sheaf). For any $x\in X$, we have
\[
\ccC_{X,Z,\bbullet,x}=S_\bbullet(X)/S_\bbullet((X\moins\{x\})\cup Z).
\]
More precisely, $\ccC_{X,Z,\bbullet,x}$ is identified with the subspace of $S_\bbullet(X)$ with basis consisting of those singular simplices whose support contains $x$ and is not contained in~$Z$. Since $\wt\ccC_{X,Z,\bbullet}(X)=S_\bbullet(X)/S_\bbullet(Z)$, properties \eqref{eq:surj} and \eqref{eq:inj} obviously hold.

\begin{prop}[{\cite[p.\,8]{Cartan50-18}}]\label{prop:Chomotopfine}
Given a finite open covering $\cU$ of $X$, there exist a closed covering~$(F_i)$ with $F_i\subset U_i$ for all $i$ and endomorphisms $\wt\ell_i$ and $\wt k$ of the presheaf $\wt\ccC_{X,Z,\bbullet}$ such~that
\begin{enumerate}
\item\label{prop:Chomotopfine1}
$\wt\ell_i=0:\wt\ccC_{X,Z,\bbullet}(U)\to\wt\ccC_{X,Z,\bbullet}(U)$ if $U\cap F_i=\emptyset$, and in particular $\wt\ell_i$ induces the zero map on each germ $\cF_x$ with $x\notin F_i$,
\item
for any singular simplex $\sigma$, each simplex component $\tau$ of $\wt\ell_i(\sigma)$ satisfies $|\tau|\subset F_i\cap|\sigma|$,
\item\label{prop:Chomotopfine3}
$\sum_i\wt\ell_i=\id+\wt k\rd+\rd \wt k$.
\end{enumerate}
\end{prop}

\begin{cor}\label{cor:Chomotopfine}
The presheaf $(\wt\ccC_{X,Z,\bbullet},\partial)$ is homotopically fine.\qed
\end{cor}

It follows that the chain complex $(\ccC_{X,Z,\bbullet},\partial)$ is homotopically fine.

\begin{proof}[Proof of Proposition \ref{prop:Chomotopfine}]
Let $(U_i)_{i\in I}$ be a finite open covering of $X$. One can find closed subsets $F_i\subset U_i$ such that $(F_i)_{i\in I}$ is a closed covering of $X$. We fix a total order on $I$ and we write $I=\{1,\dots,n\}$. We~also fix $\ell$ and $k$ as in \eqref{eq:lk}.

Let $\sigma$ be a simplex in $S_\bbullet(X)$. Then we define $\wt\ell_i(\sigma)$ inductively on $i=1,\dots,n$ as the sum of components of the chain $\ell(\sigma)-\sum_{j<i}\wt\ell_j(\sigma)$ whose underlying simplices have support in $F_i$. The~support is also contained in $|\sigma|$. We can then extend the definition to any finite chain $\sigma\in S_\bbullet(X)$. Then $\ell(\sigma)=\sum_i\wt\ell_i(\sigma)$ in $S_\bbullet(X)$ and $\sum_i\wt\ell_i(\sigma)=\sigma+k\partial\sigma+\partial k\sigma$.

Let $U$ be an open subset of $X$. Since $\ell,k,$ and $\partial$ preserve the support, the above construction also preserves $S((X\moins U)\cup Z)$, and thus defines endomorphisms $\wt\ell_i$ and $\wt k$ of $\wt\ccC_{X,Z,\bbullet}(U)$ satisfying the desired properties.
\end{proof}

\subsection{Singular chains with coefficients in a sheaf}\label{app:chainssheaf}
We make a statement of \hbox{\cite[p.\,8]{Cartan50-18}} precise. Let $\cF$ be any sheaf. We denote by $\wt\ccC_{X,Z,\bbullet}(\cF)$ the presheaf
\[
U\mto\wt\ccC_{X,Z,\bbullet}(U)\otimes\Gamma(U,\cF)
\]
and by $\ccC_{X,Z,\bbullet}(\cF)$ the associated sheaf $\ccC_{X,Z,\bbullet}\otimes\cF$. By Corollary \ref{cor:Chomotopfine} and \eqref{enum:homfinea} and \eqref{enum:homfineb} after Definition~\ref{def:homfine}, both $\wt\ccC_{X,Z,\bbullet}(\cF)$ and $\ccC_{X,Z,\bbullet}(\cF)$ are homotopically fine. These are chain complexes indexed by~$\NN$, when equipped with the boundary $\partial\otimes\id$, that we simply denote by~$\partial$. As usual, we regard them as differential sheaves with grading indexed by $-\NN$ by setting \eg $\wt\ccC_{X,Z}^{-p}(\cF)=\wt\ccC_{X,Z,p}(\cF)$ and with degree-one differential $\rd$ identified with $\partial$. Let us set, by definition,
\[
\coH_p(X,Z;\cF)=\bH^{-p}(X,(\ccC_{X,Z}^\cbbullet(\cF),\rd)).
\]
Then, since each term of the complex $\ccC_{X,Z}^\cbbullet(\cF)$ is homotopically fine, Theorem \ref{th:isofine} implies
\[
\coH_p(X,Z;\cF)\simeq\coH_p(\Gamma(X,(\ccC_{X,Z,\bbullet}(\cF),\partial)).
\]
In particular, given an exact sequence $0\to\cF'\to\cF\to\cF''\to0$ of sheaves, there is a long exact sequence
\begin{equation}\label{eq:longexhomology}
\cdots\ra\coH_p(X,Z;\cF')\ra\coH_p(X,Z;\cF)\ra\coH_p(X,Z;\cF'')\ra\coH_{p-1}(X,Z;\cF')\ra\cdots
\end{equation}
Besides, let us consider the vector space\footnote{It should not be confused with $\wt\ccC_{X,Z,\bbullet}(\cF)(X)$, that is, $\wt\ccC_{X,Z,\bbullet}(X)\otimes\cF(X)$.}
\begin{equation}\label{eq:CF}
\wt\ccC_{X,Z,\bbullet}(X,\cF)=\bigoplus_\sigma \bigl[\kk\sigma\otimes\Gamma(|\sigma|,\cF)\bigr],
\end{equation}
where the sum runs over singular simplices $\sigma$ whose support is not contained in $Z$, together with the boundary map induced by $\partial\otimes\id$. More precisely, if $\partial\sigma=\sum_k\tau_k$, then $\partial(\sigma\otimes f)=\sum_j\tau_j\otimes f|_{|\tau_j|}$, where the sum is taken over $j$ such that the support of $\tau_j$ is not contained in $Z$. A $p$-cycle is a sum $\sum_a(\sigma_a\otimes f_a)$ such that, for each $(p-1)$-simplex $\tau$ with support not contained in $Z$,
\begin{equation}\label{eq:cyclecond}
\sum_{\substack{a,k\\ \tau_{a,k}=\tau}}f_a|_{|\tau_{a,k}|}=0\in \Gamma(|\tau|,\cF).
\end{equation}

Such an element $\sigma\otimes f$, with $f\in\Gamma(|\sigma|,\cF)$, defines a global section in $\Gamma(X,\ccC_{X,Z,\bbullet}(\cF))$. Indeed, let $U$ be an open neighbourhood of $|\sigma|$ on which $f$ is defined. Then $\sigma\otimes f$ defines an element of $\wt\ccC_{X,Z,\bbullet}(U)\otimes\Gamma(U,\cF)$, and hence of $\Gamma(U,\ccC_{X,Z,\bbullet}(\cF))$. Moreover, since the image of~$\sigma$ is zero in $\Gamma(U,\ccC_{X,Z,\bbullet})$ if $U\cap|\sigma|=\emptyset$, this element is supported on $|\sigma|$, and hence extends in a unique way to a global section of $\ccC_{X,Z,\bbullet}(\cF)$ on $X$.

The natural morphism
\begin{equation}\label{eq:CFinj}
\wt\ccC_{X,Z,\bbullet}(X,\cF)\to \Gamma(X,\ccC_{X,Z,\bbullet}(\cF))
\end{equation}
is clearly compatible with $\partial$. Moreover, it is injective. Indeed, let $(\sigma_k)_k$ be distinct singular simplices and, for each $k$, let $f_k\in\Gamma(|\sigma_k|,\cF)$ be such that $\sum_k\sigma_{k,x}\otimes\nobreak f_{k,x}=0$ for all $x$, where~$\sigma_{k,x}$ denotes the image of $\sigma_k$ in $\ccC_{X,Z,\bbullet,x}$. Let~$K_x$ be the set of indices $k$ such that $x\in|\sigma_k|$. Then~$(\sigma_{k,x})_{k\in K_x}$ is part of a basis of $\ccC_{X,Z,\bbullet,x}$, and $\sum_{k\in K_x}\sigma_{k,x}\otimes f_{k,x}=0$ implies that $f_{k,x} = 0$ for any $k\in K_x$. Since $\cF$ is a sheaf, this implies the vanishing $f_k=0$ in $\Gamma(|\sigma_k|,\cF)$ for all~$k$.

\begin{prop}\label{prop:comparisonhomology}
The chain map $(\wt\ccC_{X,Z,\bbullet}(X,\cF),\partial)\to(\Gamma(X,\ccC_{X,Z,\bbullet}(\cF)),\partial)$ induces an iso\-mor\-phism
\[
\coH_p(\wt\ccC_{X,Z,\bbullet}(X,\cF),\partial)\simeq\coH_p(X,Z;\cF).
\]
\end{prop}

\begin{proof}
We argue as for Proposition \ref{prop:homotopfine}.
\begin{itemize}
\item
Let $s\in\Gamma(X,\ccC_{X,Z,\bbullet}\otimes\cF)$. There exists an open covering $\cU$ and $\wt s^i\in\wt\ccC_{X,Z,\bbullet}(U_i)\otimes\cF(U_i)$ such that $s_x=\wt s^i_x$ for any $x\in\nobreak U_i$. We~decompose $\wt s^i=\sum_j\sigma_{ij}\otimes f_{ij}$, where $\sigma_{ij}$ are singular simplices with support not contained in~$(X\moins U_i)\cup Z$, and $f_{ij}\in\cF(U_i)$. Let~$(F_i)$ be a closed covering such that $F_i\subset U_i$ for each $i$ and let $(\wt\ell_i,\wt k)$ be as in Proposition \ref{prop:Chomotopfine}, tensored with $\id$ so that they act on the presheaf $\wt\ccC_{X,Z,\bbullet}\otimes\cF$.
We denote by $\ell_i$ and $k$ the induced sheaf morphisms.
We have $s=\sum_i\ell_i(s)-(\partial k+k\partial)s$. For any $x$, $\ell_i(s)_x=0$ if $x\notin F_i$, according to Proposition \ref{prop:Chomotopfine}\eqref{prop:Chomotopfine1}, and $\ell_i(s)_x=\wt\ell_i(\wt s^i)_x$ otherwise, and therefore
\begin{align*}
s_x=\hspace*{-1mm}\sum_{i\mid U_i\ni x}\hspace*{-1mm}\wt\ell_i(\wt s^i)_x
	-((k\partial+\partial k)s)_x
&=\hspace*{-1mm}\sum_{i\mid U_i\ni x}\sum_j\wt\ell_i(\sigma_{ij})_x\otimes f_{ij,x}
	-((k\partial+\partial k)s)_x\\
&=\sum_i\sum_j\wt\ell_i(\sigma_{ij})_x\otimes f_{ij,x}
	-((k\partial+\partial k)s)_x,
\end{align*}
since $\wt\ell_i(\sigma_{ij})_x=0$ if $x\notin U_i$. Moreover, if we write $\wt\ell_i(\sigma_{ij})\otimes f_{ij}=\sum_a \sigma_a\otimes f_a$, where $\sigma_a$ are singular simplices satisfying $|\sigma_a|\subset|\sigma_{ij}|\subset F_i$, and $f_a\in\cF(U_i)$, we have $(\sigma_a\otimes f_a)_x=0$ if~$x\notin|\sigma_a|$, and hence we can replace $f_a$ with its restriction in $\Gamma(|\sigma_a|,\cF)$. Therefore,
\[
s=\image\Bigl[\sum_i\sum_j\wt\ell_i(\sigma_{ij})\otimes f_{ij}\Bigr]
-((k\partial+\partial k)s),
\]
where the term between brackets belongs to
$\wt\ccC_{X,Z,\cbbullet}(X,\cF)$.
If $s$ is closed, then $s+\partial ks$
is the image of $\sum_i\sum_j\wt\ell_i(\sigma_{ij})\otimes f_{ij}$
and the latter is closed since $\partial(s+\partial ks)=0$
and \eqref{eq:CFinj} is injective. This implies surjectivity of the homology map.
\item
Let $\wt s=\sum_a\sigma_a\otimes f_a$, where $\sigma_a$ are singular simplices and $f_a\in\Gamma(|\sigma_a|,\cF)$, be closed and such that its image $s$ is equal to $\partial s'$. The above formula shows that, up to replacing~$s'$ with $s'+(k\partial+\partial k) s'$ (and $s$ with $s+\partial ks$), we can assume that $s'$ lies in the image of $\wt s'\in\wt\ccC_{X,Z,\bbullet}(X,\cF)$. Then $\wt s-\partial\wt s'$ has image zero, and hence is zero according to the injectivity of \eqref{eq:CFinj}, so the homology map is injective.\qedhere
\end{itemize}
\end{proof}

\begin{prop}[Excision]\label{prop:excision}
Let $Y$ be a closed subset contained in the interior $Z^\circ$ of $Z$. Then the natural morphism\vspace*{-3pt}
\[
\coH_p(X\moins Y,Z\moins Y;\cF)\to \coH_p(X,Z;\cF)
\]
is an isomorphism for all $p$.
\end{prop}

\begin{proof}
We consider the open covering $\cU=(X\moins Y,Z^\circ)$. Let $(\wt\ccC_{X,Z,\bbullet}^\cU(X,\cF),\partial)$ be the subcomplex of $(\wt\ccC_{X,Z,\bbullet}(X,\cF),\partial)$ for which the spaces $\kk\sigma\otimes\Gamma(|\sigma|,\cF)$ occurring as components are those for which~$|\sigma|$ is contained in one of the open sets of $\cU$. By using $(\ell,k)$ adapted to this covering (\cf\eqref{eq:lk}), one obtains that any closed $s\in\wt\ccC_{X,Z,p}(X,\cF)$ can be written as $\ell(s)+\partial k s$, thus showing the surjectivity of the homology map. On the other hand, if $s'\in \wt\ccC_{X,Z,p}^\cU(X,\cF)$ satisfies $s'=\partial s$ with $s\in\wt\ccC_{X,Z,p+1}(X,\cF)$, we have $s'=\ell(s')=\ell(\partial s)=\partial\ell(s)$, hence the injectivity.
\end{proof}

\begin{remark}\label{rem:homologymodels}
The long exact sequence \eqref{eq:longexhomology} is not easily seen in the model $\wt\ccC_{X,Z,\bbullet}(X,\cF)$, while the excision property is better understood in that model, as well as the long exact sequence, for the closed inclusions $Z\subset Y\subset X$:\vspace*{-3pt}
\[
\cdots\to\coH_p(Y,Z;\cF)\to\coH_p(X,Z;\cF)\to\coH_p(X,Y;\cF)\to\coH_{p-1}(Y,Z;\cF)\to\cdots
\]
On the other hand, if $\cF$ is the constant sheaf with fiber $F$ (assume that $X$ is connected), then for each singular simplex $\sigma$, the support $|\sigma|$ is connected, so $\Gamma(|\sigma|,\cF)$ is canonically identified with $F$ by the restriction morphism\vspace*{-3pt}
\[
F=\Gamma(X,\cF)\to\Gamma(|\sigma|,\cF),
\]
so\vspace*{-3pt}
\[
(\wt\ccC_{X,Z,\bbullet}(X,\cF),\partial)\simeq(\wt\ccC_{X,Z,\bbullet}(X)\otimes F,\partial)
\]
and $\coH_p(X,Z;\cF)$ is the usual singular homology with coefficients in $F$.
\end{remark}

\subsection{Piecewise smooth and simplicial chains}\label{app:chainssheafsimplicial}
Assume that $X$ is a manifold (possibly with boundary). We denote by $(\wt\ccC^\sm_{X,Z,\bbullet},\partial)$ the subcomplex of $(\wt\ccC_{X,Z,\bbullet},\partial)$ consisting of piecewise smooth singular chains (\ie having a basis formed of piecewise smooth singular simplices), and by $(\ccC^\sm_{X,Z,\bbullet},\partial)$ the associated complex of sheaves. On the other hand, for a sheaf $\cF$, we define $\wt\ccC^\sm_{X,Z,\bbullet}(X,\cF)$ in a way similar to~\eqref{eq:CF}.

\begin{prop}\label{prop:Csm}
The inclusions of chain complexes
\begin{align*}
(\ccC^\sm_{X,Z,\bbullet},\partial)&\hto(\ccC_{X,Z,\bbullet},\partial),\\
(\wt\ccC^\sm_{X,Z,\bbullet}(X,\cF),\partial)&\hto(\wt\ccC_{X,Z,\bbullet}(X,\cF),\partial)
\end{align*}
induce isomorphisms in homology.
\end{prop}

\begin{proof}
We will prove the statement for the second inclusion, the proof for the first one being similar.

Let $s=\sum_a\sigma_a\otimes f_a$, with $f_a\in\Gamma(|\sigma_a|,\cF)$, be a $p$-cycle in $\wt\ccC_{X,Z,\bbullet}(X,\cF)$. We can assume that~$f_a$ is the germ along $|\sigma_a|$ of $f_a\in\Gamma(U_a,\cF)$ for some open neighbourhood $U_a$ of $|\sigma_a|$. The cycle condition is as in \eqref{eq:cyclecond} and for each $\tau$ there exists $U_\tau$ such that the corresponding sum~$\sum f_a$ is zero on $U_\tau$. We can approximate each $\sigma_a$ by piecewise smooth simplices with image contained in~$U_a$ and more precisely construct a family of approximations $\Sigma_a(t)=\sigma_a^t$ ($t\in[0,1]$) with $\sigma_a^0=\sigma_a$ and~$\sigma_a^t$ piecewise smooth for each $t\in(0,1]$ such that $\tau\times[0,1]$ has image in $U_\tau$ for each face $\tau$ of $\sigma_a$. We~can then subdivide $\Delta_p\times[0,1]$ to regard each $\Sigma_a$ as a $(p+1)$-chain and the total family $S=\sum_a\Sigma_a\otimes f_a$ as an element of $\wt\ccC_{X,Z,p+1}(X,\cF)$ (we restrict $f_a$ to $|\Sigma_a|\subset U_a$). The boundary chain~$\partial S$ reads
\[
\sum_{a,k}T_{a,k}\otimes f_a|_{|T_{a,k}|}+\sigma_a^1\otimes f_a-\sigma_a^0\otimes f_a,
\]
where $T_{a,k}$ is a map $\Delta_{p-1}\times[0,1]\to U_a$ satisfying $T_{a,k}^0=\tau_{a,k}$ and \hbox{$|T_{a,k}|\subset U_{\tau_{a,k}}$}. The cycle condition on~$s$ implies that, for each $\tau:\Delta_{p-1}\to X$, we have
\[
\sum_{a,k\mid\tau_{a,k}=\tau} f_a|_{U_\tau}=0,
\]
so that $\partial S=\sigma_a^1\otimes f_a-\sigma_a^0\otimes f_a$. This shows the surjectivity of the homology map, and its injectivity follows.
\end{proof}

On the other hand, assume that $X$ is endowed with a simplicial structure~$\cT$ compatible with~$Z$. By this, we mean that $X$ is the support of a simplicial complex $\cT$ and $Z$ is the support of a subcomplex~$\cT_Z$ of~$\cT$. We~denote by $(\wt\ccC^\triup_{\cT,\cT_Z,\bbullet},\partial)$ the subcomplex of $(\wt\ccC_{X,Z,\bbullet},\partial)$ consisting of simplicial chains of the simplicial structure. We define correspondingly the simplicial chain complex $(\wt\ccC^\triup_{\cT,\cT_Z,\bbullet}(X,\cF),\partial)$. We~make the following assumption on $\cF$ and the simplicial structure.

\begin{assumption}\label{ass:simplicial}\mbox{}
For any simplex $\sigma$ of the simplicial complex $\cT$, there exists an open neighbourhood $U_\sigma$ of $\sigma$ that retracts onto $\sigma$ and such that the sheaf $\cF_{|U_\sigma}$ can be decomposed as a direct sum of sheaves $\cF_{\sigma,i}$, each of which is constant of finite rank on $U_\sigma\moins T_{\sigma,i}$ and zero on $T_{\sigma,i}$, for some closed subset $T_{\sigma,i}$ of $U_\sigma$ that intersects $\sigma$ along a (possibly empty or full) closed face $\tau_i$~of~$\sigma$.
\end{assumption}

We note that, if $\cT'$ is a subdivision of $\cT$ and $\cT'_Z$ is the corresponding subdivision of $\cT'_Z$, then $(\cF,\cT',\cT'_Z)$ satisfies Assumption \ref{ass:simplicial} if $(\cF,\cT,\cT_Z)$ does so.

\begin{prop}\label{prop:triup}
Under Assumption \ref{ass:simplicial}, the inclusion of chain complexes
\[
(\wt\ccC^\triup_{\cT,\cT_Z,\bbullet}(X,\cF),\partial)\hto(\wt\ccC_{X,Z,\bbullet}(X,\cF),\partial)
\]
induces an isomorphism in homology.
\end{prop}

We immediately note the following consequence, which justifies that we denote from now on by $(\wt\ccC^\triup_{X,Z,\bbullet}(X,\cF),\partial)$ the complex $(\wt\ccC^\triup_{\cT,\cT_Z,\bbullet}(X,\cF),\partial)$ when Assumption \ref{ass:simplicial} holds.

\begin{cor}
The homology of $(\wt\ccC^\triup_{\cT,\cT_Z,\bbullet}(X,\cF),\partial)$ is independent of $\cT$ provided Assumption \ref{ass:simplicial} is satisfied. In particular, it is equal to the homology of $(\wt\ccC^\triup_{\cT',\cT'_Z,\bbullet}(X,\cF),\partial)$ for any subdivision~$\cT'$ of~$\cT$.\qed
\end{cor}

\begin{proof}[Proof of Proposition \ref{prop:triup}]
We can argue separately with $X$ and $Z$ and obtain the result for the pair $(X,Z)$ due to compatibility between long exact sequences of pairs. So we argue with $X$, as in \cite[\S34]{Munkres84}, by induction on the (finite) number of simplices in the simplicial decomposition $X$. Let $\sigma_o$ be a simplex of maximal dimension (that we can assume $\geq1$) and let $X'$ be the simplicial set $X$ with~$\sigma_o$ deleted (but the boundary simplices of~$\sigma_o$ kept in $X'$). The underlying topological space $X'$ is~$X$ with the interior of $\sigma_o$ deleted. Then $(X',\cF)$ also satisfies Assumption~\ref{ass:simplicial}. We consider, in both simplicial and singular homology, the long exact sequences of pairs
\[
\cdots\to\coH_p(\sigma_o;\cF)\to\coH_p(X;\cF)\to\coH_p(X,\sigma_o;\cF)\to\coH_{p-1}(\sigma_o;\cF)\to\cdots
\]
and the natural morphism between them. The assertion follows from Lemma \ref{lem:triup} below.
\end{proof}

\begin{lemma}\label{lem:triup}
With the notation and assumption from Proposition \ref{prop:triup},
\begin{enumerate}
\item\label{lem:triup1}
the natural morphism
\[
(\wt\ccC^\triup_{\sigma_o,\bbullet}(\sigma_o,\cF),\partial)\hto(\wt\ccC_{\sigma_o,\bbullet}(\sigma_o,\cF),\partial)
\]
induces an isomorphism in homology;

\item\label{lem:triup2}
both natural inclusions
\begin{align*}
(\wt\ccC^\triup_{X',\partial\sigma_o,\bbullet}(X',\cF),\partial)&\hto(\wt\ccC_{X,\sigma_o,\bbullet}^\triup(X,\cF),\partial),\\
(\wt\ccC_{X',\partial\sigma_o,\bbullet}(X',\cF),\partial)&\hto(\wt\ccC_{X,\sigma_o,\bbullet}(X,\cF),\partial),
\end{align*}
induce an isomorphism in homology.
\end{enumerate}
\end{lemma}

\begin{proof}\mbox{}
\par\eqref{lem:triup1}
Let us decompose $\cF_{|\sigma_o}$ as in Assumption \ref{ass:simplicial}, let us fix a corresponding component $\cF_{\sigma_o,i}$ of $\cF$ that we still call $\cF$, and let $\tau_o$ be the face of $\sigma_o$ on which it is zero. The assertion is trivial if $\tau_o=\sigma_o$. If $\tau_o=\emptyset$, then $\cF$ is constant and the proof is done in \cite[\S34]{Munkres84}.

Assume now that $\tau_o$ is non-empty and different from $\sigma_o$. For any face~$\tau$ of~$\sigma_o$ intersecting~$\tau_o$, $\cF_{|\tau}$ is constant on $\tau\moins(\tau\cap\tau_o)$ and zero on $\tau\cap\tau_o$, so that \hbox{$\Gamma(\tau,\cF)=0$}. Let $\tau'_o$ be the union of faces of~$\sigma_o$ not intersecting $\tau_o$. It is a single face of $\sigma_o$, and hence a simplex. Now, $\cF$ is constant on $\tau'_o$ and, obviously, $\coH_p^\triup(\sigma_o;\cF)=\coH_p^\triup(\tau'_o;\cF)$.

Let us now compute $\coH_p(\sigma_o;\cF)$. Let $\sigma:\Delta_p\to\sigma_o$ be a singular $p$-simplex of~$\sigma_o$, with image~$|\sigma|$. Note that the natural morphism $\Gamma(|\sigma|,\cF)\to\Gamma(\Delta_p,\sigma^{-1}\cF)$ is injective (since it preserves germs). If $|\sigma|\cap\tau_o\neq\emptyset$, the argument above shows that $\Gamma(\Delta_p,\sigma^{-1}\cF)=0$, and hence so is $\Gamma(|\sigma|,\cF)$. Thus, $\wt\ccC_{\sigma_o,p}(\sigma_o,\cF)$ only involves singular simplices $\sigma$ not intersecting~$\tau_o$. Let~$F$ be the constant value of~$\cF$ on $\sigma_o\moins\tau_o$. For such a simplex $\sigma$, we thus have $\Gamma(|\sigma|,\cF)\simeq\Gamma(\Delta_p,\sigma^{-1}\cF)=F$, and then $\wt\ccC_{\sigma_o,p}(\sigma_o,\cF)=\wt\ccC_{\sigma_o\moins\tau_o,p}(\sigma_o\moins\tau_o)\otimes F$. In other words, $\coH_p(\sigma_o;\cF)=\coH_p(\sigma_o\moins\tau_o)\otimes F$. Since $(\sigma_o\moins\tau_o)$ retracts to~$\tau'_o$, this group is also equal to $\coH_p(\tau'_o;\cF)$, and we conclude with \cite[\S34]{Munkres84}.

\par\eqref{lem:triup2}
The assertion for the simplicial complexes is obvious since it is already an equality. We thus consider the singular chain complexes.

Let us prove surjectivity of the homology map. Let $s$ be a closed singular chain in $\wt\ccC_{X,\sigma_o,p}(X,\cF)$. We can cover $X$ by $U_o$ and $U'$, where $U_o=U_{\sigma_o}$ is given by Assumption \ref{ass:simplicial} and $U'\subset X'$. Choosing $(\ell,k)$ as in \eqref{eq:lk}, we write $s+\partial ks=s_o+s'$, with $s_o,s'$ closed in $\wt\ccC_{U_o,\sigma_o,p}(X,\cF)$ and in $\wt\ccC_{U',p}(X,\cF)\subset \wt\ccC_{X',\sigma_o,p}(X',\cF)$ respectively. We decompose $s_o$ according to Assumption \ref{ass:simplicial}. For the component corresponding to ``$\cF$ constant on $\sigma_o$'', it is proved in \cite[\S34]{Munkres84} that it is homologous to a chain in $\wt\ccC_{U_o\cap X',\partial\sigma_o,p}(X,\cF)$. Assume thus that $\cF$ is zero on $T_o$ with $\tau_o=T_o\cap\sigma_o\neq\emptyset$, and constant on $U_o\moins T_o\neq\emptyset$. Arguing as in Case \eqref{lem:triup1}, we show that singular simplices occurring with a non-zero coefficient in~$s_o$ do not intersect~$T_o$. There is a projection $\sigma_o\moins\tau_o\to\tau'_o$ (with $\tau'_o$ as above),
so~that if we compose~$s_o$ with the corresponding retraction,
we obtain a homologous closed $p$-chain in
$\wt\ccC_{U_o\cap X',\partial\sigma_o,p}(X,\cF)$.

Assume now that $\wt s\!\in\!\wt\ccC_{X',\partial\sigma_o,p}(X,\cF)$ is equal to $\partial s$ with $s$ in $\wt\ccC_{X,\sigma_o,p+1}(X,\cF)$. It follows that $\partial\wt s=0$ in $\wt\ccC_{X,\sigma_o,p-1}(X,\cF)$, and hence in $\wt\ccC_{X',\partial\sigma_o,p-1}(X,\cF)$.
With $U_o,U'$, and $(\ell,k)$ as above,
we have
\[
\wt s+\partial k\wt s=\ell(\wt s)=\partial\ell(s)=\partial s_o+\partial s'.
\]
The above argument shows that $s_o$ is homologous to a~$(p+1)$-chain in
\[
\wt\ccC_{U_o\cap X',\partial\sigma_o,p+1}(X,\cF),
\]
and hence $\wt s$ is the boundary of a chain in~$\wt\ccC_{X',\partial\sigma_o,p+1}(X,\cF)$, which shows injectivity.
\end{proof}

\subsection{The dual chain complex with coefficient in a sheaf}\label{subsec:appdualchain}
In this section, we set $(X,Z)=(M,\bM)$, where $M$ is a manifold with corners, and we consider a simplicial decomposition $\cT$ of~$(M,\bM)$. Let $\cT'$ denote the first barycentric subdivision of $\cT$. For any simplex $\sigma$ in $\cT$, let $\wh\sigma$ denote its barycenter and let $D(\sigma)$ be the open cell dual to $\sigma$, which is the sum of open simplices of~$\cT'$ having $\wh\sigma$ as their final vertex (\cf\cite[\S64]{Munkres84} for definitions and details). If $\sigma$ is not contained in $\bM$, then the closure $\ov D(\sigma)$ does not intersect $\bM$. If $\sigma$ is contained in $\bM$, then $D(\sigma)$ is contained in $M^\circ$. In any case, if we regard $\ov D(\sigma)$ as a chain of $\cT'$ relative to $\bM$, then $\partial\ov D(\sigma)$ is a sum of terms $\ov D(\sigma_i)$ for some $\sigma_i$ in $\cT$. The following is clear.

\begin{lemma}\label{lem:A7}
Let $\sigma$ be a simplex of $\cT$ of dimension $p$. Then $\sigma$ is the only simplex of dimension~$p$ of~$\cT$ intersected by $\ov D(\sigma)$, and the intersection is transversal.
\end{lemma}

\begin{figure}[htb]
\includegraphics[scale=.65]{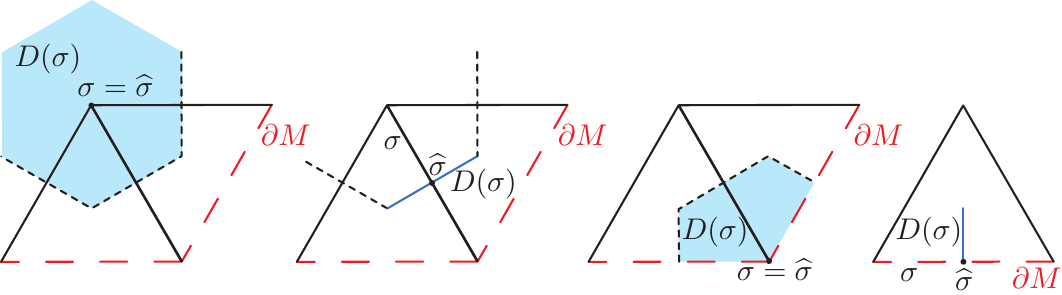}
\caption{Various examples of dual cells}
\end{figure}

We consider a sheaf $\cF$ on $M$ satisfying Assumption \ref{ass:simplicial} with respect to $\cT$, with the supplementary condition that $\cF$ is locally constant on \hbox{$M^\circ=M\moins\bM$}, that is, the closed subsets~$T_{\sigma,i}$ of Assumption \ref{ass:simplicial} are contained in $\bM$. We~note that $\cF$ also satisfies the previous properties with respect to $\cT'$. The chain complex $(\wt\ccC_{\cT^\vee,\bM,\bbullet}^\smallsquare(M,\cF),\partial)$ whose term of index $p$ consists of the direct sum  of the terms $\Gamma(\ov D(\sigma),\cF)$ with $\sigma$ of codimension $p$, is thus a subcomplex of $(\wt\ccC_{\cT',\bM,\bbullet}^\triup(M,\cF),\partial)$.

\begin{lemma}
Under the previous assumptions, the inclusion of chain complexes
\[
(\wt\ccC_{\cT^\vee,\bM,\bbullet}^\smallsquare(M,\cF),\partial)\hto(\wt\ccC_{\cT',\bM,\bbullet}^\triup(M,\cF),\partial)
\]
induces an isomorphism in homology.
\end{lemma}

\begin{proof}
Similar to that of \cite[Th.\,64.2]{Munkres84}.
\end{proof}

It follows from Proposition \ref{prop:triup} that the chain complex $(\wt\ccC_{\cT^\vee,\bM,\bbullet}^\smallsquare(M,\cF),\partial)$ also computes the relative homology $\coH_\bbullet(M,\bM;\cF)$.

\section{Poincaré lemma for currents}\label{app:currents}
In this section, we work in a local setting. We let $U$ be a convex open subset of $\RR^m$ (with coordinates $x_1,\dots,x_m$) containing the origin, we fix $r$ such that $1\leq r\leq m$, we set $g(x)=\textstyle\prod_{i=1}^rx_i$ and
\begin{align*}
M&=\{x_i\geq0\mid i=1,\dots,r\}\cap U,& \partial M&=\{g(x)=0\}\cap M,\\
 M^\circ&=\{g(x)>0\}\cap M& U^\circ&=\{g(x)\neq0\}.
\end{align*}
A $C^\infty$ function on $M$ can be defined either as the restriction to $M$ of a $C^\infty$ function defined on some neighbourhood of $M$ in $U$, or such that the derivatives to any order computed in $M$ exist and are continuous in $M$. It is known that both definitions define the same sheaf of functions (\eg by mirroring a function on $M$ with respect to the hyperplanes $x_i=0$, $i=1,\dots,r$). The first definition means that the sheaf $\cC^\infty_M$ is the sheaf-theoretic restriction of $\cC^\infty_U$.
 
\Subsection{Functions with moderate growth and rapid decay along \texorpdfstring{$\bM$}{bM}}
Recall that we work in a local setting made precise at the beginning of this section \ref{app:currents}.

\begin{defi}\label{def:rdmod}\mbox{}
\begin{enumerate}
\item
A $C^\infty$ function $f$ on $U^\circ$, \resp on $M^\circ$, has moderate growth along $g^{-1}(0)$, \resp along $\bM$, if, for any relatively compact open subset $W\subset U$, \resp $W\subset M$, and any multi-index $\alpha$, there exists $N\geq0$ and $C>0$ such that~$\partial^\alpha f_{|W^\circ}$ is bounded by $C|g|^{-N}$.
\item
A $C^\infty$ function $f$ on $U$, \resp $M$, has rapid decay along $g^{-1}(0)$, \resp along $\bM$, if all its derivatives vanish at any point of $g^{-1}(0)$, \resp $\bM$. Equivalently, the function $f/g^N$ is $C^\infty$ for every $N$. A $C^\infty$ function $f$ on $U^\circ$, \resp $M^\circ$ is said to have rapid decay along $g^{-1}(0)$, \resp along $\bM$, if its extension by $0$ to $U$, \resp to $M$, is $C^\infty$ with rapid decay.
\end{enumerate}
\end{defi}

The corresponding sheaves are denoted respectively by $\cC_U^{\infty,\rmod}$, $\cC_U^{\infty,\rrd}$ and $\cC_M^{\infty,\rmod}$, $\cC_M^{\infty,\rrd}$. Let~$x_o$ be a point in the intersection of exactly $\ell$ smooth components of $g^{-1}(0)$. Then, in the neighbourhood of $x_o$, $U^\circ$ has $2^\ell$ connected components and $M^\circ$ is one of them. A $C^\infty$ function with moderate growth, \resp rapid decay, on such a neighbourhood is nothing but the family of its restrictions to the connected components, since these restrictions trivially  glue together along~$g^{-1}(0)$.

The corresponding de~Rham complexes on $M$ are denoted by $(\cE_M^{\rmod,\cbbullet},\rd)$ and $(\cE_M^{\rrd,\cbbullet},\rd)$ (we~omit the boundary $\bM$ in the notation). On the other hand, we will also consider the sheaf of $C^\infty$ functions on $M$ with poles along $\partial M$, that we simply denote by $\cC_M^\infty(*)$. It is a subsheaf of $\cC_M^{\infty,\rmod}$.

\begin{prop}[Poincaré lemma]\label{prop:rrdZcinf}
The complexes $(\cE_M^{\rmod,\cbbullet},\rd)$ and $(\cE_M^{\rrd,\cbbullet},\rd)$ have cohomology in degree zero only, and are a resolution of $j_*\CC_{M^\circ}=\CC_M$ and $j_!\CC_{M^\circ}$ respectively.
\end{prop}

\begin{proof}
We prove the statements on global sections. Let $\eta=\sum_{|I|=p}\eta_I\rd x_I\in\cE^{\rrd,p}(M)$ (\resp $\eta\in\cE^{\rmod,p}(M^\circ)$) be such that $\rd\eta=0$ and $p\geq1$. There is an explicit formula (see for example \cite[(1.23)]{Demailly12}) computing $\psi\in\cE^{p-1}(M)$ such that $\rd\psi=\eta$:
\[
\psi(x)=\sum_{\substack{|I|=p\\k\in[1,p]}}\biggl(\int_0^1t^{p-1}\eta_I(tx)\rd t\biggr)(-1)^{k-1}x_{i_k}\rd x_{i_1}\wedge\cdots\wedge\wh{\rd x_{i_k}}\wedge\cdots\wedge\rd x_{i_p}.
\]
This formula shows that $\psi$ belongs to $\cE^{\rrd,p-1}(M)$ (\resp to $\cE^{\rmod,p-1}(M^\circ)$), hence the first assertion. The second assertion in both cases is clear.
\end{proof}

\subsection{Poincaré lemma for currents}\label{subsec:currentsrdmod}
The sheaf of distributions $\Db_M$ on $M$ is defined as the sheaf-theoretic restriction to $M$ of $\Db_U$. As in \cite{deRham73,deRham84}, we define the chain complex of currents $(\Cb_{M,\bbullet},\partial)$. By definition, it is the sheaf-theoretic restriction to $M$ of the complex $\Cb_{U,\bbullet}$ that we consider now. A~current $T$ of dimension $q$ on $U$ can be paired with a test $q$-form $\varphi$ on~$U$ ($C^\infty$ with compact support) and the Stokes formula holds: $\langle\partial T,\varphi\rangle=\langle T,\rd\varphi\rangle$. If $\omega$ is a $C^\infty$ $(m-q)$-form, it defines a $q$\nobreakdash-dimensional current by the formula $\langle T_\omega,\varphi\rangle=\int_U\omega\wedge\varphi$. We can consider a current as a differential form with distributional coefficients, \ie we can write $T = \sum_{\#J=m-q}\theta_J\rd x_J$ with $\theta_J\in\Db(U)$. When considered as a form, the differential $\rd$ on currents extends that on $C^\infty$ forms. The relation between $\rd$ and $\partial$ is given, for a current of dimension $q$ (\ie degree $m-q$) by
\[
\partial T_q=(-1)^{m-q+1}\rd T_q.
\]
For a tensor product taken over $\CC$, like $(\Cb_{M,\bbullet},\partial)\otimes(\cE_M^{\rrd,\cbbullet},\rd)$, the boundary is given by the formula, for a current $T_q\otimes\omega^p$ (of dimension $q-p$):
\begin{equation}\label{eq:tensorcurrents}
\partial(T_q\otimes\omega^p)=(-1)^p(\partial T_q\otimes\omega^p-T_q\otimes\rd\omega^p)=(-1)^p\partial T_q\otimes\omega^p+T_q\otimes\partial\omega^p.
\end{equation}

\begin{prop}[Poincaré lemma for $(\Cb_{M,\bbullet},\partial)$]\label{prop:poincarecurrentsrdZ}
The complex of currents on $M$ satisfies
\[
\cH_q(\Cb_{M,\bbullet},\partial)=\begin{cases}
0&\text{if }q\neq m,\\
\CC_M&\text{if }q=m.
\end{cases}
\]
More precisely, the inclusion $(\cE_M^{m-\cbbullet}\rd)\hto(\Cb_{M,\bbullet},\partial)$ is a quasi-isomorphism.
\end{prop}

\begin{proof}
We prove the result on an open subset $U$ of $\RR^m$, and we obtain the proposition by sheaf\nobreakdash-theoretic restriction to $M\cap U$. We refer to \cite[\S2.D.4]{Demailly12} for the proof of the lemma below, originally in \cite{deRham73,deRham84}. Let \hbox{$T\in\Cb_q(U)$}.

\begin{lemma}\label{lem:poincarecurrentsrdZ}
For any $\epsilon\in(0,1)$, there exist $\CC$-linear morphisms
\[
R_\epsilon:\Cb_q(U)\to\Cb_q(U)\quad\text{and}\quad S_\epsilon:\Cb_q(U)\to\Cb_{q+1}(U)
\]
such that
\begin{enumerate}
\item\label{lem:poincarecurrentsrdZ1}
$R_\epsilon$ takes values in $\cE^{m-q}(U)$,
\item\label{lem:poincarecurrentsrdZ2}
$R_\epsilon(T)-T=\partial S_\epsilon(T)+S_\epsilon(\partial T)$,
\item\label{lem:poincarecurrentsrdZ3}
$R_\epsilon(\partial T)=\partial R_\epsilon(T)$ and $\lim_{\epsilon\to0}R_\epsilon(T)=T$ weakly.
\end{enumerate}
\end{lemma}

In particular, $R_\epsilon$ is a morphism of complexes $(\Cb_\bbullet(U),\partial)\to(\cE^{m-\cbbullet}(U),\rd)$, and \eqref{lem:poincarecurrentsrdZ2} implies that it is a quasi-isomorphism. Since $(\cE^{m-\cbbullet}(U),\rd)$ has homology concentrated in degree $\cbbullet =m$, it follows that the same holds for $(\Cb_\bbullet(U),\partial)$. Then we can sheafify the construction.
\end{proof}

\begin{remark}\label{rem:regularization}
Arguing as in \cite[\S15, Th.\,12]{deRham84},
we can globalize the construction on a manifold with corners.
More precisely, given a covering $(U_i)_{i\in I}$ of $M$ by charts, and $\epsilon_i>0$ small enough ($i\in I$), one can construct morphisms
\[
R_\epsilonb:\Cb_{M,\bbullet}\to\cE_M^{m-\cbbullet}\quad\text{and}\quad S_\epsilonb:\Cb_{M,\bbullet}\to\Cb_{M,\bbullet+1}
\]
such that the conclusion of Lemma \ref{lem:poincarecurrentsrdZ} holds for any $T\in\Gamma(M,\Cb_{M,\bbullet})$, up to replacing $\epsilon$ with $\epsilonb=(\epsilon_i)_{i\in I}$. In particular, $R_\epsilonb:(\Cb_{M,\bbullet},\partial)\to(\cE_M^{m-\cbbullet},\rd)$ is a quasi-isomorphism.
\end{remark}

Let us end by recalling the computation of the Kronecker index made by de~Rham. Let $T_p$ and~$T^\vee_{m-p}$ be two currents of complementary dimension on $U$. If both currents are closed, let us choose a decomposition $T_p=\Theta^{m-p}+\partial S_{p+1}$ and $T^\vee_{m-p}=\Theta^{\vee,p}+\partial S^\vee_{m-p+1}$, where $\Theta^{m-p}$ and $\Theta^{\vee,p}$ are $C^\infty$ closed forms. Then $\Bpairing_\dR(T_p,T^\vee_{m-p})$ is by definition the integral $\Qpairing(\Theta^{m-p},\Theta^{\vee,p})$ of the wedge product and is independent of the choices. Without the closedness assumption, but if one of the supports is compact, the intersection $\Bpairing_\dR(T_p,T^\vee_{m-p})$ is defined by the limit, when it exists
\[
\lim_{\epsilon',\epsilon\to0} \Qpairing(R_{\epsilon'}(T_p),R_\epsilon(T^\vee_{m-p}))=\lim_{\epsilon\to0}\langle T_p,R_\epsilon(T^\vee_{m-p})\rangle.
\]

We now consider the setting of the end of Section \ref{sec:chains}. Let us choose a simplex $\sigma_p$ in $M$ not contained in~$\partial M$ and let $\ov D(\sigma_p)$ denote the dual cell in $M^\circ$, which intersects $\sigma_p$ at its barycenter $\wh\sigma_p$. Let us set $T_p=\int_{\sigma_p}$, $T_{m-p}^\vee=\int_{\ov D(\sigma_p)}$ and $\Bpairing_\dR(\sigma_p,\ov D(\sigma_p))= \Bpairing_\dR(T_p,T^\vee_{m-p})$.

\begin{prop}[{\cite[p.\,85--86]{deRham84}}]\label{prop:Kronecker}
Under the previous assumptions, the intersection $\Bpairing_\dR(\sigma_p,\ov D(\sigma_p))$ exists and is given by the formula
\[
\Bpairing_\dR(\sigma_p,\ov D(\sigma_p))=\pm1,
\]
where $\pm$ is the orientation change between $\sigma_p\times\ov D(\sigma_p)$ and $M^\circ$.
\end{prop}

\begin{proof}
Let us choose coordinates $(x',x'')=(x_1,\dots,x_p,x_{p+1},\dots,x_m)$ so that $\sigma_p$ (with orientation) (\resp $\ov D(\sigma_p)$) is contained in the coordinate plane $(x_1,\dots,x_p)$ (\resp $(x_{p+1},\dots,x_m)$). The formula for $R_\epsilon$ in Lemma \ref{lem:poincarecurrentsrdZ}~is:\vspace*{-3pt}
\begin{align*}
R_\epsilon(T_{m-p}^\vee)&=\theta_{m-p,\epsilon}(x)\rd x_1\wedge\cdots\wedge\rd x_p,\\
\tag*{with}\theta_{m-p,\epsilon}(x)&=\int_{\ov D(\sigma_p)}\hspace*{-2mm}\frac{\chi_\epsilon((x-y)/\psi(y))}{\psi(y)^m}\,\rd y_{p+1}\wedge\cdots\wedge\rd y_m,
\end{align*}
where $\chi:B(0,1)\to\RR_+$ is any $C^\infty$-function with compact support and integral equal to $1$ and, for~$\epsilon>0$,
let $\chi_\epsilon(v)=\epsilon^{-m}\chi(v/\epsilon)$ with support in $B(0,\epsilon)$. Hence, if we assume that $\chi$ is even with respect to the variables $x''$,
\begin{align*}
\Bpairing_\dR(\sigma_p,\ov D(\sigma_p))&=\lim_{\epsilon\to0}\int_{\sigma_p}\rd x_1\wedge\cdots\wedge\rd x_p\cdot \theta_{m-p,\epsilon}(x')\\
&=\lim_{\epsilon\to0}\int_{\sigma_p\times\ov D(\sigma_p)}\hspace*{-4mm}\frac{\chi_\epsilon((x',-x'')/\psi(x''))}{\psi(x'')^m}\,\rd x_1\wedge\cdots\wedge\rd x_m=\pm1
\end{align*}
by the properties of $\chi$.
\end{proof}

\Subsection{Poincaré lemma for current with moderate growth and rapid decay}
Let us first consider the sheaf $\Db^\rmod_U$ on $U$, defined as the subsheaf of $j_*\Db_{U^\circ}$ consisting of distributions extendable to $U$. Its space of sections on an open set $W\subset U$ is the topological dual of the space of $C^\infty$ functions on $W$ with compact support and with rapid decay along $g^{-1}(0)$, endowed with its usual family of semi-norms. It~follows from \cite[Chap.\,VII]{Malgrange66} that $\Db^\rmod_U=\cC_U^\infty(*)\otimes_{\cC_U^\infty}\Db_U$, where we have set \hbox{$\cC_U^\infty(*)=\cC_U^\infty[1/g]$.} We also have an exact sequence\vspace*{-3pt}
\begin{equation}\label{eq:exdist}
0\to\Db_{[g^{-1}(0)]}\to\Db_U\to\Db_U^\rmod\to0,
\end{equation}
where $\Db_{[g^{-1}(0)]}$ denotes the subsheaf of distributions supported on~$g^{-1}(0)$.

The sheaf $\Db^\rmod_M$ is defined as the subsheaf of $j_*\Db_{M^\circ}$ consisting of distributions extendable to~$M$, and its space of section on an open set $V\subset M$ is the topological dual of the space of $C^\infty$ functions on $V$ with compact support and with rapid decay along $\bM$, endowed with its usual family of semi-norms.

Near a point $x_o\in g^{-1}(0)$ where $g^{-1}(0)$ has exactly $\ell$ components, the space of rapid decay functions decomposes as the direct sum of the spaces of rapid decay functions in each local connected component of $U^\circ$, and by duality so does the space of moderate distributions in this neighbourhood. If $x_o\in\bM$, then the space of moderate distribution on the intersection of this neighbourhood with~$M$ is one of the summands.

\begin{defi}[Currents with moderate growth]
We set
\[
\Cb_{U,p}^\rmod = \cE_U^{m-p}\otimes_{\cC_U^\infty}\Db_U^\rmod\quad\text{and}\quad\Cb_{M,p}^\rmod=\cE_M^{m-p}\otimes_{\cC_M^\infty}\Db_M^\rmod.
\]
\end{defi}

In a way similar to currents, a current $T$ of dimension $q$ with moderate growth on $M$ can be paired with a test $q$-form $\varphi$ on $M$ with rapid decay and Stokes formula holds.

\begin{prop}[Poincaré lemma for $(\Cb_{M,\bbullet}^\rmod,\partial)$]\label{prop:Cmod}
The chain complex $(\Cb_{M,\bbullet}^\rmod,\partial)$ satisfies Poincaré lemma as in Proposition \ref{prop:poincarecurrentsrdZ}.
\end{prop}

\begin{proof}
We will first work on $U$ and then restrict to $M$, with the cohomological version $(\Cb_M^{\rmod,\cbbullet},\rd)$, that we wish to prove to be a resolution of its $\cH^0$, that is, $\CC_M$. In other words, we wish to prove that the natural morphism $(\Cb_M^\cbbullet,\rd)\to(\Cb_M^{\rmod,\cbbullet},\rd)$ is a quasi-isomorphism.

\begin{lemma}
The complex $(\Cb_{[g^{-1}(0)]}^\cbbullet,\rd)$ has cohomology in degree one only and, for $x_o$ in the intersection of $\ell$ smooth components of $g^{-1}(0)$, we have $\dim\cH^1(\Cb_{[g^{-1}(0)]}^\cbbullet,\rd)_{x_o}=2^\ell-1$.
\end{lemma}

\begin{proof}
Let $Z\subset U$ denote the $\ell$-th intersection
defined by \hbox{$x_1=\cdots=x_\ell=0$},
so that $\Db_{[Z]}\simeq\Db_{Z}[\partial_{x_1},\dots,\partial_{x_\ell}]$. We write $(\Cb_{[Z]}^\cbbullet,\rd)$ as the simple complex associated to the $\ell$-cube with $i$-th edges all equal~to
\[
\Cb_{Z}^\cbbullet[\partial_{x_1},\dots,\partial_{x_\ell}]\To{\partial_{x_i}}\Cb_{Z}^\cbbullet[\partial_{x_1},\dots,\partial_{x_\ell}],
\]
which is quasi-isomorphic to $\Cb_{Z}^{\cbbullet-\ell}$, so $\Cb_{[Z]}^\cbbullet$ has cohomology in degree $\ell$ only, equal to the constant sheaf~$\CC_{Z}$ (by the Poincaré lemma for $Z$).

Let us set $Z_i=\{x_i=0\}$ and $Z^{(\ell)}=\bigcup_{i=1}^\ell Z_i$. We prove by induction on $\ell$ that the lemma hods for $Z^{(\ell)}$, the case $\ell=1$ begin proved in the first point. According to \cite[Prop.\,VII.1.4]{Malgrange66}, the natural sequence of complexes
\[
0\to(\Cb_{[Z_\ell^{(\ell-1)}]}^\cbbullet,\rd)\to(\Cb_{[Z_\ell]}^\cbbullet,\rd)\oplus(\Cb_{[Z^{(\ell-1)}]}^\cbbullet,\rd)\to(\Cb_{[Z^{(\ell)}]}^\cbbullet,\rd)\to0
\]
is exact. Let us restrict the complexes on $Z=\bigcap_{i=1}^\ell Z_i$. By induction, the middle complex has cohomology in degree one only, of dimension $1+(2^{\ell-1}-1)=2^{\ell-1}$. Also by induction, the left complex has cohomology in degree $2$ only, of dimension $2^{\ell-1}-1$. This completes the proof.
\end{proof}

The exact sequence \eqref{eq:exdist} gives rise to an exact sequence of complexes
\[
0\to\Cb_{[g^{-1}(0)]}^\cbbullet\to\Cb_U^\cbbullet\to\Cb_U^{\rmod,\cbbullet}\to0
\]
and the lemma above reduces it to an exact sequence
\[
0\to\CC_U\to\cH^0(\Cb_U^{\rmod,\cbbullet})\to\cH^1(\Cb_{[g^{-1}(0)]}^\cbbullet)\to\nobreak0.
\]
At a point $x_o\in g^{-1}(0)$ at the intersection of exactly $\ell$ smooth components, $\cH^0(\Cb_U^{\rmod,\cbbullet})_{x_o}$ has thus dimension $2^\ell$. If $U^\circ_{x_o,a}$ are the $2^\ell$ local connected components of~$U^\circ_{x_o}$, then
\[
\cH^0(\Cb_U^{\rmod,\cbbullet})_{x_o}=\coH^0(\Cb^{\rmod,\cbbullet}(U^\circ_{x_o}))=\bigoplus_a\coH^0(\Cb^{\rmod,\cbbullet}(U^\circ_{x_o,a})),
\]
so that each term in the sum is isomorphic to~$\CC$. In particular, for $x_o\in\bM$, considering the component $M^\circ_{x_o}$, the natural composed morphism (after sheaf-theoretically restricting to $M$) $\Cb_{U|M}^\cbbullet\to\Cb_{U|M}^{\rmod,\cbbullet}\to\Cb_M^{\rmod,\cbbullet}$ induces an isomorphism on the $\cH^0$ at each $x_o$, and hence is an isomorphism.
\end{proof}

We define the complex of currents with rapid decay as $(\Cb_{M,\bbullet}^\rmod,\partial)\otimes(\cE_M^{\rrd,\cbbullet},\rd)$, with boundary operator given by \eqref{prop:poincarecurrentsrdZ}. From Propositions \ref{prop:Cmod} and \ref{prop:rrdZcinf}, we thus obtain:

\begin{prop}[Poincaré lemma for $(\Cb_{M,\bbullet}^\rrd,\partial)$]\label{prop:Crd}
The chain complex $(\Cb_{M,\bbullet}^\rrd,\partial)$ has homology in degree $m$ only, which is equal to $j_!\CC_{M^\circ}$.
\end{prop}

\begin{remark}\label{rem:BM}
If $M$ is compact, the homology of the complex $(\Gamma(M,\Cb_{M,\bbullet}^\rrd),\partial)$ is isomorphic to $\coH_\bbullet(M^\circ,\CC)$, while that of $(\Gamma(M,\Cb_{M,\bbullet}^\rmod),\partial)$ is isomorphic to the Borel-Moore homology $\coH^{\scriptscriptstyle\mathrm{BM}}_\bbullet(M^\circ,\CC)$.
\end{remark}

\subsection{Dolbeault lemmas}
We recall here the various Dolbeault lemmas that we have used in the main text. Let $X$ be a complex manifold, let $\Div$ be a normal crossing divisor and let $\varpi:\wt X\to X$ be the real oriented blow-up of the components of $\Div$ (local coordinates on $\wt X$ are local polar coordinates on~$X$ with respect to a coordinate system adapted to $D$). Then $\wt X$ is a $C^\infty$ manifold with corners as in the beginning of this section, so is endowed with the sheaves $\cC^{\infty,\rmod}_{\wt X}$ and $\cC^{\infty,\rrd}_{\wt X}$ as in Definition~\ref{def:rdmod}. Moreover, one checks that the $\ov\partial$ operator defined on $\cC^\infty_X$ can be lifted to these sheaves of functions, and we can define the corresponding Dolbeault complexes $(\cE_{\wt X}^{\rmod,0,\cbbullet},\ov\partial)$ and $(\cE_{\wt X}^{\rrd,0,\cbbullet},\ov\partial)$ whose $\cH^0$ are respectively denoted by $\cA_{\wt X}^\rmod$ and $\cA_{\wt X}^\rrd$.

\begin{lemma}
These Dolbeault complexes are a resolution of their $\cH^0$.
\end{lemma}

\begin{proof}
For the moderate case, one can use the proof of Dolbeault's lemma on bounded domains of~$\CC^n$ with piecewise $C^1$ boundary using the Bochner-Martinelli kernel (\cf\eg\cite[\S III.2]{LaurentC97}). The rapid decay case is treated in \cite{Majima84b} (\cf also \cite[Lem.\,II.1.1.18]{Bibi97} and \cite[Prop.\,4.2.2]{Mochizuki10}).
\end{proof}

One can similarly define the sheaves on $X$ of $C^\infty$ functions with moderate growth or rapid decay along $\Div$, together with their corresponding Dolbeault complexes, which do not enter the frame considered at the beginning of this section. Nevertheless, one has obviously
\[
(\cE_X^{\rmod,0,\cbbullet},\ov\partial)\simeq(\varpi_*\cE_{\wt X}^{\rmod,0,\cbbullet},\ov\partial)\quad\text{and}\quad(\cE_X^{\rrd,0,\cbbullet},\ov\partial)\simeq(\varpi_*\cE_{\wt X}^{\rrd,0,\cbbullet},\ov\partial).
\]
Furthermore, $\varpi_*\cA_{\wt X}^\rmod\simeq\cO_X(*\Div)$.

\begin{lemma}
The Dolbeault complex $(\cE_X^{\rmod,0,\cbbullet},\ov\partial)$ is a resolution of $\cO_X(*D)$, and the Dolbeault complex $(\cE_X^{\rrd,0,\cbbullet},\ov\partial)$ is quasi-isomorphic to the two-term complex \hbox{$\cO_X\to\cO_{\wh\Div}$}.
\end{lemma}

\begin{proof}
The moderate case is treated as in that on $\wt X$, while the rapid decay case is proved more generally in \cite{Bingener78}.
\end{proof}

\section{Remarks on Verdier duality}\label{sec:Verdierduality}
To pass from local results on pairings to global ones, we use compatibility of Verdier duality with proper pushforward. In this section is made precise an ``obvious result'' (Corollary \ref{cor:accperfect}) which we could not find in the literature. We refer to \cite[Chap.\,2\,\&\,3]{K-S90} for standard results of sheaf theory.

We fix a field $\kk$ and we work in the category of sheaves of $\kk$-vector spaces. All topological spaces we consider are assumed to be locally compact, and all maps are assumed to have finite local cohomological dimension. Let $\DD_X$ be the dualizing complex. If~\hbox{$a_X:X\to\mathrm{pt}$} denotes the constant map, one has $\DD_X=a_X^!\kk$.

\subsection{The duality isomorphism}
Let $f:X\to Y$ be a continuous map, let $\cF$ be an object of~$\catD^\rb(\kk_X)$ and~$\cG$ an object of $\catD^\rb(\kk_Y)$. There is a bifunctorial isomorphism \begin{equation}\label{eq:Verdier}
\bR f_*\RcHom(\cF, f^!\cG)\isom\RcHom(\bR f_!\cF,\cG)
\end{equation} in $\catD^\rb(\kk_Y)$ (\cf \cite[Prop.\,3.1.10]{K-S90}). By applying $\bR\Gamma(Y,\cbbullet)$, one obtains
\[
\bR\Hom(\cF,f^!\cG)\isom\bR\Hom(\bR f_!\cF,\cG),
\]
and taking cohomology in degree $0$,
\[
\Hom_{\catD^\rb(X)}(\cF,f^!\cG)\isom\Hom_{\catD^\rb(Y)}(\bR f_!\cF,\cG).
\]

There exists a natural morphism of functors $\bR f_!f^!\mto\id$. Indeed, taking $\cF=f^!\cG$ above, one finds
\[
\Hom_{\catD^\rb(X)}(f^!\cG,f^!\cG)\isom\Hom_{\catD^\rb(Y)}(\bR f_!f^!\cG,\cG)
\]
and the desired morphism is the image of the identity. 

One can thus write the duality isomorphism \eqref{eq:Verdier} as the composition of two bifunctorial morphisms (see \cite[(2.6.25)]{K-S90} for the first one)
\[
\bR f_*\RcHom(\cF, f^!\cG)\to\RcHom(\bR f_!\cF,\bR f_!f^!\cG)\to\RcHom(\bR f_!\cF,\cG).
\]

\subsection{Poincaré-Verdier duality}
We set $\bD\cF=\RcHom(\cF,\DD_X)$. One then has an isomorphism
\[
\bR\Gamma(X,\bD\cF)\isom\Hom(\bR\Gamma_\rc(X,\cF),\kk),
\]
and taking hypercohomology,
\[
\bH^\ell(X,\bD\cF)\isom\bH^{-\ell}_\rc(X,\cF)^\vee.
\]
Via this isomorphism and by using the natural duality pairing, we deduce a perfect pairing
\begin{equation}\label{eq:1}
\bH^{-\ell}_\rc(X,\cF)\otimes\bH^\ell(X,\bD\cF)\to \kk.
\end{equation}

On the other hand, one has a natural morphism $\cF\otimes \bD\cF\to\DD_X$
and, in addition,
according to \cite[(2.6.23)]{K-S90} for $f=a_X$, a morphism
\[
\bR\Gamma_\rc(X,\cF)\otimes\bR\Gamma(X,\bD\cF)\to\bR\Gamma_\rc(X,\cF\otimes \bD\cF).
\]
By composing, one obtains a morphism
\[
\bR\Gamma_\rc(X,\cF)\otimes\bR\Gamma(X,\bD\cF)\to\bR\Gamma_\rc(X,\cF\otimes\bD\cF)
\to\bR\Gamma_\rc(X,\DD_X)=\bR a_{X!}a_X^!\kk\to \kk.
\]
By taking hypercohomology, one obtains for each $\ell\in\ZZ$ a pairing
\begin{equation}\label{eq:2}
\bH^{-\ell}_\rc(X,\cF)\otimes\bH^\ell(X,\bD\cF)\to \kk.
\end{equation}

\begin{prop}\label{prop:accperfect}
The pairings \eqref{eq:1} and \eqref{eq:2} coincide. In particular, \eqref{eq:2} is a perfect pairing.
\end{prop}

For the sake of completeness, we will prove the following lemma at the end of this section. 

\begin{lemma}\label{lem:composes}
Let $\cF$ and $\cG$ be objects of $\catD^\rb(\kk_X)$ and let $f:X\to Y$ be a morphism. The two following composed natural morphisms coincide:
\begin{gather*}
\bR f_!\cF\otimes\bR f_*\RcHom(\cF,\cG)\to\bR f_!\bigl(\cF\otimes\RcHom(\cF,\cG)\bigr)\to\bR f_!\cG,\\
\bR f_!\cF\otimes\bR f_*\RcHom(\cF,\cG)\to\bR f_!\cF\otimes\RcHom(\bR f_!\cF,\bR f_!\cG)\to\bR f_!\cG.
\end{gather*}
\end{lemma}

\begin{proof}[\proofname\ of Proposition \ref{prop:accperfect}]
Taking $f=a_X$ and $\cG=\DD_X$, we have equality of the morphisms
\[
\bR\Gamma_\rc(X,\cF)\otimes\bR\Gamma(X,\bD\cF)\to\bR\Gamma_\rc(X,\cF\otimes\bD\cF)\to\bR\Gamma_\rc(X,\DD_X)
\]
and
\begin{multline*}
\bR\Gamma_\rc(X,\cF)\otimes\bR\Gamma(X,\bD\cF)\\\to\bR\Gamma_\rc(X,\cF)\otimes\Hom(\bR\Gamma_\rc(X,\cF),\bR\Gamma_\rc(X,\DD_X))
\to \bR\Gamma_\rc(X,\DD_X).
\end{multline*}
Besides, we have a commutative diagram, by considering the natural morphism \hbox{$\bR\Gamma_\rc(X,\DD_X) \to \kk$},
\[
\xymatrix{
\bR\Gamma_\rc(X,\cF)\otimes\Hom(\bR\Gamma_\rc(X,\cF),\bR\Gamma_\rc(X,\DD_X))\ar[r]\ar[d]& \bR\Gamma_\rc(X,\DD_X).\ar[d]\\
\bR\Gamma_\rc(X,\cF)\otimes\Hom(\bR\Gamma_\rc(X,\cF),\kk)\ar[r]&\kk
}
\]
Hence, both composed natural morphisms
\[
\bR\Gamma_\rc(X,\cF)\otimes\bR\Gamma(X,\bD\cF)\to\bR\Gamma_\rc(X,\bD\cF\otimes\cF)\to \kk
\]
and
\[
\bR\Gamma_\rc(X,\cF)\otimes\bR\Gamma(X,\bD\cF)\to\bR\Gamma_\rc(X,\cF)\otimes\Hom(\bR\Gamma_\rc(X,\cF),\kk)
\to \kk
\]
coincide, and we deduce the proposition for every $\ell$.
\end{proof}

Let now $\cF,\cG$ in $\catD^\rb(\kk_X)$. We have a natural isomorphism (\cf \cite[(2.6.8)]{K-S90})
\[
\Hom(\cF\otimes\cG,\DD_X)\simeq\Hom(\cF,\bD\cG)\simeq\Hom(\cG,\bD\cF).
\]
Giving a pairing $\varphi:\cF\otimes\cG\to\DD_X$ amounts thus to giving a morphism $\cF\to\bD\cG$, or as well a morphism $\cG\to\bD\cF$. We say that the pairing $\varphi$ is \emph{perfect} if the corresponding morphism $\cF\to\bD\cG$ (or as well $\cG\to\bD\cF$) is an isomorphism.

A pairing $\varphi$ induces in a natural way a pairing
\[
\bR\Gamma_\rc(X,\cG)\otimes\bR\Gamma(X,\cF)\to\bR\Gamma_\rc(X,\cF\otimes\cG)\to\bR\Gamma_\rc(X,\DD_X)\to\kk
\]
(and a similar pairing by permuting the roles of $\cF$ and $\cG$), and thus, for every $\ell\in\ZZ$, pairings
\[
\varphi^\ell:\bH^{-\ell}_\rc(X,\cG)\otimes\bH^\ell(X,\cF)\to\bH^0_\rc(X,\cF\otimes\cG)\to\bH^0_\rc(X,\DD_X)=\kk.
\]

\begin{cor}\label{cor:accperfect}
If $\varphi$ is a perfect pairing, the pairings $\varphi^\ell$ are also perfect.
\end{cor}

\begin{proof}
By construction, the composition of $\varphi^\ell$ and the isomorphism
\[
\bH^\ell(X,\cF)\isom \bH^\ell(X,\bD\cG)
\]
induced by $\varphi$ is the pairing \eqref{eq:2} for $\cG$.
\end{proof}

\begin{proof}[Proof of Lemma \ref{lem:composes}]
We first show that, for sheaves $\cF$ and $\cG$ of $\kk$-vector spaces, the natural morphisms
\begin{gather}\label{eq:mor1}
f_*\cF\otimes f_*\cHom(\cF,\cG)\to f_*\bigl(\cF\otimes\cHom(\cF,\cG)\bigr)\to f_*\cG,\\
f_*\cF\otimes f_*\cHom(\cF,\cG)\to f_*\cF\otimes\cHom( f_*\cF, f_*\cG)\to f_*\cG\label{eq:mor2}
\end{gather}
coincide. Let us first recall the construction of the natural morphism
\[
\cF \otimes \cHom(\cF,\cG) \to \cG.
\]
We denote by $[\cF\otimes\cHom(\cF,\cG)]^\sim$ the presheaf
\[
U\mto \Gamma(U,\cF)\otimes\Gamma(U,\cHom(\cF,\cG))
\]
whose associated sheaf is $\cF\otimes\cHom(\cF,\cG)$ by definition. Recall that
\[
\Gamma(U,\cHom(\cF,\cG))=\Hom(\cF_{|U},\cG_{|U})
\]
(compatible families of morphisms $\Hom(\Gamma(U',\cF),\Gamma(U',\cG))$ for $U'$ open in $U$), so that there is a forgetful morphism
\[
\Gamma(U,\cHom(\cF,\cG))\to\Hom(\Gamma(U,\cF),\Gamma(U,\cG)).
\]
By composition, we obtain a morphism
\[
[\cF\otimes\cHom(\cF,\cG)]^\sim(U)\to\Gamma(U,\cF)\otimes\Hom(\Gamma(U,\cF),\Gamma(U,\cG))\to\Gamma(U,\cG).
\]
Since $\cG$ is a sheaf, it induces a sheaf morphism $\cF\otimes\cHom(\cF,\cG)\to\cG$. Considering only open sets of the form $f^{-1}(V)$ with $V$ open in $Y$ gives rise to the natural morphism $f_*(\cF\otimes\cHom(\cF,\cG))\to f_*\cG$. But the presheaf
\[
V\mto [\cF\otimes\cHom(\cF,\cG)]^\sim(f^{-1}(V))
\]
is identified with the presheaf giving rise to $f_*\cF\otimes f_*\cHom(\cF,\cG)$, so
\[
f_*(\cF\otimes\cHom(\cF,\cG))=f_*\cF\otimes f_*\cHom(\cF,\cG).
\]
We thus have obtained \eqref{eq:mor1}. Finally, the morphism above reads
\[
[f_*\cF\otimes f_*\cHom(\cF,\cG)]^\sim(V)\to\Gamma(V,f_*\cG),
\]
and its sheafification is nothing but \eqref{eq:mor2}.

By suitably restricting to $f$-proper supports, we obtain the non-derived version of the identification of morphisms in the lemma. In order to get the general case, one replaces $\cG$ with an injective resolution $\cG^\bbullet\in\catD^+(\kk_X)$. It follows that, for $\cF,\cG\in\catD^\rb(\kk_X)$, $\RcHom(\cF,\cG)=\cHom(\cF,\cG^\cbbullet)$ is flabby (\cf\cite[Prop.\,2.4.6(vii)]{K-S90}) and, by considering a flabby resolution $\cF^\cbbullet\in\catD^+(\kk_X)$ of $\cF$, we replace \eg $\cF\otimes\RcHom(\cF,\cG)$ with $\cF^\cbbullet\otimes\cHom(\cF,\cG^\cbbullet)$, so that we can apply the result for sheaves (we~use flabbiness of $\cHom(\cF,\cG^\cbbullet)$ since we cannot consider $\cHom(\cF^\cbbullet,\cG^\cbbullet)$, as $\RcHom$ is not defined on $\catD^+(\kk_X)\times\catD^+(\kk_X)$).
\end{proof}

\begin{example}\label{ex:dualizing}
In the case of a differentiable manifold $M$ with corners as in Section \ref{subsec:settings}, the dualizing complex $\DD_M$ can be represented by $j_!\CC_{M^\circ}[m]$. Proposition~\ref{prop:rrdZcinf} implies that the shifted complex of rapid decay forms $(\cE_M^{\rrd,\bbullet},\partial)[m]$ is a representative of the dualizing complex. Perfectness of $\cV^\rrd\otimes\cV^\rmod\to j_!\CC_{M^\circ}$ as stated in Assumption \ref{hyp:hyp} means perfectness, in the above sense, of the shifted morphism $\varphi:\cV^\rrd[m]\otimes\cV^\rmod\to j_!\CC_{M^\circ}[m]\simeq\DD_M$.
\end{example}

\backmatter
\providecommand{\sortnoop}[1]{}\providecommand{\eprint}[1]{\href{http://arxiv.org/abs/#1}{\texttt{arXiv\string:\allowbreak#1}}}
\providecommand{\bysame}{\leavevmode\hbox to3em{\hrulefill}\thinspace}
\providecommand{\MR}{\relax\ifhmode\unskip\space\fi MR }
\providecommand{\MRhref}[2]{%
  \href{http://www.ams.org/mathscinet-getitem?mr=#1}{#2}
}
\providecommand{\href}[2]{#2}

\end{document}